\documentclass[reqno]{amsart}
\usepackage{CJK}
\usepackage{hyperref}
\usepackage{cite}
\usepackage{amsmath}
\usepackage{mathrsfs}
\usepackage{xcolor}
\usepackage{bbm}
\usepackage{pifont}
\usepackage{bbding}
\usepackage{amsmath, esint}
\usepackage{amsfonts}
\usepackage{amssymb}
\usepackage{amsfonts,amssymb,amsmath,indentfirst,cases,amsthm}
\usepackage{epsfig}
\usepackage[numbers,sort&compress]{natbib}
\def\dive{\operatorname{div}}

\makeatletter
\@namedef{subjclassname@2020}{%
	\textup{2020} Mathematics Subject Classification}
\makeatother
\numberwithin{equation}{section}

\newtheorem{theorem}{Theorem}[section]
\newtheorem{lemma}[theorem]{Lemma}
\newtheorem{definition}[theorem]{Definition}

\newtheorem{proposition}[theorem]{Proposition}
\newtheorem{remark}[theorem]{Remark}
\newtheorem{corollary}[theorem]{Corollary}

\allowdisplaybreaks

\newcommand{\B}{B}
\newcommand{\Rd}{\mathbb{R}^d}

\newcommand{\xbar}{\bar x}
\newcommand{\ybar}{\bar y}
\newcommand{\abar}{\bar a}
\newcommand{\cL}{\mathcal{L}}
\newcommand{\cC}{\mathcal{C}}
\newcommand{\cD}{\mathcal{D}}

\newcommand{\Tail}{\operatorname{Tail}}
\newcommand{\abs}[1]{\left|#1\right|}
\newcommand{\norm}[1]{\left\|#1\right\|}
\newcommand{\ip}[2]{\left\langle #1,#2\right\rangle}
\newcommand{\del}{\delta}
\newcommand{\eps}{\varepsilon}
\newcommand{\Jop}[1]{J_{#1}}

\begin{document}
	
\title[\hfil Regularity for nonlocal double phase equations]{Gradient regularity for nonlocal double phase equations}

\author[Y. Fang and C. Zhang  \hfil \hfilneg]{Yuzhou Fang and Chao Zhang$^*$}

\thanks{$^*$Corresponding author.}

\address{Yuzhou Fang \hfill\break School of Mathematics, Harbin Institute of Technology, Harbin 150001, China}
\email{18b912036@hit.edu.cn}

\address{Chao Zhang  \hfill\break School of Mathematics and Institute for Advanced Study in Mathematics, Harbin Institute of Technology, Harbin 150001, China}
\email{czhangmath@hit.edu.cn}


\maketitle

\begin{abstract}
This paper is devoted to investigating the  interior $C^{1, \alpha}$ regularity of viscosity solutions to the nonlocal double phase equations
$$
\int_{\mathbb{R}^d} \left(\frac{|u(x)-u(y)|^{p-2}(u(x)-u(y))}{|x-y|^{d+sp}}+a(x,y)\frac{|u(x)-u(y)|^{q-2}(u(x)-u(y))}{|x-y|^{d+tq}}\right)dy=0,
$$
where $2\le p\le q$, $0<s\le t<1$, and $a(x, y)\ge0$. By assuming the Lipschitz continuity of $a(\cdot)$, we show that the gradient of solution is H\"older continuous, provided the distance of $tq$ and $sp$ is suitably small. As a key ingredient to this conclusion, the Lipschitz property of solutions is also established under weaker assumptions on the modulating coefficient $a(\cdot)$, which is of independent interest. Our results develop a nonlocal counterpart of the gradient regularity theory for classical double phase problems due to Colombo \& Mingione [Arch. Ration. Mech. Anal., 2015] and solve the higher regularity issue raised by De Filippis \& Palatucci [J. Differential Equations, 2019]. The core challenges consist in precisely characterizing the subtle interaction among the pointwise behaviour of the coefficient $a(\cdot)$, the growth exponents and the differentiability orders.
\end{abstract}


\maketitle

\tableofcontents

\section{Introduction}
\label{sec1}

In this paper, we are interested in the interior gradient H\"{o}lder continuity of viscosity solutions to nonlocal double phase equations of the type
\begin{align}
\label{main}
\mathrm{P.V.}\int_{\mathbb{R}^d} \left(\frac{|u(x)-u(y)|^{p-2}(u(x)-u(y))}{|x-y|^{d+sp}}+a(x,y)\frac{|u(x)-u(y)|^{q-2}(u(x)-u(y))}{|x-y|^{d+tq}}\right)\,dy=0
\end{align}
for $x\in\Omega\subset\mathbb{R}^d$, whose leading operator switches abruptly between two distinct fractional elliptic phases depending on whether the modulating coefficient $a$ is zero or not. Here P.V. stands for the Cauchy principal value and
\begin{equation*}
s,t\in (0,1),   \quad    1< p\le q<\infty, \quad a(x,y)\ge0.
\end{equation*}
Such equations, first introduced by De Filippis and Palatucci \cite{DP19}, could be regarded as the nonlocal counterpart of classical double phase problems given by
\begin{equation}
	\label{local}
\dive(|Du|^{p-2}Du+a(x)|Du|^{q-2}Du)=0, \quad 1<p\le q,\ a(x)\ge0.
\end{equation}
This class of equations was originally proposed by Zhikov \cite{Zhi86} as one of the models designed to describe strongly anisotropic materials in the context of homogenisation.  Hereafter, we simply denote  Eq. \eqref{main} by
$$
\mathcal{L}u(x)=0  \quad \text{for } x\in\Omega.
$$
More precise structural conditions for inferring the local $C^{1, \alpha}$ regularity of Eq. \eqref{main} will be presented in the subsequent sections.

\subsection{Overview of related literature}

\mbox{}\par
\medskip

Over the past two decades, the regularity theory for nonlocal problems has emerged as an active research field. For the fractional Laplace equation, Caffarelli and Silvestre \cite{CS07} established the Harnack estimate using an extension argument; see \cite{Sil06} for results on H\"{o}lder continuity and \cite{CCV11,KW} for those concerning the parabolic version. In the nonlinear framework, Di Castro, Kuusi and Palatucci \cite{DKP14, DKP16} extended the De Giorgi-Nash-Moser theory to nonlocal $p$-Laplacian equations
$$
\mathrm{P.V.}\int_{\mathbb{R}^d} |u(x)-u(y)|^{p-2}(u(x)-u(y))K(x,y)\,dy=0,  \ \ K(x,y)\approx|x-y|^{-d-sp},
$$
investigating local behaviors of weak solutions, such as H\"older continuity and Harnack inequalities; see \cite{Coz} for the stability of these results as the differentiability order $s$ approaches to $1$.

On the other hand, the concept of viscosity solutions to fractional $p$-Laplace equations was considered in 2010 by Ishii and Nakamura \cite{IN10}, where the authors prvoed the solvability of solutions via Perron's method. Subsequently, Lindgren \cite{Lin16} verified that viscosity solutions are H\"older continuous. The inner relationship between these two different solutions above was established in \cite{KKL19}. Very recently, Giovagnoli, Jesus and Silvestre \cite{Sil25} achived a major breakthrough in the regularity area for the fractional $p$-Laplacian. Specially,  in the case $p\in [2, \frac{2}{1-s})$, the authors showed the $C^{1, \alpha}$ regularity for viscosity solutions and then extended this result to weak solutions by leveraging the equivalence between the two solution notions, advancing beyond the previously known gradient boundedness result in \cite{BT25}.  For more topics, including potential estimates, the nonlocal Wiener criterion, fractional $(s, p)$-harmonic functions and higher Sobolev regularity among others, the readers can refer to \cite{KMS15, KLL23, KKP16, KKP17, BL17}.

In recent years, significant progress has been made in the regularity theory of nonlocal problems with non-standard growth. For instance, the methods developed in \cite{DKP14,DKP16} were applied to demonstrate the local boundedness, H\"{o}lder regularity and Harnack inequalities for nonlocal $G$-Laplace equations with Orlicz growth in \cite{BKO, FZ23, BKS23}; alternative approaches can be found in \cite{CKW22, CKW23}. In addition, De Filippis and Mingione \cite{DeFM24} established the interior gradient H\"older continuity for mixed local and nonlocal equations with nonuniform growth modeled by
\begin{equation}
\label{1-1-1}
-\Delta_p u+(-\Delta_q)^su=0, \quad 1<p\le q,
\end{equation}
through exploiting a perturbation argument, which benefits from the direct regularity refinement provided by classical $p$-Laplace operator.

Turning to the object of study in this article, De Filippis and Palatucci \cite{DP19} first explored the regularity theory for \eqref{main}, proving the $C^{0, \alpha}$ property of its viscosity solution in the spirit of Krylov-Safonov. In terms of distinct solution notions, Fang and Zhang \cite{FZ23} discussed the H\"{o}lder continuity for weak solutions and showed that weak solutions are viscosity solutions. The reverse implication has been claimed in \cite{Zhang}. Moreover, under suitable assumptions, the H\"older continuity of viscosity solutions was enhanced to Lipschitz in \cite{BS}. It is worth pointing out that the aforementioned papers addressed the case where $s\ge t$, treating the $q$-growth term as a lower-order term to some extent. Conversely, Byun, Ok and Song \cite{BOS22} considered the case $s\le t$, and concluded that solutions to \eqref{main} belong to the class $C^{0,\beta}$ if the coefficient $a$ is bounded and satisfies
\begin{equation}
\label{ads}
\begin{cases}
&|a(x_1,y_1)-a(x_2,y_2)|\le [a]_\alpha(|x_1-x_2|+|y_1-y_2|)^\alpha  \\[2mm] 
&tq\le sp+\alpha
\end{cases}
\end{equation}
with $\alpha\in(0, 1]$ in $\mathbb{R}^d\times\mathbb{R}^d$. Furthermore, weak Harnack inequalities were established in \cite{FKZ26} for both cases $s>t$ and $s\le t$; additionally, the authors in \cite{Ok26} proved the Harnack estimate in the special situation $p=q$ (where only the differentiability orders differ). More results on the nonlocal problem \eqref{main} can be found in \cite{Gia, HP, Par, SM22} and references therein.

In contrast to local double phase problems, the development of regularity theory for the nonlocal version remains far from complete; in particular, no results are yet available concerning the gradient H\"older continuity or Harnack inequality. For the local equation \eqref{local}, Colombo and Mingione, in their pioneering works \cite{CM15a, CM15b}, deduced the $C_{\rm loc}^{1, \beta}$ regularity of weak solutions under the following sharp conditions:
\begin{equation}
\label{1-1-2}
\begin{cases}
&a\in C^{0,\alpha},\ \frac{q}{p}\le 1+\frac{\alpha}{d}\ \ \  \Rightarrow u\in C^{1,\beta} \\[2mm] 
&u\in L^\infty,\ a\in C^{0,\alpha},\ q\le p+\alpha \ \ \ \Rightarrow u\in C^{1,\beta}
\end{cases}
\end{equation}
with $\alpha\in(0, 1]$ and $\beta\in (0, 1)$. Moreover, Fang, R\u{a}dulescu and Zhang \cite{FRZ24} proved the equivalence between weak and viscosity solutions to \eqref{local}, which implies that viscosity solutions are also $C^{1, \alpha}$ regular. Remarkably, De Filippis and Mingione \cite{DeFM23} obtained a striking gradient H\"{o}lder regularity result for minimizers of nonuniformly elliptic functionals, even in cases where an Euler-Lagrange equation may not exist. We refer the reader to \cite{BM20, DeFM25, DeFM25_2, Mar89, HO22} and references therein for further developments and related results on Eq. \eqref{local}.

\subsection{Our main result}

\mbox{}\par
\medskip

The main goal of this paper is to establish interior $C^{1, \alpha}$ regularity for solutions to Eq. \eqref{main}, which resolves an open problem raised by De Filippis and Palatucci \cite{DP19}, thus providing a regularity theory that parallels its local counterpart.


Our approach is influenced by the ideas introduced in \cite{Sil25} in the context of fractional $p$-Laplace equations. However, the techniques there for the fractional $p$-Laplacian cannot be directly applied to \eqref{main} due to the nonstandard ($p, q$)-growth and the mixed differentiability orders. To be precise, this nonuniform ellipticity, combined with such mixed orders, prevents the natural scaling properties inherent to fractional $p$-Laplace equations. Furthermore, addintional efforts are required to deal with the interaction between the structural exponents and the nonnegative coefficient $a(\cdot)$. To deal with this challenge, specific quantitative relationships among these quantities will be presented in the assumptions below. These assumptions are primarily motivated by Lemma \ref{lem3-0}, Lemma \ref{lem3-9}, Proposition \ref{pro7.1}, and Section \ref{sec5}.

Before stating the main contributions of this work, we elaborate on the preconditions on the coefficient $a(\cdot)$, differentiability orders $s,t$ and summability exponents $p, q$. For the measurable coefficient $a:  \mathbb{R}^d\times\mathbb{R}^d\rightarrow\mathbb{R}$, we make the following  assumptions:

\smallskip

\begin{itemize}

\item[($A_1$)] Symmetry: $a(x,y)=a(y,x)$ for all $x,y\in\mathbb{R}^d$;

\smallskip

\item[($A_2$)] Translation invariance: $a(x+z,y+z)=a(x,y)$ for all $x,y,z\in\mathbb{R}^n$;

\smallskip

\item[($A_3$)] Boundedness: $0\leq a(x,y)\leq \|a\|_{\infty}$ for $x,y\in\mathbb{R}^d$;

\smallskip

\item[($A_4$)] Lipschitz continuity:
$$
|a(x_1,y_1)-a(x_2,y_2)|\le [a]_{\rm lip}(|x_1-x_2|+|y_1-y_2|)
$$
 for $x_1,x_2,y_1,y_2\in \mathbb{R}^d$.
\end{itemize}
We denote $\|a\|_{\rm lip}:=\|a\|_{\infty}+[a]_{\rm lip}$. The basic assumptions on $s,t,p,q$ are given as follows:

\smallskip

\begin{itemize}

\item[($H_1$)] $0<s\le t<1$, $2\le p< q<\infty$;

\smallskip

\item[($H_2$)] $p\in \left[2,\frac{1}{1-s}\right]$, $q\in \left[2,\frac{1}{1-t}\right)$;

\smallskip

\item[($H_3$)] $tq\le sp+\min\{1,q-p\}$.

\end{itemize}

\smallskip

We are now in a position to state the first result: the \textit{a priori} $C^{1, \alpha}$ estimate for solutions to \eqref{main}. Here and in what follows, we refer to  $C^1$-regular solutions simply as smooth solutions.

\begin{theorem}
\label{thm1}
Let $u$ be a smooth solution to Eq. \eqref{main} in $B_2$. Suppose that the conditions $(A_1)$--$(A_4)$ and $(H_1)$--$(H_3)$ stated above are satisfied. Then there exist two universal constants $\alpha\in(0,1)$ and $C\ge1$, both depending on $d, s, t, p, q, \|a\|_{\rm lip}$, such that
$$
\|u\|_{C^{1,\alpha}(B_1)}\le C\left(\|u\|_{{\rm lip}(B_2)}+\mathrm{Tail}_{q-1,tq}(u; 2)\right).
$$
\end{theorem}

In Theorem \ref{thm1}, though we begin with the $C^1$ regularity of solutions as a starting point to show the $C^{1,\alpha}$ regularity, the $C^1$ regularity serves merely as a qualitative condition for the proof (indeed, for Lemma \ref{lem3-7} below), which is also stated in \cite{Sil25}. The estimates we derived depend only on the Lipschitz norm of $u$ in the $C^{0,1}$ space. As a matter of fact, the Lipschitz continuity of $u$ can be deduced by the Ishii-Lions method even without requiring the coefficient $a(\cdot)$ to be Lipschitz continuous, and the associated Lipschitz (semi-) norm on $u$ is bounded in terms of the $L^\infty$-norm of $u$ and its tails. Now we give the different weaker conditions on the coefficient $a(\cdot)$ and the exponents to establish Lipschitz property of viscosity solutions to \eqref{main}:

\begin{itemize}
\item[($A'_4$)] H\"older continuity of $a(\cdot)$:
$$
|a(x_1,y_1)-a(x_2,y_2)|\le [a]_\alpha(|x_1-x_2|^\alpha+|y_1-y_2|^\alpha)  \quad\text{with }  \ \alpha\in(0,1]
$$
 for $x_1,x_2,y_1,y_2\in \mathbb{R}^d$.
\end{itemize}
We will denote $\|a\|_\alpha:=\|a\|_{\infty}+[a]_\alpha$. Moreover, the basic distance condition on $s,t,p,q$ is given by
\begin{equation}
\label{dis}
tq\le sp+\alpha.
\end{equation}

As our second result, the Lipschitz continuity of viscosity solutions reads as follows.

\begin{theorem}
\label{thm2}
Let $0<s\le t<1$ and $2\le p\le q<\infty$. Suppose that the conditions on $a(\cdot)$, $(A_1)$--$(A_3)$, $(A'_4)$ and the distance requirement \eqref{dis} hold true. For any viscosity solution $u$ to \eqref{main}, there holds 
$$
u\in C^{0,\gamma}_{\rm loc}(B_2) \ \ \ \text{for any } \ \gamma\le\min\left\{1,\frac{sp}{p-1},\frac{tq}{q-1}\right\}.
$$
Besides, if $\frac{sp}{p-1}=\max\left\{\frac{sp}{p-1},\frac{tq}{q-1}\right\}<1$, we can get the improved H\"older continuity $u\in C^{0,\frac{sp}{p-1}}_{\rm loc}(B_2)$.
\end{theorem}

\begin{corollary}
\label{cor3}
From Theorem \ref{thm2}, one can see that if $p\in\left[2,\frac{1}{1-s}\right]$ and $q\in\left[2,\frac{1}{1-t}\right]$, then $u$ is locally Lipschitz continuous in $B_2$ and
$$
\|u\|_{C^{0,1}(B_1)}\le C\left(\|u\|_{L^\infty(B_2)}, \mathrm{Tail}_{p-1,sp}(u; 2), \mathrm{Tail}_{q-1,tq}(u; 2)\right),
$$
where the positive constant $C$ depends also on the structural conditions $d,s,t,p,q,\|a\|_\alpha$.
\end{corollary}
Now Theorem \ref{thm1} and Corollary \ref{cor3} together imply that the $C^{1,\alpha}$-norm of solutions to \eqref{main} could be controlled by their maximum norm and nonlocal tails, that is,
$$
\|u\|_{C^{1,\alpha}(B_1)}\le C\left(\|u\|_{L^\infty(B_2)}, \mathrm{Tail}_{p-1,sp}(u; 2), \mathrm{Tail}_{q-1,tq}(u; 2)\right).
$$

Although Lipschitz continuity for solutions to Eq. \eqref{main} was obtained by \cite{BS} for the case $tq\le sp$ (i.e., $t\le s$), there are several differences and difficulties in our setting $s\le t$.  To see this, we write Eq. \eqref{main} as
$$
\int_{\mathbb{R}^d}\left(1+a(x,y)\frac{|u(x)-u(y)|^{q-p}}{|x-y|^{tq-sp}}\right)\frac{|u(x)-u(y)|^{p-2}(u(x)-u(y))}{|x-y|^{d+sp}}\,dy=0 \quad \text{with } p\le q.
$$
Here we can see that if $tq\le sp$, then the $q$-growth term may be regarded as a perturbation thanks to
$$
a(x,y)\frac{|u(x)-u(y)|^{q-p}}{|x-y|^{tq-sp}}\approx1  \quad\text{for } \ (x,y)\in \Omega\times\Omega,
$$
when $a(\cdot)$ and $u$ are bounded. Nonetheless, our framework for $s\le t$ is more delicate than the former case, since $a(x,y)\frac{|u(x)-u(y)|^{q-p}}{|x-y|^{tq-sp}}$ can be unbounded so that the $q$-growth term constitutes the leading contribution (not a perturbation any more) if $a(\cdot)$ is not zero. This feature exactly matches that of local double phase equations.

\begin{remark}
A few observations regarding the conditions on $a(\cdot)$ are in order. The symmetry enables Eq. \eqref{main} to possess a variational structure for the De Giorgi iteration in Section \ref{sec3}. If $a(\cdot)$ is not translation-invariant, in the processes of linearizing Eq. \eqref{main} (Lemma \ref{lem2-1}), we require $q>1/(1-t)$ to ensure the integral
$$
\int_{B_R(x)}\frac{dy}{|x-y|^{d+q(t-1)+1}}<\infty.
$$
This, together with $p\in \left[2,1/(1-s)\right)$, is not ideal when $t$ is close to 1 and $s$ is away from 1. Moreover, it is known from \cite{FZ23} that equivalence of weak and viscosity solutions to Eq. \eqref{main} for $s\ge t$ does need translation invariance of $a(\cdot)$.  When we linearize Eq. \eqref{main}, we will differentiate the coefficient $a(\cdot)$; Hence, the Lipschitz continuity in $(A_4)$ is imposed instead of H\"older continuity like $(A'_4)$.
\end{remark}

\begin{remark}
The assumption ($H_3$) is composed of $tq\le sp+1$ and $tq\le sp+q-p$. The former corresponds to the condition $tq\le sp+\alpha$ in Theorem \ref{thm2}. The latter is indeed $p(1-s)\le q(1-t)$, which ensures that the coefficient $a(\cdot)$ does not blow up after (infinite) scaling. 
Intuitively, if we let $s,t\rightarrow1$ in the condition $tq\le sp+\min\{1,q-p\}$ from ($H_3$), it reduces to $q\le p+1$, where the number $1$ originates from the $C^{0,1}$ assumption on $a(\cdot)$. This condition corresponds to the condition $\eqref{1-1-2}_2$ for local double phase problems.
\end{remark}

We will establish the gradient H\"older continuity directly from boundedness of solutions $u$ to Eq. \eqref{main}, through bridging the gap between the $C^{0,1}$ and $C^{1}$ regularity of $u$ in our forthcoming manuscript \cite{FZ26}. Therein, we overcome the $C^1$ regularity assumption on $u$ in Lemma \ref{lem3-7} by means of an approximation argument analogous to \cite[Section 9]{Sil25}, and we also need prove the equivalence of weak and viscosity solutions to \eqref{main} for the range $s\le t$. This process is lengthy, and including it in the present paper would make the key points less prominent. Thus we organize it into a separate paper.

This paper is organized as follows. In Section \ref{sec2} we introduce some basic notations, notions and useful technical lemmas. Section \ref{sec3} is dedicated to establishing the improvement of gradient norms of solutions, and  we subsequently study the gradient H\"older continuity in Section \ref{sec4}. In Section \ref{sec5}, we complete the proof of Lemma \ref{lem4-3}. In the last two portion, we apply the Ishii-Lions methods to verify the Lipschitz continuity of viscosity solutions of \eqref{main} in Sections \ref{sec6} and \ref{sec7}.

\section{Preliminaries}
\label{sec2}

In this section, we give some concepts, notations, and auxiliary results. Throughout this manuscript, $C$ or $c$ denotes a positive constant that may vary from line to line. Relevant dependencies on parameters are emphasized by using parentheses; that is, $C(p, q, d)$ means $C$ depends on $p,q,d$. A universal constant means that it depends at most on the structural parameters $d, s, t, p, q$ and $\|a\|_{\rm lip}$ (or $\|a\|_{\infty},[a]_{\rm lip}$). Let $B_r(x_0):=\{x\in\mathbb{R}^d| |x-x_0|<r\}$ represent a ball with radius $r>0$ and center $x_0$. If not important or clear from the context, we will omit the center of ball as $B_r=B_r(x_0)$.

We now give the notion of viscosity solutions to \eqref{main} in $\Omega\subset\mathbb{R}^d$ for $2\le p\le q$.

\begin{definition}
\label{def1}
A lower semicontinuous function $u\in L_{sp}^{p-1}(\mathbb{R}^d)\cap L_{tq}^{q-1}(\mathbb{R}^d)$ is called a viscosity supersolution to Eq. \eqref{main} in $\Omega$, if whenever there is $\psi\in C^2(B_r(x_0))$ for some $B_r(x_0)\subset\Omega$ such that $\psi(x_0)=u(x_0)$ and $\psi(x)\leq u(x)$ in $B_r(x_0)$, then one has
$$
\mathcal{L}\psi_r(x_0)\geq 0,
$$
where
\begin{equation*}
\psi_r(x)=\begin{cases}\psi(x),  & \text{\textmd{for }} x\in B_r(x_0),\\[2mm]
u(x), & \text{\textmd{for }} x\in \mathbb{R}^d\setminus B_r(x_0).
\end{cases}
\end{equation*}
If $-u$ is a viscosity supersolution, we call $u$ a viscosity subsolution. A function $u$ is a viscosity solution if and only if it is super- and subsolution. 
\end{definition}

For $s\in(0,1)$, the fractional Sobolev space $H^{s}(\Omega)$ represents $W^{s,2}(\Omega)$ equipped with the norm $\|u\|_{H^{s}(\Omega)}=\|u\|_{L^2(\Omega)}+[u]_{H^{s}(\Omega)}$. Taking the nonlocal characteristic of Eq. \eqref{main} into account, we introduce the tail space given as
$$
L^{k}_{\gamma}(\mathbb{R}^d):=\left\{u:\mathbb{R}^d\rightarrow\mathbb{R}\Bigg|\int_{\mathbb{R}^d}\frac{|u(x)|^{k}}{(1+|x|)^{d+\gamma}}\,dx<\infty\right\}.
$$
for $k\ge1,\gamma>0$ and the corresponding tail function
$$
\mathrm{Tail}_{k,\gamma}(u;x_0,R):=\left(\int_{B^c_R(x_0)}\frac{|u(y)|^k}{|x_0-y|^{d+\gamma}}\,dy\right)^\frac{1}{k}.
$$
We can find the tail here is finite according to the definition of tail space. For $x_0=0$, let plainly $\mathrm{Tail}_{k,\gamma}(u;R)=\mathrm{Tail}_{k,\gamma}(u;0,R)$. Particularly, the quantities such as $\mathrm{Tail}_{p-2,sp}(u;R)$, $\mathrm{Tail}_{p-1,sp+1}(u;R)$, $\mathrm{Tail}_{q-1,tq+1}(u;R)$ will naturally occur in our forthcoming procedures. The properties of them shall be considered in Lemma \ref{lem2-0}. In this paper, for simplicity, we denote
$$
J_k(t)=|t|^{k-2}t   \quad \text{for } t\in\mathbb{R}.
$$

We now present the following basic estimates for the function $J_k(t)$ and the tail, which will be used repeatedly throughout the subsequent analysis.

\begin{lemma}
\label{lem2-0-0}
Let $k>1$ and $t,\tau\in \mathbb{R}$. Then
$$
J_k(t)-J_k(\tau)=(k-1)\int^1_0|\tau+r(t-\tau)|^{k-2}(t-\tau)\,dr
$$
and
$$
C^{-1}_k(|\tau|+|t|)^{k-2}\le\int^1_0|\tau+rt|^{k-2}\,dr\le C_k(|\tau|+|t|)^{k-2}
$$
with $C_k>0$ depending only upon $k$.
\end{lemma}

\begin{lemma}
\label{lem2-0}
Let $u\in L^\infty(B_{2R})\cap L^{q-1}_{tq}(\mathbb{R}^d)$ and $q>1,R>0$. The inequalities below hold for a universal constant $C>0$:
\begin{itemize}

\smallskip

\item[(1)] $\mathrm{Tail}_{q-1,tq+1}(u;R)\le R^{\frac{-1}{q-1}}\mathrm{Tail}_{q-1,tq}(u;R)$;

\smallskip

\item[(2)] If $tq\le sp+1$, then
$$
\mathrm{Tail}_{p-1,sp+1}(u;R)^{p-1}\le CR^{-\left(sp+1-tq\frac{p-1}{q-1}\right)}\mathrm{Tail}_{q-1,tq}(u;R)^{p-1};
$$

\smallskip

\item[(3)] For $p\in\left[2,\frac{2}{1-s}\right)$,
$$
\mathrm{Tail}_{p-2,sp}(u;R)^{p-2}\le CR^\frac{p(1-s)-2}{p-1}\mathrm{Tail}_{p-1,sp+1}(u;R)^{p-2}.
$$
\end{itemize}
\end{lemma}

\begin{proof}
Estimate (1) is immediate, and (3) follows directly from H\"{o}lder inequality. For (2), in the case $p<q$, applying H\"{o}lder inequality and $tq\le sp+1$, we have
\begin{align*}
\int_{B^c_R}\frac{|u(y)|^{p-1}}{|y|^{d+sp+1}}dy&=\int_{B^c_R}\frac{|u(y)|^{p-1}}{|y|^{(d+tq)\frac{p-1}{q-1}+(d+tq)\frac{q-p}{q-1}+sp+1-tq}}dy\\
&\le\left(\int_{B^c_R}\frac{|u(y)|^{q-1}}{|y|^{d+tq}}dy\right)^\frac{p-1}{q-1}\left(\int_{B^c_R}\frac{dy}{|y|^{d+tq+(sp+1-tq)\frac{q-1}{q-p}}}\right)^\frac{q-p}{q-1}\\
&\le CR^{-\left(tq+(sp+1-tq)\frac{q-1}{q-p}\right)\frac{q-p}{q-1}}\left(\int_{B^c_R}\frac{|u(y)|^{q-1}}{|y|^{d+tq}}dy\right)^\frac{p-1}{q-1},
\end{align*}
with $C>0$ depending on $d,s,t,p,q$. Here we would like to mention that in the second line, we just require $tq+(sp+1-tq)\frac{q-1}{q-p}>0$ to ensure the convergence of the integral, which is equivalent to  $tq<\frac{q-1}{p-1}(sp+1)$. Clearly, the condition $tq\le sp+1$ meets this requirement. In the case $p=q$, via $tp\le sp+1$, we straight derive
$$
\int_{B^c_R}\frac{|u(y)|^{p-1}}{|y|^{d+sp+1}}dy\le R^{tp-sp-1}\int_{B^c_R}\frac{|u(y)|^{p-1}}{|y|^{d+tp}}dy,
$$
as desired.
\end{proof}

The following geometric convergence lemma shall play a crucial role in the De Giorgi iteration scheme in Section \ref{sec3}.

\begin{lemma} [\cite{Giu03}\label{geo}]

Let $\{Y_n\}$ be a sequence of positive real numbers fulfilling the recursive inequalities
$$
Y_{n+1}\le Cb^nY_n^{1+\gamma}
$$
with $C,b>1$ and $\gamma>0$. If $Y_0\le C^{\frac{-1}{\gamma}}b^{\frac{-1}{\gamma^2}}$, then $Y_n\rightarrow0$ as $n\rightarrow\infty$.
\end{lemma}

We now write $\mathcal{L}_u^a$ for the linearization of $\mathcal{L}$ in Eq. \eqref{main} at $u$. More precisely, define the kernel $K^a_u$ as
\begin{equation}
\label{ku}
K^a_u(x,y)=(p-1)\frac{|u(x)-u(y)|^{p-2}}{|x-y|^{d+sp}}+(q-1)a(x,y)\frac{|u(x)-u(y)|^{q-2}}{|x-y|^{d+tq}},
\end{equation}
and then $\mathcal{L}_u^a$ is given by the formula below
$$
\mathcal{L}_u^av(x):=\int_{\mathbb{R}^d}(v(x)-v(y))K^a_u(x,y)\,dy.
$$

The linearized operator $\mathcal{L}^a_u$ will appear only in Section \ref{sec3} and Proposition \ref{pro7.1}, in the form of quadratic expressions of the type $\int_B\mathcal{L}^a_uv(x)\varphi(x)\,dx$.
This integral is well-defined under rather mild regularity assumptions on $u, v,\varphi$, which will be clarified in Lemma \ref{lem2-2} below. In fact, as mentioned in Introduction, we will prove Theorem \ref{thm1} by {\itshape a priori} assuming that the solution $u$ is $C^1$-regular. This assures all the expressions make sense classically.

\begin{lemma}
\label{lem2-2}
Let $p\in \big[2,\frac{2}{1-s}\big)$, $tq\le sp+q-p$ and let the assumptions $(A_1), (A_3)$ of $a(\cdot)$ be in force. Suppose that $u\in L^{p-1}_{sp}(\mathbb{R}^d)\cap L^{q-1}_{tq}(\mathbb{R}^d)$ is Lipschitz in $B_1$ and $v\in L^{p-1}_{sp}(\mathbb{R}^d)\cap L^{q-1}_{tq}(\mathbb{R}^d)$ is in $H^\alpha(B_1)$ for some $\alpha\in(\max\{0,1-p(1-s)\}, 1)$. Then $\mathcal{L}^a_uv(x)$ is in $H^{-\beta}(B_r):=(H_0^\beta(B_r))^*$ for any $r<1$ with $\beta=2-p(1-s)-\alpha$.
\end{lemma}

\begin{proof}
Let $\varphi\in H^{\beta}_0(B_r)$. Employing the Lipschitz property of $u$ and the hypotheses $(A_1), (A_3)$ of $a(\cdot)$, we can treat the bilinear form
\begin{align*}
&\quad\left|\int_{B_1}\mathcal{L}^a_uv(x)\varphi(x)\,dx\right|\\
&=\Bigg|\frac{1}{2}\int_{B_1}\int_{B_1}(v(x)-v(y))(\varphi(x)-\varphi(y))K^a_u(x,y)\,dxdy\\
&\quad+\int_{B_1}\int_{B^c_1}(v(x)-v(y))\varphi(x)K^a_u(x,y)\,dydx\Bigg|\\
&\le C_1\int_{B_1}\int_{B_1}\left(\frac{1}{|x-y|^{d+sp-p+2}}+\frac{1}{|x-y|^{d+tq-q+2}}\right)|v(x)-v(y)||\varphi(x)-\varphi(y)|\,dxdy\\
&\quad+C_2\int_{B_1}\int_{B^c_1}\left(\frac{(|u(x)|+|u(y)|)^{p-2}}{|x-y|^{d+sp}}+\frac{(|u(x)|+|u(y)|)^{q-2}}{|x-y|^{d+tq}}\right)
(|v(x)|+|v(y)|)|\varphi(x)|\,dydx\\
&=:C_1I_1+C_2I_2,
\end{align*}
where $C_2>0$ depends on $p,q,\|a\|_\infty$, and $C_1>0$ depends additionally on $[u]_{{\rm lip}(B_1)}$. Let us first consider $I_1$. In view of $tq\le sp+q-p$,
$$
\frac{1}{|x-y|^{d+tq-q+2}}=\frac{|x-y|^{sp-p+q-tq}}{|x-y|^{d+sp-p+2}}\le \frac{2^{sp-p+q-tq}}{|x-y|^{d+sp-p+2}}.
$$
Then by H\"{o}lder inequality,
\begin{align*}
I_1&\le C\int_{B_1}\int_{B_1}\frac{|v(x)-v(y)||\varphi(x)-\varphi(y)|}{|x-y|^{d+sp-p+2}}\,dxdy\\
&=C\int_{B_1}\int_{B_1}\frac{|\varphi(x)-\varphi(y)|}{|x-y|^{\frac{d}{2}+sp-p+2-\alpha}}\frac{|v(x)-v(y)|}{|x-y|^{\frac{d}{2}+\alpha}}\,dxdy\\
&\le C[v]_{H^\alpha(B_1)}[\varphi]_{H^{\beta}(B_1)}.
\end{align*}
For $I_2$, we note $(1-r)|y|\le |x-y|$ with $x\in B_r$ and $y\in B^c_1$, and thus by H\"{o}lder inequality again, we get
\begin{align*}
I_2&\le \frac{1}{(1-r)^{d+tq}}\int_{B_1}\int_{B^c_1}\bigg(\frac{\|u\|_{L^\infty(B_1)}^{p-2}+|u(y)|^{p-2}}{|y|^{d+sp}}\left(|v(x)|+|v(y)|\right)
\left|\varphi(x)\right|\\
&\qquad\qquad\qquad\qquad\qquad+\frac{\|u\|_{L^\infty(B_1)}^{q-2}+|u(y)|^{q-2}}{|y|^{d+tq}}\left(|v(x)|+|v(y)|\right)\left|\varphi(x)\right|\bigg)\,dydx\\
&\le \frac{C}{(1-r)^{d+tq}}\|\varphi\|_{L^2(B_1)},
\end{align*}
where $C>0$ depends on $d,s,t,p,q,\|u\|_{L^\infty(B_1)},\|v\|_{L^2(B_1)}$ and $\mathrm{Tail}_{p-1,sp}(u;1)$, $\mathrm{Tail}_{p-1,sp}(v;1)$, $\mathrm{Tail}_{q-1,tq}(u;1)$, $\mathrm{Tail}_{q-1,tq}(v;1)$.

As a consequence, it yields that
$$
\left|\int_{B_1}\mathcal{L}^a_uv(x)\varphi(x)\,dx\right|\le C\|\varphi\|_{H^{\beta}(B_1)}
$$
for all $\varphi\in H^{\beta}_0(B_r)$, which means that $\mathcal{L}^a_uv$ is well-defined in the space $H^{-\beta}(B_r)$.
\end{proof}

Let $e\in \mathbb{S}^{d-1}$ be any unit vector. We are going to localize the directional derivative $e\cdot \nabla u$, and compute the linearization of Eq. \eqref{main} applied to this. Let $\eta:\mathbb{R}^d\rightarrow[0,1]$ be a smooth radially-symmetric cut-off function satisfying
$$
\eta(x)=1 \quad\text{for } |x|\le\frac{3}{2} \quad \text{and} \quad \eta(x)=0 \quad \text{for } |x|\ge\frac{7}{4}.
$$
For an $e\in \mathbb{S}^{d-1}$ and a $R>0$, introduce a function
$$
v_e(x)=\eta\left(\frac{x}{R}\right)(e\cdot\nabla u).
$$

\begin{proposition}[Linearization of \eqref{main}]
\label{lem2-1}
Assume that the coefficient $a(\cdot)$ satisfies $(A_2)$--$(A_4)$, and $2\le p\le q$, $tq\le sp+1$. Let $u\in L^{q-1}_{tq}(\mathbb{R}^d)$ be a smooth solution to \eqref{main} in $B_{2R}$. Then the localized function $v_e$ solves the following equation
\begin{align*}
|\mathcal{L}_u^av_e(x)|\le C\Bigg(&\frac{\|u\|^{p-1}_{L^\infty(B_{R})}}{R^{sp+1}}+\frac{\|u\|^{q-1}_{L^\infty(B_{R})}}{R^{tq}}
+\frac{\|u\|^{q-1}_{L^\infty(B_{R})}}{R^{tq+1}}\\
&+(1+R^{-1})\mathrm{Tail}_{q-1,tq}(u;R)^{q-1}+R^{-\left(sp+1-\frac{p-1}{q-1}tq\right)}\mathrm{Tail}_{q-1,tq}(u;R)^{p-1}\Bigg)
\end{align*}
for $x\in B_R$, where the universal constant $C$ depends on $d,s,t,p,q$ and $\|a\|_{\rm lip}$.
\end{proposition}

\begin{proof}
We rewrite $\mathcal{L}u(x)$ as
\begin{align*}
\mathcal{L}u(x)=&\int_{\mathbb{R}^d} \left(\frac{|u(x)-u(y)|^{p-2}}{|x-y|^{d+sp}}+a(x,y)\frac{|u(x)-u(y)|^{q-2}}{|x-y|^{d+tq}}\right)(u(x)-u(y))\eta_R(y)\,dy\\
&+\int_{\mathbb{R}^d} \left(\frac{|u(x)-u(y)|^{p-2}}{|x-y|^{d+sp}}+a(x,y)\frac{|u(x)-u(y)|^{q-2}}{|x-y|^{d+tq}}\right)(u(x)-u(y))(1-\eta_R(y))\,dy\\
=:&\, F_{1p}(x)+F_{1q}(x)+F_{2p}(x)+F_{2q}(x)
\end{align*}
with $\eta_R(x)=\eta\left(\frac{x}{R}\right)$. We know from \eqref{main} that $F_{1p}(x)+F_{1q}(x)+F_{2p}(x)+F_{2q}(x)=0$, and then differentiate this equation to get
\begin{equation}
\label{2-1-1}
e\cdot\nabla F_{1p}(x)+e\cdot\nabla F_{1q}(x)+e\cdot\nabla F_{2p}(x)+e\cdot\nabla F_{2q}(x)=0.
\end{equation}
Note the translation invariance of $a(\cdot)$, i.e., $a(x,x+h)=a(0,h)$ for $x,h\in\mathbb{R}^d$. Now for $x\in B_R$ we calculate
\begin{align*}
e\cdot\nabla F_{1q}(x)=&\, e\cdot\nabla\int_{\mathbb{R}^d}a(0,h)\frac{|u(x)-u(x+h)|^{q-2}}{|h|^{d+tq}}(u(x)-u(x+h))\eta_R(x+h)\,dh\\
=&\int_{\mathbb{R}^d}(q-1)a(0,h)\frac{|u(x)-u(x+h)|^{q-2}}{|h|^{d+tq}}(e\cdot\nabla u(x)-e\cdot\nabla u(x+h))\eta_R(x+h)\,dh\\
&+\int_{\mathbb{R}^d}a(0,h)\frac{|u(x)-u(x+h)|^{q-2}}{|h|^{d+tq}}(u(x)-u(x+h))e\cdot\nabla\eta_R(x+h)\,dh,
\end{align*}
and similarly, we can get an equality of $e\cdot\nabla F_{1p}(x)$ with $a(\cdot)\equiv1$ and $q:=p$ here. As a result, it yields that
\begin{align}
\label{2-1-2}
e\cdot\nabla F_{1p}(x)+e\cdot\nabla F_{1q}(x)=&\int_{\mathbb{R}^d}K^a_u(x,x+h)(e\cdot\nabla u(x)-e\cdot\nabla u(x+h))\eta_R(x+h)\,dh \nonumber\\
&+\int_{\mathbb{R}^d}\widehat{h}(x,x+h)e\cdot\nabla\eta_R(x+h)\,dh \nonumber\\
=&\int_{\mathbb{R}^d}K^a_u(x,y)(e\cdot\nabla u(x)-e\cdot\nabla u(y))\eta_R(y)\,dy  \nonumber\\
&+\int_{\mathbb{R}^d}\widehat{h}(x,y)e\cdot\nabla\eta_R(y)\,dy  \nonumber\\
=&\, \mathcal{L}_u^av_e(x)+\int_{\mathbb{R}^d}K^a_u(x,y)e\cdot\nabla u(x)(\eta_R(y)-\eta_R(x))\,dy   \nonumber\\
&+\int_{\mathbb{R}^d}\widehat{h}(x,y)e\cdot\nabla\eta_R(y)\,dy,
\end{align}
where the notation
$$
\widehat{h}(x,y):=\frac{|u(x)-u(y)|^{p-2}(u(x)-u(y))}{|x-y|^{d+sp}}+a(x,y)\frac{|u(x)-u(y)|^{q-2}(u(x)-u(y))}{|x-y|^{d+tq}}.
$$

On the other hand, we compute
\begin{align*}
e\cdot\nabla F_{2q}(x)=&\int_{\mathbb{R}^d}(q-1)a(x,y)\frac{|u(x)-u(y)|^{q-2}}{|x-y|^{d+tq}}e\cdot\nabla u(x)(1-\eta_R(y))\,dy\\
&+\int_{\mathbb{R}^d}e\cdot \nabla_xa(x,y)\frac{|u(x)-u(y)|^{q-2}(u(x)-u(y))}{|x-y|^{d+tq}}(1-\eta_R(y))\,dy\\
&+\int_{\mathbb{R}^d}a(x,y)|u(x)-u(y)|^{q-2}(u(x)-u(y))e\cdot\nabla_x\left(\frac{1-\eta_R(y)}{|x-y|^{d+tq}}\right)\,dy
\end{align*}
and an analogous equality of $e\cdot\nabla F_{2p}(x)$ also is obtained. Thus, it holds that
\begin{align}
\label{2-1-3}
&\quad e\cdot\nabla F_{2p}(x)+e\cdot\nabla F_{2q}(x)  \nonumber\\
=&\int_{\mathbb{R}^d}K^a_u(x,y)e\cdot\nabla u(x)(1-\eta_R(y))\,dy \nonumber\\
&+\int_{\mathbb{R}^d}e\cdot \nabla_xa(x,y)\frac{|u(x)-u(y)|^{q-2}(u(x)-u(y))}{|x-y|^{d+tq}}(1-\eta_R(y))\,dy \nonumber\\
&+\int_{\mathbb{R}^d}|u(x)-u(y)|^{p-2}(u(x)-u(y))e\cdot\nabla_x\left(\frac{1-\eta_R(y)}{|x-y|^{d+sp}}\right)\,dy  \nonumber\\
&+\int_{\mathbb{R}^d}a(x,y)|u(x)-u(y)|^{q-2}(u(x)-u(y))e\cdot\nabla_x\left(\frac{1-\eta_R(y)}{|x-y|^{d+tq}}\right)\,dy.
\end{align}
Combining inequalities \eqref{2-1-2}, \eqref{2-1-3} with \eqref{2-1-1} and noticing $\eta_R(x)=1$ for $x\in B_R$, we arrive at
\begin{align*}
\mathcal{L}_u^av_e(x)=&\int_{\mathbb{R}^d}|u(x)-u(y)|^{p-2}(u(x)-u(y))e\cdot\nabla_y\left(\frac{1-\eta_R(y)}{|x-y|^{d+sp}}\right)\,dy \\
&+\int_{\mathbb{R}^d}a(x,y)|u(x)-u(y)|^{q-2}(u(x)-u(y))e\cdot\nabla_y\left(\frac{1-\eta_R(y)}{|x-y|^{d+tq}}\right)\,dy\\
&-\int_{\mathbb{R}^d}e\cdot \nabla_xa(x,y)\frac{|u(x)-u(y)|^{q-2}(u(x)-u(y))}{|x-y|^{d+tq}}(1-\eta_R(y))\,dy\\
=:&\ I_1+I_2-I_3.
\end{align*}

In the sequel, we first treat the integral $I_3$. Recall the Lipschitz continuity of $a(\cdot)$, we get
$$
|e\cdot \nabla_xa(x,y)|\le [a]_{\rm lip}.
$$
For $x\in B_R,y\in B^c_{3R/2}$, we have
\begin{equation}
\label{2-1-4}
|y|\le |y-x|\left(1+\frac{|x|}{|x-y|}\right)\le |y-x|\left(1+\frac{R}{3R/2-R}\right)=3|y-x|.
\end{equation}
Thereby, via the fact $\eta_R(x)=1$ for $|x|\le \frac{3R}{2}$,
\begin{align*}
|I_3|&\le[a]_{\rm lip}\int_{\mathbb{R}^d}\frac{|u(x)-u(y)|^{q-1}}{|x-y|^{d+tq}}(1-\eta_R(y))\,dy \\
&\le[a]_{\rm lip}\int_{\mathbb{R}^d\setminus B_{\frac{3R}{2}}}\frac{|u(x)-u(y)|^{q-1}}{|x-y|^{d+tq}}\,dy \\
&\le C\frac{\|u\|^{q-1}_{L^\infty(B_{R})}}{R^{tq}}+C\int_{\mathbb{R}^d\setminus B_{\frac{3R}{2}}}\frac{|u(y)|^{q-1}}{|x-y|^{d+tq}}\,dy \\
&\le C\frac{\|u\|^{q-1}_{L^\infty(B_{R})}}{R^{tq}}+C\int_{\mathbb{R}^d\setminus B_{\frac{3R}{2}}}\frac{|u(y)|^{q-1}}{|y|^{d+tq}}\,dy.
\end{align*}
Next, we evaluate $I_2$ as follows,
\begin{align*}
|I_2|&\le \int_{B^c_\frac{3R}{2}}a(x,y)|u(x)-u(y)|^{q-1}\left|e\cdot\nabla_y\left(\frac{1-\eta_R(y)}{|x-y|^{d+tq}}\right)\right|\,dy \\
&\le C\int_{B^c_\frac{3R}{2}}\frac{|u(x)-u(y)|^{q-1}}{|x-y|^{d+tq+1}}\,dy+C\int_{B_\frac{7R}{4}\setminus B_\frac{3R}{2}}\frac{|u(x)-u(y)|^{q-1}}{|x-y|^{d+tq}}|\nabla \eta_R(y)|\,dy \\
&\le C\frac{\|u\|^{q-1}_{L^\infty(B_{R})}}{R^{tq+1}}+C\int_{B^c_\frac{3R}{2}}\frac{|u(y)|^{q-1}}{|x-y|^{d+tq+1}}\,dy+\frac{C}{R}\int_{B_\frac{7R}{4}\setminus B_\frac{3R}{2}}\frac{|u(x)-u(y)|^{q-1}}{|x-y|^{d+tq}}\,dy \\
&\le C\frac{\|u\|^{q-1}_{L^\infty(B_{R})}}{R^{tq+1}}+\frac{C}{R}\int_{B_\frac{7R}{4}\setminus B_\frac{3R}{2}}\frac{|u(y)|^{q-1}}{|y|^{d+tq}}\,dy+C\int_{B^c_\frac{3R}{2}}\frac{|u(y)|^{q-1}}{|y|^{d+tq+1}}\,dy \\
&\le C\frac{\|u\|^{q-1}_{L^\infty(B_{R})}}{R^{tq+1}}+C\int_{B^c_\frac{3R}{2}}\frac{|u(y)|^{q-1}}{|y|^{d+tq+1}}\,dy \\
&\le C\frac{\|u\|^{q-1}_{L^\infty(B_{R})}}{R^{tq+1}}+C\mathrm{Tail}_{q-1,tq+1}(u;R)^{q-1}
\end{align*}
with the positive constant $C$ depending upon $d,t,q,\|a\|_{\infty}$, where we used \eqref{2-1-4} and the facts that $1-\eta_R(y)=1$ for $|y|\ge\frac{7R}{4}$ and $1-\eta_R(y)=0$ for $|y|\le\frac{3R}{2}$. In a similar manner, we derive
$$
|I_1|\le C\frac{\|u\|^{p-1}_{L^\infty(B_{R})}}{R^{sp+1}}+C\mathrm{Tail}_{p-1,sp+1}(u;R)^{p-1},
$$
with the positive constant $C$ depending upon $d,s,p$.

All in all, we infer
\begin{align*}
|\mathcal{L}_u^av_e(x)|\le C\Bigg(&\frac{\|u\|^{p-1}_{L^\infty(B_{R})}}{R^{sp+1}}+\frac{\|u\|^{q-1}_{L^\infty(B_{R})}}{R^{tq}}
+\frac{\|u\|^{q-1}_{L^\infty(B_{R})}}{R^{tq+1}}+\mathrm{Tail}_{p-1,sp+1}(u;R)^{p-1}\\
&+\mathrm{Tail}_{q-1,tq}(u;R)^{q-1}+\mathrm{Tail}_{q-1,tq+1}(u;R)^{q-1}\Bigg),
\end{align*}
where $C>0$ depends on $d,s,t,p,q$ and $\|a\|_{\rm lip}$. Finally, by means of Lemma \ref{lem2-0}, the desired result is naturally obtained.
\end{proof}

\section{Improvement of gradient norm}
\label{sec3}

This section aims at establishing reduction of the gradient norm for viscosity solutions to Eq. \eqref{main}. We mainly follow the ideas developed in \cite{Sil25} to derive a type of improvement of oscillation lemma for the nonlocal double phase equation. However, there are some nontrivial treatment owing to the nonstandard ($p,q$)-growth. First, we force the linearized equation in Proposition \ref{lem2-1} to fall into a smallness regime under some hypotheses. This is our main tool for improving the gradient norm in this section.

\begin{lemma}
\label{lem3-1}
Let $2\le p\le q$, the coefficient $a(\cdot)$ fulfill $(A_2)$--$(A_4)$ and let $v_e(x)=\eta\left(\frac{x}{2^{K_0}}\right)(e\cdot\nabla u(x))$ be the localized directional derivative of $u$. Suppose $u$ is a solution to Eq. \eqref{main} in a ball $B_{2^{K_0+1}}$, such that
\begin{itemize}


\item[(1)] $u(0)=0$ and $\|\nabla u\|_{L^\infty(B_{2^k})}\le (1-\delta)^{-k}$ for $k=0,1,\cdots,K_0$; 

\smallskip

\item[(2)] $\mathrm{Tail}_{q-1,tq}(u;2^{K_0})\le \varepsilon_1$.
\end{itemize}
Then under the conditions that
\begin{equation}
\label{3-1-1}
\begin{cases}
&tq\le sp+1\\[2mm]
&p<\frac{2}{1-s}, q<\frac{1}{1-t},
\end{cases}
\end{equation}
we conclude
\begin{equation}
\label{3-1-2}
|\mathcal{L}^a_uv_e(x)|\le \varepsilon_0, \quad x\in B_1,
\end{equation}
where $\varepsilon_0$ is arbitrarily small when we take $\varepsilon_1,\delta$ small and $K_0$ large.
\end{lemma}

\begin{proof}
From Proposition \ref{lem2-1} with $R=2^{K_0}$, we can discover
\begin{align*}
|\mathcal{L}^a_uv_e(x)|&\le C\Bigg(\frac{\|u\|^{p-1}_{L^\infty(B_{2^{K_0}})}}{2^{K_0(sp+1)}}+\frac{\|u\|^{q-1}_{L^\infty(B_{2^{K_0}})}}{2^{K_0tq}}
+\frac{\|u\|^{q-1}_{L^\infty(B_{2^{K_0}})}}{2^{K_0(tq+1)}}+(1+2^{-K_0})\\
&\qquad\times\mathrm{Tail}_{q-1,tq}(u;2^{K_0})^{q-1}+2^{-K_0\left(sp+1-\frac{p-1}{q-1}tq\right)}\mathrm{Tail}_{q-1,tq}(u;2^{K_0})^{p-1}   \Bigg)\\
&\le C\left(\frac{\|u\|^{p-1}_{L^\infty(B_{2^{K_0}})}}{2^{K_0(sp+1)}}+\frac{\|u\|^{q-1}_{L^\infty(B_{2^{K_0}})}}{2^{K_0tq}}+
2\varepsilon_1^{q-1}+\varepsilon_1^{p-1}\right),
\end{align*}
where in the last line we have used $\eqref{3-1-1}_1$, and $C>0$ depends only on $d,s,t,p,q,\|a\|_{\rm lip}$. Observe that by (1),
$$
\|u\|_{L^\infty(B_{2^{K_0}})}=\sup_{x\in B_{2^{K_0}}}|u(x)-u(0)|\le (1-\delta)^{-K_0}2^{K_0}.
$$
Then we get
\begin{align*}
|\mathcal{L}^a_uv_e(x)|\le C\Big[&\left((1-\delta)^{-(p-1)}2^{p(1-s)-2}\right)^{K_0}+\left((1-\delta)^{-(q-1)}2^{q(1-t)-1}\right)^{K_0}\\
&+
2\varepsilon_1^{q-1}+\varepsilon_1^{p-1}\Big]=:\varepsilon_0.
\end{align*}
Now we choose $\delta>0$ small enough, depending only upon $s,t,p,q$, which together with $\eqref{3-1-1}_2$ could ensure
$$
(1-\delta)^{-(p-1)}2^{p(1-s)-2}<1
$$
and
$$
 (1-\delta)^{-(q-1)}2^{q(1-t)-1}<1.
$$
Thus if $\varepsilon_1$ and $K_0$ are separately sufficiently small and large, then we make $\varepsilon_0$ small enough.
\end{proof}

In the sequel, based on Eq. \eqref{3-1-2}, we perform a modified De Giorgi iteration to verify the improvement of flatness for the directional derivatives of $u$ that is stated in Lemma \ref{lem3-0} as the core result of this section.

\begin{lemma}
\label{lem3-0}
Let the assumptions $(A_1)$--$(A_4)$ on $a(\cdot)$ and $(H_2), (H_3)$ be in force. For any $r,\mu>0$, there are small $\delta,\varepsilon_1>0$ and large $K_0>0$, depending on $r,\mu,d,s,t,p,q,\|a\|_{\rm lip}$, such that if the conditions below hold:
\begin{itemize}
\item[(1)] $u$ solves Eq. \eqref{main} in $B_{2^{K_0+1}}$ with $u(0)=0$;

\smallskip

\item[(2)] $\|\nabla u\|_{L^\infty(B_{2^k})}\le (1-\delta)^{-k}$ for $k=0,1,\cdots,K_0$; 

\smallskip

\item[(3)] $\mathrm{Tail}_{q-1,tq}(u;2^{K_0})\le \varepsilon_1$;

\smallskip

\item[(4)] $|\{x\in B_1:|\nabla u(x)-e|\ge r\}|\ge \mu$ for some $e\in \mathbb{S}^{d-1}$.
\end{itemize}
Then we have $e\cdot \nabla u\le 1-\delta$ in $B_\frac{1}{2}$.
\end{lemma}

To verify Lemma \ref{lem3-0}, we need first establish the forthcoming Lemma \ref{lem3-0-1}. Then Lemma \ref{lem3-0} is a consequence of this lemma.

\begin{lemma}
\label{lem3-0-1}
Let $v_e$ be given by Lemma \ref{lem3-1}. Under the conditions of Lemma \ref{lem3-0}, there is a point $x_0\in \partial B_\frac{1}{2}$ such that
$$
v_e(x)\le 1-\delta\quad \text{for } x\in B_\frac{1}{16}(x_0).
$$
\end{lemma}

Before proving this lemma, let us introduce some notations to be used later. Given $e\in \mathbb{S}^{d-1}, \Gamma\subset\mathbb{S}^{d-1}$ and $\rho,b>0$, we define
\begin{align*}
&\mathcal{A}_{\rho}=\{x\in B_1:|\nabla u(x)-e|\ge \rho\},\\
&\Gamma(e,b)=\big\{\nu\in \mathbb{S}^{d-1}:|\nu\cdot e|>b\big\},\\
&\mathcal{C}_{\Gamma}(x)=\{x+t\nu:t\in[0,1),\ \nu\in\Gamma\}.
\end{align*}
The choice of the point $x_0$ in Lemma \ref{lem3-0-1} is realized by \cite[Lemma 4.4]{Sil25} that we restate here for convenience.

\begin{lemma}
\label{lem3-0-2}
Let $|\mathcal{A}_r|\ge\mu$. Then there is a point $x_0\in \partial B_\frac{1}{2}$ fulfilling
$$
|\mathcal{C}_{\Gamma(e,1/6)}(x)\cap \mathcal{A}_r|\ge\frac{\mu}{2} \quad\text{for all } x\in B_\frac{1}{8}(x_0).
$$
\end{lemma}

Lemma \ref{lem3-0-2} guarantees the existence of a ball $B_\frac{1}{8}(x_0)$ where, from any point inside it, the cone of non-degenerate directions of the kernel $K^a_u$ intersects $\mathcal{A}_r$. We then in Lemma \ref{lem3-0-1} exploit this fact to get a better bound on $v_e$ inside the smaller ball $B_\frac{1}{16}(x_0)$. Ultimately, since the same intersection property holds for every point in $B_\frac{1}{2}$, we can propagate the improved bounds to establish Lemma \ref{lem3-0}.

\medskip

We now prove Lemma \ref{lem3-0-1} following the ideas of De Giorgi iteration. We first describe the setup regarding the proof of Lemma \ref{lem3-0-1}. Let the smooth bump functions $\eta_k$ ($k=1,2,\cdots$) satisfy that
$$
\eta_k=1 \text{ in } B_{\frac{1}{16}+\frac{1}{4^{k+2}}}(x_0), \ \  \eta_k=0  \text{ on } B^c_{\frac{1}{16}+\frac{1}{4^{k+1}}}(x_0),
$$
and
$$
 |\nabla \eta_k|\le C4^{k},  \ \ |D^2 \eta_k|\le C4^{2k},
$$
where the point $x_0$ is from Lemma \ref{lem3-0-2}. For the parameter $\delta>0$ in Lemma \ref{lem3-0-1}, define
$$
\delta_k=(1+2^{-k})\delta
$$
and
$$
A_k=\{x\in B_1:v_e(x)>\varphi_k(x):=1-\delta_k\eta_k(x)\}.
$$
Here we observe that by $|v_e|\le1$, $A_k$ is a subset of the support of $\eta_k$. Through the De Giorgi iteration scheme, our ultimate purpose is to show
$$
|A_k|\rightarrow0 \quad\text{ as } k\rightarrow\infty,
$$
 which implies $v_e\le 1-\delta$ in $B_{1/16}(x_0)$.

Next, we denote
\begin{equation*}
v_k(x)=
\begin{cases}
(v_e(x)-(1-\delta_k\eta_k(x)))_+ &\text{in } B_1,\\[2mm]
0 &\text{on } \mathbb{R}^d\setminus B_1.
\end{cases}
\end{equation*}
We multiply the equation $\mathcal{L}^a_uv_e(x)\le \varepsilon_0$ by the function $v_k$, and then integrate over $B_1$ to arrive at
\begin{equation}
\label{m3}
\int_{B_1}[\mathcal{L}^a_uv_e(x)]v_k(x)\,dx\le\varepsilon_0\int_{B_1}v_k(x)\,dx,
\end{equation}
where we can find that the integral on the left-hand side makes sense whenever $u\in L^{p-1}_{sp}\cap L^{q-1}_{tq}$ is Lipschitz in $B_1$, and $v_e$ (and thus also $v_k$) is in $H^{1-p(1-s)/2}$.

We shall divide the integral on the left-hand side into four parts. Through the definition of $K^a_u(x,y)$ and the symmetry of $a(\cdot)$, we can know that the kernel $K^a_u(x,y)$ also is symmetric, i.e., $K^a_u(x,y)=K^a_u(y,x)$. So the operator $\mathcal{L}^a_u$ has a variational structure. Thus by simple manipulations we obtain
\begin{align*}
\int_{B_1}[\mathcal{L}^a_uv_e(x)]v_k(x)\,dx\ge\,&\int_{B_1}\int_{B_1}(v_k(x)-v_k(y))v_k(x)K^a_u(x,y)\,dxdy\\
&+\int_{B_1}\int_{B_1}(\varphi_k(x)-\varphi_k(y))v_k(x)K^a_u(x,y)\,dxdy\\
&+\int_{B_1}\int_{B_1}[\varphi_k(y)-v_e(y)]_+v_k(x)K^a_u(x,y)\,dxdy\\
&+\int_{\mathbb{R}^d\setminus B_1}\int_{B_1}(\varphi_k(x)-v_e(y))v_k(x)K^a_u(x,y)\,dxdy\\
=:\,&J_1+J_2+J_3+J_4.
\end{align*}
In what follows, we aim at evaluating the lower bounds on these four integrals term by term. Then merging theses estimates with \eqref{m3} will lead to a recursive inequality on the measure of the sets $A_k$, so that we can discuss the limit of $|A_k|$.

First, we are going to deal with the nonlocal integral $J_4$. Now let us check properties of the kernel $K^a_u(x,y)$.

\begin{lemma}
\label{lem3-2}
Assume that $u$ fulfills the conditions given by Lemma \ref{lem3-1}. Then under the assumption \eqref{3-1-1}, if $\delta,\varepsilon_1>0$ are small enough, there holds that
$$
\int_{\mathbb{R}^d\setminus B_1}K^a_u(x,y)\,dy\le C \qquad \text{for } x\in B_\frac{3}{4},
$$
where $C>0$ is a universal constant. 
\end{lemma}

\begin{proof}
For $x\in B_\frac{3}{4}$, via (1) in Lemma \ref{lem3-1}, it holds $|u(x)|\le1$. Taking into account boundedness of $a$, H\"{o}lder inequality and Lemma \ref{lem2-0} (3) and (4), we could estimate
\begin{align}
\label{3-2-1}
\int_{\mathbb{R}^d\setminus B_1}K^a_u(x,y)\,dy&\le C\int_{\mathbb{R}^d\setminus B_1}\left(\frac{1+|u(y)|^{p-2}}{|x-y|^{d+sp}}+\frac{1+|u(y)|^{q-2}}{|x-y|^{d+tq}}\right)\,dy \nonumber\\
&\le C\int_{\mathbb{R}^d\setminus B_1}\left(\frac{1+|u(y)|^{p-2}}{|y|^{d+sp}}+\frac{1+|u(y)|^{q-2}}{|y|^{d+tq}}\right)\,dy  \nonumber\\
&\le C\left[1+\left(\int_{\mathbb{R}^d\setminus B_1}\frac{|u(y)|^{p-1}}{|y|^{d+sp+1}}\,dy\right)^\frac{p-2}{p-1}+\left(\int_{\mathbb{R}^d\setminus B_1}\frac{|u(y)|^{q-1}}{|y|^{d+tq}}\,dy \right)^\frac{q-2}{q-1}\right]\nonumber\\
&\le C\left(1+\mathrm{Tail}_{q-1,tq}(u;1)^{p-2}+\mathrm{Tail}_{q-1,tq}(u;1)^{q-2}\right),
\end{align}
where $C>0$ depends upon $d,s,t,p,q$ and $\|a\|_\infty$. When $y\in B_{2^{k+1}}\setminus B_{2^{k}}$ with $k\le K_0$, from Lemma \ref{lem3-1} (1), we have
$$
|u(y)|\le (1-\delta)^{-(k+1)}2^{k+1}
$$
and
$$
\int_{B_{2^{k+1}}\setminus B_{2^k}}\frac{dy}{|y|^{d+tq}}\le C2^{-(k+1)tq}.
$$
We can further evaluate
\begin{align}
\label{3-2-2}
\int_{\mathbb{R}^d\setminus B_1}\frac{|u(y)|^{q-1}}{|y|^{d+tq}}\,dy&= \sum^{K_0-1}_{k=0}\int_{B_{2^{k+1}}\setminus B_{2^k}}\frac{|u(y)|^{q-1}}{|y|^{d+tq}}\,dy+\int_{B^c_{2^{K_0}}}\frac{|u(y)|^{q-1}}{|y|^{d+tq}}\,dy \nonumber\\
&\le C\sum^{K_0-1}_{k=0}\left[(1-\delta)^{-(q-1)}2^{q-tq-1}\right]^{k+1}+\varepsilon_1^{q-1}.
\end{align}
Let $\theta_0:=2^{q(1-t)-1}<1$ via $q<\frac{1}{1-t}$. Taking $\delta<1$ small enough yields that
$$
\theta:=(1-\delta)^{-(q-1)}2^{q(1-t)-1}<\frac{1+\theta_0}{2}<1.
$$
Then,
$$
\sum^{K_0-1}_{k=0}\left[(1-\delta)^{-(q-1)}2^{q-tq-1}\right]^{k+1}=\sum^{K_0-1}_{k=0}\theta^{k+1}\le \frac{\theta}{1-\theta}.
$$
That is, if $\varepsilon_1<1$ is small, the display \eqref{3-2-2} becomes
\begin{equation}
\label{3-2-3}
\int_{\mathbb{R}^d\setminus B_1}\frac{|u(y)|^{q-1}}{|y|^{d+tq}}dy\le C_1
\end{equation}
with $C_1>0$ depending on $d,t,q$. Combining the inequalities \eqref{3-2-1} and \eqref{3-2-3}, the desired result is proved now.
\end{proof}

\begin{lemma}
\label{lem3-3}
Assume that $u$ satisfies the conditions given by Lemma \ref{lem3-1}. Then under the condition \eqref{3-1-1}, if $\delta,\varepsilon_1>0$ are small enough, then for $x\in B_\frac{3}{4}$ there holds that
\begin{equation}
\label{3-3-1}
\sum^\infty_{n=0}\int_{B_{2^{n+1}}\setminus B_{2^n}}(1-\delta)^{-(n+1)}K^a_u(x,y)\,dy\le C
\end{equation}
and
\begin{equation}
\label{3-3-2}
\sum^\infty_{n=0}\int_{B_{2^{n+1}}\setminus B_{2^n}}\left((1-\delta)^{-(n+1)}-1\right)K^a_u(x,y)\,dy\le C(\delta),
\end{equation}
where $C>0$ depends on $d,s,t,p,q,\|a\|_\infty$, and $C(\delta)>0$ depends also on $\delta$, and $C(\delta)\rightarrow0$ as $\delta\rightarrow0$.
\end{lemma}

\begin{proof}
For $x\in B_\frac{3}{4}$ and $y\in B_{2^{n+1}}\setminus B_{2^n}$, it holds $|u(x)|\le1$ and
$$
|y|\le |y-x|\left(1+\frac{|x|}{|y-x|}\right)\le 4|y-x|
$$
and
$$
(1-\delta)^{-(n+1)}=(1-\delta)^{-1}2^{\log_2(1-\delta)^{-n}}=(1-\delta)^{-1}2^{n\alpha_0}\le(1-\delta)^{-1}|y|^{\alpha_0}\le 2|y|^{\alpha_0},
$$
where $\alpha_0:=\log_2(1-\delta)^{-1}$ and $\delta\in(0,\frac{1}{2}]$. We proceed to select $\delta$ small enough such that $\alpha_0<sp(\le tq)$, and then we can treat
\begin{align*}
&\quad\sum^\infty_{n=0}\int_{B_{2^{n+1}}\setminus B_{2^n}}(1-\delta)^{-(n+1)}K^a_u(x,y)\,dy\\
&\le C\sum^\infty_{n=0}\int_{B_{2^{n+1}}\setminus B_{2^n}}(1-\delta)^{-(n+1)}\left(\frac{1+|u(y)|^{p-2}}{|y|^{d+sp}}+\frac{1+|u(y)|^{q-2}}{|y|^{d+tq}}\right)dy \\
&\le C\sum^\infty_{n=0}\int_{B_{2^{n+1}}\setminus B_{2^n}}\frac{1+|u(y)|^{p-2}}{|y|^{d+sp-\alpha_0}}+\frac{1+|u(y)|^{q-2}}{|y|^{d+tq-\alpha_0}}\,dy  \\
&\le C\left(1+\int_{\mathbb{R}^d\setminus B_1}\frac{|u(y)|^{p-2}}{|y|^{d+sp-\alpha_0}}dy+\int_{\mathbb{R}^d\setminus B_1}\frac{|u(y)|^{q-2}}{|y|^{d+tq-\alpha_0}}\,dy\right).
\end{align*}
Now we can pick $\delta $ so small that $\alpha_0<\frac{tq}{q-1}$, and use H\"{o}lder inequality to get
\begin{align*}
\int_{B^c_1}\frac{|u(y)|^{q-2}}{|y|^{d+tq-\alpha_0}}\,dy&\le \left(\int_{B^c_1}\frac{|u(y)|^{q-1}}{|y|^{d+tq}}\,dy\right)^\frac{q-2}{q-1}\left(\int_{B^c_1}\frac{dy}{|y|^{d+tq-\alpha_0(q-1)}}
\right)^\frac{1}{q-1}\\
&=C\mathrm{Tail}_{q-1,tq}(u;1)^{q-2}.
\end{align*}

By means of the facts that $tq\le sp+1$ and $p<\frac{2}{1-s}$, we can find such small $\delta$ that
\begin{align*}
&0<tq-(tq-sp+\alpha_0)\frac{q-1}{q-p+1} \\
\Leftrightarrow\ &(sp-\alpha_0)\frac{q-1}{q-p+1}+tq\frac{2-p}{q-p+1}>0 \\
\Leftrightarrow\ &(sp-\alpha_0)(q-1)+tq(2-p)>0,
\end{align*}
namely,
\begin{align}
\label{3-3-3}
sp-tq\frac{p-2}{q-1}>\alpha_0.
\end{align}
In fact, it follows from $tq\le sp+1$ and $2\le p<\frac{2}{1-s}$ that $sp-tq\frac{p-2}{q-1}>0$, which can be seen from the following
\begin{align*}
sp-tq\frac{p-2}{q-1}&\ge \frac{1}{q-1}[(q-1)sp-(sp+1)(p-2)]\\
&=\frac{1}{q-1}[(q-p+1)sp-(p-2)]\\
&>\frac{1}{q-1}[(q-p+1)sp-sp]=\frac{q-p}{q-1}sp\ge0.
\end{align*}
So if $\delta$ is sufficiently small, the inequality \eqref{3-3-3} could be achieved. At this stage, we obtain
\begin{align*}
\int_{B^c_1}\frac{|u(y)|^{p-2}}{|y|^{d+sp-\alpha_0}}\,dy&\le \left(\int_{B^c_1}\frac{|u(y)|^{q-1}}{|y|^{d+tq}}\,dy\right)^\frac{p-2}{q-1}\left(\int_{B^c_1}\frac{dy}{|y|^{d+tq-(tq-sp+\alpha_0)\frac{q-1}{q-p+1}}}
\right)^\frac{q-p+1}{q-1}\\
&=C\mathrm{Tail}_{q-1,tq}(u;1)^{p-2}.
\end{align*}
Analogous to \eqref{3-2-2}, we prove \eqref{3-3-1} from the displays above.

We next focus on \eqref{3-3-2}. For $y\in B_{2^{n+1}}\setminus B_{2^n}$, we have $(1-\delta)^{-(n+1)}=2^{(n+1)\alpha_0}\le(2|y|)^{\alpha_0}$. Then,
\begin{align*}
&\quad\sum^\infty_{n=0}\int_{B_{2^{n+1}}\setminus B_{2^n}}\left((1-\delta)^{-n-1}-1\right)K^a_u(x,y)\,dy\\
&\le C\sum^\infty_{n=0}\int_{B_{2^{n+1}}\setminus B_{2^n}}\left((2|y|)^{\alpha_0}-1\right)\left(\frac{1+|u(y)|^{p-2}}{|y|^{d+sp}}+\frac{1+|u(y)|^{q-2}}{|y|^{d+tq}}\right)dy \\
&=C\int_{B^c_1}\left(\frac{(2|y|)^{\alpha_0}-1}{|y|^{d+sp}}+\frac{(2|y|)^{\alpha_0}-1}{|y|^{d+tq}}\right)dy\\
&\quad+C\int_{B^c_1}\left((2|y|)^{\alpha_0}-1\right)
\left(\frac{|u(y)|^{p-2}}{|y|^{d+sp}}+\frac{|u(y)|^{q-2}}{|y|^{d+tq}}\right)dy\\
&\le CI^1+C\mathrm{Tail}_{p-1,sp+1}(u;1)^{p-2}\left(\int_{B^c_1}\frac{\big((2|y|)^{\alpha_0}-1\big)^{p-1}}{|y|^{d+sp-p+2}}\,dy\right)^\frac{1}{p-1}\\
&\quad+C\mathrm{Tail}_{q-1,tq}(u;1)^{q-2}\left(\int_{B^c_1}\frac{\big((2|y|)^{\alpha_0}-1\big)^{q-1}}{|y|^{d+tq}}\,dy\right)^\frac{1}{q-1}.
\end{align*}
Here $I^1:=\int_{B^c_1}\left(\frac{(2|y|)^{\alpha_0}-1}{|y|^{d+sp}}+\frac{(2|y|)^{\alpha_0}-1}{|y|^{d+tq}}\right)dy$. Via the dominant convergence theorem, it holds that the integrals
$$
I^1, \quad \int_{B^c_1}\frac{\big((2|y|)^{\alpha_0}-1\big)^{p-1}}{|y|^{d+sp-p+2}}\,dy, \quad \int_{B^c_1}\frac{\big((2|y|)^{\alpha_0}-1\big)^{q-1}}{|y|^{d+tq}}\,dy
$$
all tend to 0 as $\alpha_0\rightarrow0$. Recall $\mathrm{Tail}_{p-1,sp+1}(u;1)\le C\mathrm{Tail}_{q-1,tq}(u;1)$ by Lemma \ref{lem2-0}, and $\alpha_0\rightarrow0$ as $\delta\rightarrow0$. The inequality \eqref{3-3-2} has been deduced from above.
\end{proof}

Now applying Lemmas \ref{lem3-2} and \ref{lem3-3}, we are able to derive the lower bound on the integral $J_4$ as below.

\begin{lemma}
\label{lem3-4}
With the hypotheses of Lemma \ref{lem3-0}, we infer that
\begin{equation*}
J_4=\int_{\mathbb{R}^d\setminus B_1}\int_{B_1}(\varphi_k(x)-v_e(y))v_k(x)K^a_u(x,y)\,dxdy\ge -C(\delta)\int_{B_1}v_k(x)\,dx,
\end{equation*}
where $C(\delta)>0$, depending on $d,s,t,p,q,\|a\|_\infty$ and $\delta$, converges to 0 as $\delta\rightarrow0$.
\end{lemma}

\begin{proof}
It is known that the support of $v_e$ is contained in $B_{2^{K_0+1}}$, and $|v_e|\le |\nabla u|\le (1-\delta)^{-n}$ in $B_{2^n}$. Hence it yields that
\begin{align*}
J_4&=\int_{\mathbb{R}^d\setminus B_1}\int_{B_1}[(1-\delta_k\eta_k(x))-v_e(y)]v_k(x)K^a_u(x,y)\,dxdy\\
&=-\int_{B_1}v_k(x)\int_{\mathbb{R}^d\setminus B_1}[(v_e(y)-1)+\delta_k\eta_k(x)]K^a_u(x,y)\,dydx\\
&\ge - \int_{B_1}v_k(x)\,dx\times\sup_{x\in {\rm supp}\,v_k}\int_{\mathbb{R}^d\setminus B_1}[(v_e(y)-1)+\delta_k\eta_k(x)]K^a_u(x,y)\,dy\\
&\ge -\int_{B_1}v_k(x)\,dx\\
&\quad\times\sup_{x\in {\rm supp}\, v_k} \left( \delta_k\int_{\mathbb{R}^d\setminus B_1}K^a_u(x,y)\,dy+\sum^\infty_{n=0}\int_{B_{2^{n+1}}\setminus B_{2^n}}\left((1-\delta)^{-n-1}-1\right)K^a_u(x,y)\,dy\right) \\
&\ge -(C\delta_k+C(\delta))\int_{B_1}v_k(x)\,dx,
\end{align*}
where in the last inequality we have employed Lemmas \ref{lem3-2} and \ref{lem3-3}. Remember $\delta_k\le 2\delta$ and $C(\delta)\rightarrow0$ as $\delta\rightarrow0$. Now the result follows.
\end{proof}

Next, the terms $J_2$, whose bound is related to $J_1$, will be handled. 

\begin{lemma}
\label{lem3-5}
There holds that
\begin{equation*}
|J_2|=\left|\int_{B_1}\int_{B_1}(\varphi_k(x)-\varphi_k(y))v_k(x)K^a_u(x,y)\,dxdy\right|\le C4^k\delta_k|A_k|^\frac{1}{2}\sqrt{J_1},
\end{equation*}
where $C>0$ is a universal constant. 
\end{lemma}

\begin{proof}
In view of the symmetry $K^a_u(x,y)=K^a_u(y,x)$, we rewrite $J_2$ as
\begin{align*}
J_2&=\int_{B_1}\int_{B_1}(\varphi_k(x)-\varphi_k(y))v_k(x)K^a_u(x,y)\,dydx\\
&=-\frac{\delta_k}{2} \int_{B_1} \int_{B_1}(\eta_k(x) - \eta_k(y)) (v_k(x) - v_k(y)) K^a_u(x,y) \, dydx.
\end{align*}
Then, utilizing the Lipschitz continuity of \(\eta_k\) and \(u\), as well as the boundedness of \(a(\cdot)\), we directly compute
\begin{align*}
|J_2| 
&\le\frac{\delta_k}{2}\left(\int_{B_1}\int_{B_1}(\eta_k(x)-\eta_k(y))^2 \chi_{\{v_k(x)>0\ \text{or }v_k(y)>0\}}K^a_u(x,y) \,dydx \right)^\frac{1}{2} \\
&\quad \times \left(\int_{B_1}\int_{B_1} (v_k(x) -v_k(y))^2 K^a_u(x,y) \,dydx \right)^\frac{1}{2} \\
&\le C\delta_k \|\eta_k\|_{{\rm lip}(B_1)}\left(\int_{B_1}\int_{B_1}\left(\frac{1}{|x-y|^{d+sp-p}}+\frac{1}{|x-y|^{d+tq-q}}\right)\chi_{\{v_k(x)>0\}} \,dydx \right)^{\frac{1}{2}} \\
&\quad \times \left(\int_{B_1}\int_{B_1} (v_k(x) -v_k(y))^2 K^a_u(x,y) \,dydx \right)^\frac{1}{2} \\
&\le C\delta_k\|\eta_k\|_{{\rm lip}(B_1)}\left(\int_{B_1}\chi_{\{v_k(x)>0\}}\,dx \int_{B_2(x)}\bigg(\frac{1}{|x-y|^{d+p(s-1)}}+\frac{1}{|x-y|^{d+q(t-1)}}\bigg)\, dy \right)^{\frac{1}{2}} \\
&\quad \times \left(\int_{B_1}\int_{B_1} (v_k(x) -v_k(y))^2 K^a_u(x,y) \,dydx \right)^\frac{1}{2} \\
&\leq C 4^k\delta_k|A_{k}|^\frac{1}{2} \sqrt{J_1}, 
\end{align*}
where $C>0$ depends upon $d,s,t,p,q$ and $\|a\|_\infty$. This completes the proof.
\end{proof}

For the estimates of $J_1,J_3$, we only need to take the $p$-growth terms as the lower bounds of the integrands, given that the $q$-growth term is always nonnegative. With this, we can immediately invoke Lemmas 4.10 and 4.12, which were derived in \cite{Sil25} for fractional $p$-Laplace equations. These two lemmas are stated as follows.
\begin{lemma}
\label{lem3-7-0}
Let \(r, \mu>0\) and \(u \in C^1(B_2)\). If \(|\mathcal{A}_r|>\mu|B_1|\), there is a \(\delta>0\) depending on \(r,\mu, d, s, p\) such that
$$
\int_{B_1}\int_{B_1}[\varphi_k(y)-v_e(y)]_+v_k(x)K_u(x,y)\,dxdy\ge Cr^2\mu^2\int_{B_1}v_k(x)\,dx
$$
with
\begin{equation}
\label{3-6-2}
K_u(x,y)=(p-1)\frac{|u(x)-u(y)|^{p-2}}{|x-y|^{d+sp}}.
\end{equation}
Here the constant $C>0$ depends only on $d,s,p$.
\end{lemma}

In Lemma \ref{lem3-7-0}, the $C^1$-regularity of $u$ is simply a qualitative condition for the proof of this lemma. To be precise, only the continuity of $\nabla u$ was used here.

\begin{lemma}
\label{lem3-6-0}
It holds that
$$
\int_{B_1}\int_{B_1}(v_k(x)-v_k(y))^2K_u(x,y)\,dxdy\ge C\delta^22^{-2k}|A_{k+1}||A_{k}|^{-\frac{2-p(1-s)}{d}}
$$
for every $k\ge1$, where the constant $C>0$ depends only on $d,s,p$.
\end{lemma}

Now in the nonlocal double phase framework, we have the upcoming estimates.

\begin{lemma}
\label{lem3-6}
It holds that
$$
J_1=\int_{B_1}\int_{B_1}(v_k(x)-v_k(y))v_k(x)K^a_u(x,y)\,dxdy\ge C\delta^22^{-2k}|A_{k+1}||A_{k}|^{-\frac{2-p(1-s)}{d}}
$$
for every $k\ge1$, where the constant $C>0$ depends only on $d,s,p$.
\end{lemma}

\begin{proof}
Note that by the symmetry $K^a_u(x,y)=K^a_u(y,x)$, $J_1$ can be rewritten as
\begin{equation*}
\begin{split}
J_1&=\frac{1}{2} \int_{B_1}\int_{B_1}(v_k(x)-v_k(y))^2K^a_u(x,y)\,dxdy\\
&\ge \frac{1}{2} \int_{B_1}\int_{B_1}(v_k(x)-v_k(y))^2K_u(x,y)\,dxdy
\end{split}
\end{equation*}
with $K_u(x,y)$ same as \eqref{3-6-2}. Now the subsequent processes fall completely into the framework of fractional $p$-Laplacian equations, so we apply Lemma \ref{lem3-6-0} to get the lower bound. 
\end{proof}

\begin{lemma}
\label{lem3-7}
Let \(r, \mu>0\) and \(u \in C^1(B_2)\). If \(|\mathcal{A}_r|>\mu|B_1|\), there is a \(\delta>0\) depending on \(r,\mu, d, s, p\) such that
$$
J_3=\int_{B_1}\int_{B_1}[\varphi_k(y)-v_e(y)]_+v_k(x)K^a_u(x,y)\,dxdy\ge Cr^2\mu^2\int_{B_1}v_k(x)\,dx.
$$
Here the constant $C>0$ depends only on $d,s,p$.
\end{lemma}

\begin{proof}
Observe a simple fact that
\begin{equation}
\label{3-6-1}
J_3\ge \int_{B_1}\int_{B_1}[\varphi_k(x)-v_e(y)]_+v_k(x)K_u(x,y)\,dxdy
\end{equation}
with $K_u(x,y)$ same as \eqref{3-6-2}. Hence the subsequent procedures fall fully into the framework of fractional $p$-Laplacian equations, so we apply Lemma \ref{lem3-7-0} to get the lower bound. 
\end{proof}

Now we implement the De Giorgi iteration to show Lemma \ref{lem3-0-1}. We are going to deduce a recursive inequality of $|A_k|$ and show one of these level sets has a very small measure. Then by applying Lemma \ref{geo} to get $|A_k|\rightarrow0$ as $k\rightarrow0$, which may imply $v_e\le 1-\delta$ in $B_\frac{1}{16}(x_0)$. 

\begin{proof}[\textbf{Proof of Lemma \ref{lem3-0-1}}] It follows from the inequality \eqref{m3} and the estimates obtained in Lemmas \ref{lem3-4}--\ref{lem3-5} and Lemmas \ref{lem3-6}--\ref{lem3-7} that
\begin{equation}
\label{it}
C(r,\mu)\int_{B_1}v_k\,dx +J_1\le C4^k\delta_k|A_k|^\frac{1}{2}\sqrt{J_1},
\end{equation}
where we need notice the lower bound on $J_3$ can absorb $\varepsilon_0\int_{B_1}v_k\,dx$ if $\varepsilon_0$ is small enough, and also $J_4$ if $\delta$ is sufficiently small. From \eqref{it} and Cauchy inequality,
$$
J_1\le C2^{4k}\delta^2_k|A_k|.
$$
This inequality together with Lemma \ref{lem3-6} leads to
$$
|A_{k+1}|\le C2^{6k}|A_k|^{1+\gamma} \quad\text{for } k=2,3,\cdots,
$$
where $\gamma=\frac{2-p(1-s)}{d}>0$ and $C$ depends on $d,s,t,p,q,\|a\|_\infty$.

On the other hand, exploiting \eqref{it} with $k=1$ and using Cauchy inequality again, we have
\begin{align*}
C(r,\mu)\int_{B_1}v_1\,dx \le C\delta^2_1|A_1|.
\end{align*}
Furthermore, there holds that
$$
\int_{B_1}v_1\,dx\ge 2^{-3}\delta_1|A_2|.
$$
Therefore, $|A_2|\le C(r,\mu)\delta$. It is easy to know $|A_2|$ is as small as we want if $\delta$ is sufficiently small. Finally, we invoke Lemma \ref{geo} to get $|A_k|\rightarrow0$ as $k\rightarrow0$ and so the desired result follows.
\end{proof}

\medskip

The proof of Lemma \ref{lem3-0} follows from the same type of arguments as above and a covering argument; moreover, it is identical to that of \cite[Lemma 4.1]{Sil25}. To avoid excessive repetition, we omit the details here. 



If the measure density condition $\left|\{x\in B_1:|\nabla u(x)-e|\ge r\}\right|\ge\mu|B_1|$ holds for all $e\in \mathbb{S}^{d-1}$ and for the rescaling of $u$ up to a certain scale, we could refine by iterating Lemma \ref{lem3-0} the upper bound on $|\nabla u(x)|$ in dyadic balls.

\begin{lemma}
\label{lem3-9}
Let the hypotheses $(A_1)$--$(A_4)$ and $(H_1)$--$(H_3)$ be in force. For any $r,\mu>0$, there are small $\delta,\varepsilon_1>0$ and large $K_0>0$, all of which depend on $r,\mu,d,s,t,p,q,\|a\|_{\rm lip}$, such that if the following requirements hold
\begin{itemize}


\item[(1)] $u$ solves Eq. \eqref{main} in $B_{2^{K_0+1}}$ with $u(0)=0$;

\smallskip

\item[(2)] $\|\nabla u\|_{L^\infty(B_{2^n})}\le (1-\delta)^{-n}$ for $n=0,1,\cdots,K_0$; 

\smallskip

\item[(3)] $\mathrm{Tail}_{q-1,tq}(u;2^{K_0})\le \varepsilon_1$,
\end{itemize}
and moreover there is an integer $n_0\ge0$ fulfilling
\begin{equation}
\label{3-9-1}
\left|\left\{x\in B_{2^{-n}}:|\nabla u(x)-(1-\delta)^ne|\ge r(1-\delta)^n\right\}\right|\ge\mu|B_{2^{-n}}|
\end{equation}
for any $ e\in \mathbb{S}^{d-1}$ and $ n=0,1,\cdots,n_0$, then we conclude that
\begin{equation}
\label{3-9-2}
|\nabla u|\le (1-\delta)^n \quad \text{in } B_{2^{-n}} \ \text{for } n=0,1,\cdots,n_0+1.
\end{equation}
\end{lemma}

\begin{proof}
For $n_0=0$, this statement is plain from Lemma \ref{lem3-0} stating that $e\cdot\nabla u\le 1-\delta$ in $B_\frac{1}{2}$ for all $ e\in \mathbb{S}^{d-1}$, i.e., $|\nabla u|\le 1-\delta$ in $B_\frac{1}{2}$.

Suppose this conclusion holds true for $n=0,1,\cdots,n_0$, and then we shall prove $|\nabla u|\le (1-\delta)^{n_0+1}$ in $B_{2^{-n_0-1}}$. Consider a scaling function
$$
v(x)=(1-\delta)^{-n_0}2^{n_0}u(2^{-n_0}x).
$$
Now let us examine that $v$ satisfies the preconditions of Lemma \ref{lem3-0}. Obviously, $v(0)=0$ and
\begin{align*}
\|\nabla v\|_{L^\infty(B_{2^n})}&= (1-\delta)^{-n_0}\left\|\nabla u\left(2^{-n_0}x\right)\right\|_{L^\infty(B_{2^n})}\\
&=(1-\delta)^{-n_0}\|\nabla u(x)\|_{L^\infty(B_{2^{n-n_0}})}\le(1-\delta)^{-n}
\end{align*}
for $n=0,1,\cdots,K_0,\cdots,K_0+n_0$. Next, we justify $\mathrm{Tail}_{q-1,tq}(v;2^{K_0})\le \varepsilon_1$. It is easy to see that
\begin{equation}
\label{3-9-3}
\|v\|_{L^\infty(B_{2^n})}\le (2(1-\delta)^{-1})^n \quad \text{for } n=0,1,\cdots,n_0+K_0.
\end{equation}
Thus by changing variables there holds
\begin{align}
\label{3-9-4}
\mathrm{Tail}_{q-1,tq}(v;2^{n_0+K_0})^{q-1}&= [(1-\delta)^{-(q-1)}2^{q(1-t)-1}]^{n_0}\int_{B^c_{2^{K_0}}}\frac{|u(x)|^{q-1}}{|x|^{d+tq}}\,dx \nonumber\\
&=\theta^{n_0}\mathrm{Tail}_{q-1,tq}(u;2^{K_0})^{q-1}
\end{align}
with $\theta:=(1-\delta)^{-(q-1)}2^{q(1-t)-1}$. Due to $q<\frac{1}{1-t}$, we can take $\delta$ sufficiently close to 0 to get $\theta<1$. By \eqref{3-9-3}, \eqref{3-9-4} and (3) in Lemma \ref{lem3-9}, we continue to check
\begin{align*}
\mathrm{Tail}_{q-1,tq}(v;2^{K_0})^{q+1}&=\sum^{n_0-1}_{n=1}\int_{B_{2^{n+K_0+1}}\setminus B_{2^{n+K_0}}}\frac{|v|^{q-1}}{|y|^{d+tq}}\,dy+
\mathrm{Tail}_{q-1,tq}(v;2^{n_0+K_0})^{q-1}\\
&\le C\sum^{n_0-1}_{n=1}(2^{q-1}(1-\delta)^{-(q-1)})^{n+K_0+1}2^{-tq(n+K_0)}+\theta^{n_0}\varepsilon_1^{q-1}\\
&\le C\frac{\theta^{K_0+2}}{1-\theta}+\theta^{n_0}\varepsilon_1^{q-1}\le\varepsilon_1^{q-1},
\end{align*}
where in the last line we select $K_0$ sufficiently large that depends on $\varepsilon_1$.

In addition, we need show $v$ is a solution to Eq. \eqref{main}.
Since $u$ solves Eq. \eqref{main} in $B_{2^{K_0+1}}$, for $x\in B_{2^{n_0+K_0+1}}$ we can readily have
\begin{align*}
0&=\int_{\mathbb{R}^d}\left[\frac{J_p(u(2^{-n_0}x)-u(y))}{|2^{-n_0}x-y|^{d+sp}}+a(2^{-n_0}x,y)
\frac{J_q(u(2^{-n_0}x)-u(y))}{|2^{-n_0}x-y|^{d+tq}}\right]\,dy\\
&=\int_{\mathbb{R}^d}\Bigg[\frac{((1-\delta)^{p-1}2^{-(p-1)})^{n_0}}{2^{-n_0 sp}}\frac{J_p(v(x)-v(y))}{|x-y|^{d+sp}}\\
&\qquad\quad+
\frac{((1-\delta)^{q-1}2^{-(q-1)})^{n_0}}{2^{-n_0tq}}a(2^{-n_0}x,2^{-n_0}y)
\frac{J_q(v(x)-v(y))}{|x-y|^{d+tq}}\Bigg]\,dy.
\end{align*}
Denote
$$
\overline{a}(x,y)=((1-\delta)^{q-p}2^{tq-sp-q+p})^{n_0}a(2^{-n_0}x,2^{-n_0}y).
$$
Then $v$ solves the following equation in $B_{2^{n_0+K_0+1}}$,
\begin{align}
\label{veq}
0=\int_{\mathbb{R}^d}\left[\frac{|v(x)-v(y)|^{p-2}(v(x)-v(y))}{|x-y|^{d+sp}}+\overline{a}(x,y)
\frac{|v(x)-v(y)|^{q-2}(v(x)-v(y))}{|x-y|^{d+tq}}\right]\,dy.
\end{align}
Via $tq\le sp+q-p$, we see that $(1-\delta)^{q-p}2^{tq-sp-q+p}\le1$, and further $|\overline{a}(x,y)|\le\|a\|_{\infty}$ in $\mathbb{R}^d\times\mathbb{R}^d$. Furthermore, in view of $tq\le sp+1$,
\begin{align*}
|\overline{a}(x_1,y_1)-\overline{a}(x_2,y_2)|&=\left((1-\delta)^{q-p}2^{tq-sp-q+p}\right)^{n_0}\left|a(2^{-n_0}x_1,2^{-n_0}y_1)
-a(2^{-n_0}x_2,2^{-n_0}y_2)\right|  \\
&\le \left((1-\delta)^{q-p}2^{tq-sp-q+p-1}\right)^{n_0}[a]_{\rm lip}(|x_1-x_2|+|y_1-y_2|) \\
&\le [a]_{\rm lip}(|x_1-x_2|+|y_1-y_2|).
\end{align*}
Therefore, the structure conditions on \eqref{veq} are the same as that of Eq. \eqref{main}. At this time, we have verified that $v$ satisfies the preconditions of Lemma \ref{lem3-0}.

Finally, \eqref{3-9-1} implies
$$
\left|\left\{x\in B_1:|\nabla v(x)-e|\ge r\right\}\right|\ge\mu|B_1| \quad\text{for each }   e\in \mathbb{S}^{d-1}.
$$
Hence the desired result follows from Lemma \ref{lem3-0} again.
\end{proof}

Based on Lemma \ref{lem3-9}, we could deduce the following result.

\begin{proposition}
\label{pro3-10}
Let the hypotheses $(A_1)$--$(A_4)$ and $(H_1)$--$(H_3)$ be in force. For any $r,\mu>0$, there exist small $\delta,\varepsilon_1>0$ and large $K_0>0$, depending on $r,\mu,d,s,t,p,q,\|a\|_{\rm lip}$, such that whenever the requirements below hold
\begin{itemize}
\item[(1)] $u(0)=0$;

\smallskip

\item[(2)] $u$ solves Eq. \eqref{main} in $B_{2^{K_0+1}}$;

\smallskip

\item[(3)] $\|\nabla u\|_{L^\infty(B_{2^n})}\le (1-\delta)^{-n}$ for $n=0,1,\cdots,K_0$; 

\smallskip

\item[(4)] $\mathrm{Tail}_{q-1,tq}(u;2^{K_0})\le \varepsilon_1$.
\end{itemize}
Then one of the following statements must be true:
\begin{itemize}
\item There is a nonnegative integer $N$ satisfying
$$
\|\nabla u\|_{L^\infty(B_{2^{-n}})}\le (1-\delta)^{n} \quad\text{for } n=0,1,\cdots,N.
$$
Furthermore, there exists a vector $e\in \mathbb{S}^{d-1}$ such that
$$
\left|\left\{x\in B_{2^{-N}}:|\nabla u(x)-(1-\delta)^Ne|\le r(1-\delta)^N\right\}\right|\ge(1-\mu)|B_{2^{-N}}|.
$$

\smallskip

\item For $\alpha_0=\log_2(1-\delta)^{-1}$ and $C_0=(1-\delta)^{-1}$, it holds that
$$
|\nabla u(x)|\le C_0|x|^{\alpha_0}  \quad\text{for } x\in B_1.
$$
\end{itemize}
\end{proposition}

\begin{proof}
For any $r,\mu>0$, let $N$ be the minimum integer such that the condition \eqref{3-9-1} does not hold. If $N=\infty$, we may apply indefinitely Lemma \ref{lem3-9} to get
$$
\|\nabla u\|_{L^\infty(B_{2^{-n}})}\le (1-\delta)^{n} \ \text{for all } n>0  \quad\Rightarrow\quad |\nabla u(x)|\le \frac{1}{1-\delta}|x|^{\log_2(1-\delta)^{-1}}.
$$
If $N<\infty$, we could apply Lemma \ref{lem3-9} $N-1$ times to have
$$
 \|\nabla u\|_{L^\infty(B_{2^{-n}})}\le (1-\delta)^{n} \quad\text{for } n=0,1,\cdots,N.
$$
Moreover, via the definition of $N$, we know that there is an $e\in \mathbb{S}^{d-1}$ such that
$$
\left|\left\{x\in B_{2^{-N}}:|\nabla u(x)-(1-\delta)^Ne|\le r(1-\delta)^N\right\}\right|\ge(1-\mu)|B_{2^{-N}}|.
$$
This completes the proof.
\end{proof}

We end this section by gathering the properties, to be exploited in the next section, of a rescaling of $u$ from Proposition \ref{pro3-10}.

\begin{corollary}
\label{cor3-11}
Let $N<\infty$ be an integer and $u$ be a function coming from Proposition \ref{pro3-10}. Define
$$
v(x):=(1-\delta)^{-N}2^Nu(2^{-N}x).
$$
Then we have
\begin{itemize}
\item[(1)] $v(0)=0$;

\smallskip

\item[(2)] $v$ is a solution to Eq. \eqref{veq} in $B_{2^{N+K_0+1}}$;

\smallskip

\item[(3)] $\|v\|_{L^\infty(B_{2^n})}\le (2(1-\delta)^{-1})^n $ for $n=0,1,\cdots,N+K_0$;

\smallskip

\item[(4)] $\|\nabla v\|_{L^\infty(B_{2^n})}\le (1-\delta)^{-n}$ for $n=0,1,\cdots,N+K_0$; 

\smallskip

\item[(5)] $\mathrm{Tail}_{q-1,tq}(v;1)\le C_1$ with $C_1>0$ being a constant independent of $N$;

\smallskip

\item[(6)] $\left|\left\{x\in B_1:|\nabla v(x)-e|<r\right\}\right|\ge(1-\mu)|B_1|$ for some $ e\in \mathbb{S}^{d-1}$.
\end{itemize}
\end{corollary}

In a similar way to calculating \eqref{3-2-2} in Lemma \ref{lem3-2}, we can verify the result (5) in Corollary \ref{cor3-11}. We also remark that the structural features of Eq. \eqref{veq} are identical to those of Eq. \eqref{main}, so in the next section we will directly refer to $v$ as a solution to Eq. \eqref{main}. The condition (6) here shall be regarded as a starting point of Section \ref{sec4}.

\section{H\"{o}lder continuity of gradient}
\label{sec4}

This section, based on Proposition \ref{pro3-10}, is dedicated to establishing H\"{o}lder estimates of the gradients of solutions to Eq. \eqref{main}. Now starting from conditions (1)--(6) in Corollary \ref{cor3-11}, especially the condition (6), we first verify that the gradient of solution is everywhere close to a unit vector $e$ in a smaller ball if it is close to $e$ in a large portion of a ball.

\begin{proposition}
\label{pro4-1}
Let the conditions $(A_1)$--$(A_4)$, $(H_2)$ and $tq\le sp+1$ be in force. Suppose that $u$ be a solution to \eqref{main} in $B_2$ with $\|u\|_{\rm lip(B_1)}\le1$ and $\mathrm{Tail}_{q-1,tq}(u;1)\le C_1$. For any $L\in\big(0,\frac{1}{4}\big]$ and $C_1>0$, there exist small enough $\mu,r>0$, depending on $L$ and $C_1$, such that whenever for some  vector $e\in \mathbb{S}^{d-1}$ there holds
\begin{equation}
\label{4-1-1}
|\{x\in B_1:|\nabla u(x)-e|\le r\}| \ge (1-\mu)|B_1|,
\end{equation}
then we have
$$
|\nabla u(x)-e|\le L \quad\text{for every } x\in B_\frac{1}{2}.
$$
\end{proposition}

From the measure density condition \eqref{4-1-1}, we have an oscillation estimate on $u(x)-e\cdot x$ showed by \cite[Lemma 6.2]{Sil25}, which is restated as below.

\begin{lemma}
\label{lem4-2}
Given \(\varepsilon > 0\), there are small numbers \(r ,\mu > 0\) such that if \(|\nabla u| \leq 1\) in \(B_1\) and for some \(e \in \mathbb{S}^{d-1}\) we have
\[
|\{x \in B_1 : |\nabla u(x)-e|<r\}| \geq (1 - \mu)|B_1|,
\]
then it holds
\begin{equation}
\label{4-2-1}
\operatorname{osc}_{x \in B_1} (u(x) - e \cdot x) \leq \varepsilon.
\end{equation}
\end{lemma}

Under the small oscillation condition \eqref{4-2-1}, we could deduce the gradient $\nabla u$ of the solution to Eq. \eqref{main} is close to the vector $e$. The precise statement is as follows.

\begin{lemma}
\label{lem4-3}
Let $u$ be a solution to \eqref{main} in $B_2$ with $\|u\|_{\rm {lip} (B_1)}\le1$ and $\mathrm{Tail}_{q-1,tq}(u;1)\le C_1$. Assume the conditions $(A_1)$--$(A_4)$, $(H_2)$ and $tq\le sp+1$ is true. For given \(L\in\big(0,\frac{1}{4}\big], C_1 > 0\), there is a small number \(\epsilon_0 > 0\) so that if for some \(e \in \mathbb{S}^{d-1}\) we have
\[
\operatorname{osc}_{x \in B_1} (u(x) - e \cdot x) \leq \epsilon_0,
\]
then there holds that
\[
|\nabla u - e| < L \quad \text{in } B_\frac{1}{2}.
\]
Here the constant \(\epsilon_0\) depends on $L,C_1,d,s,t,p,q$ and $\|a\|_{\rm lip}$.
\end{lemma}

At this point, merging Lemmas \ref{lem4-2} and \ref{lem4-3} implies Proposition \ref{pro4-1}. The proof of Lemma \ref{lem4-3} is too long and is one of the crucial parts of this work. So we postpone the detailed proof of Lemma \ref{lem4-3} to the next section, to keep the proof of gradient H\"older continuity flowing smoothly. Furthermore, after we obtain Proposition \ref{pro4-1}, we can combine Proposition \ref{pro3-10} to infer the following result.

\begin{corollary}
\label{cor4-4}
Suppose $u$ fulfills the hypotheses of Proposition \ref{pro3-10}. Then we can find a small number $\delta$ such that one of the following conclusions must hold:
\begin{itemize}
\item There exists a nonnegative integer $N$ satisfying $\|\nabla u\|_{L^\infty(B_{2^{-n}})}\le (1-\delta)^{n}$ for $n=0,1,\cdots,N$.
Moreover, there is such a vector $e\in \mathbb{S}^{d-1}$ that
$$
\left|\nabla u(x)-(1-\delta)^Ne\right|\le\frac{(1-\delta)^N}{4}  \quad\text{for } x\in B_{2^{-N-1}}.
$$

\smallskip

\item For $\alpha_0=\log_2(1-\delta)^{-1}$ and $C_0=(1-\delta)^{-1}$, it holds that
$$
|\nabla u(x)|\le C_0|x|^{\alpha_0}  \quad\text{for } x\in B_1.
$$
\end{itemize}
\end{corollary}

\begin{proof}
Now for $L=\frac{1}{4}$, we may take small $r,\mu>0$ such that Proposition \ref{pro4-1} holds, and in turn we can use Proposition \ref{pro3-10} to fix $\delta>0$. Hence from Proposition \ref{pro3-10} we know that if the second statement does not hold in Corollary \ref{cor4-4}, then there holds for some positive integer $N$
$$
|\nabla u(x)|\le (1-\delta)^{n} \quad\text{for } x\in B_{2^{-n}}, \ n=0,1,\cdots,N,
$$
and a vector $e\in \mathbb{S}^{d-1}$ for which
$$
\left|\left\{x\in B_{2^{-N}}:|\nabla u(x)-(1-\delta)^Ne|\le r(1-\delta)^N\right\}\right|\ge(1-\mu)|B_{2^{-N}}|.
$$
At this time, we could apply Proposition \ref{pro4-1} to the scaled function $v$ in Corollary \ref{cor3-11}, and then rescale back to $u$ to get the desired result.
\end{proof}

Now remember the rescaling function $v(x)=(1-\delta)^{-N}2^Nu(2^{-N}x)$ defined in Corollary \ref{cor3-11}. From the first alternative of Corollary \ref{cor4-4}, we see that $v$ fulfills
$$
|\nabla v(x)-e|<\frac{1}{4} \quad\text{for } x\in B_\frac{1}{2},
$$
except that $v$ satisfies the conditions above along with (1)--(5) in Corollary \ref{cor3-11}. Therefore, due to the gradient approaching pointwise a unit vector $e$ in a ball, we can work in the scenario where the linearized kernel is mildly uniform elliptic, so that we may get regularity of the gradient (Proposition \ref{pro7.1}) by a known variational theory.

To this end, we now recall two results as the essential ingredients for demonstrating H\"older continuity of the gradient in the non-degenerate setting. The first lemma is derived in \cite{CS20}.

\begin{lemma}
\label{lem7.1.0}
Suppose that there are such two constants \(\mu,\lambda > 0\) that, for every ball \(B \subset \mathbb{R}^d\) and \(x \in B\), the kernel \(K\) fulfills
\begin{equation}
\label{7.1}
 \big|\{y \in B : K(x, y) \geq \lambda|x - y|^{-d-2\gamma}\}\big| \geq \mu|B|  \quad\text{for some } \gamma > 0.
\end{equation}
Then there is \(C>0\) depending only on \(\mu\) and \(d\) such that
\[
\int_{B_2} \int_{B_2} K(x, y)(w(x) - w(y))^2 \, dy \, dx \geq c\lambda[w]_{H^{\gamma}(B_1)}^2.
\]
\end{lemma}

The second one comes from \cite{KS14}.

\begin{lemma}
\label{lem7.1.00}
Let \(w : \mathbb{R}^d \to \mathbb{R}\) be a solution of
\[
\int_{\mathbb{R}^d} K(x, y)(w(x) - w(y)) \, dy = f(x) \quad \text{for } x \in B_1
\]
with \(f\) a bounded function in \(B_1\). Suppose that there exist \(\gamma > 0\), \(\Lambda > 1\) such that, for every \(x_0 \in B_1\), the following hypotheses on $K(x, y)$ are satisfied:
\begin{itemize}
\item[(i)] $K(x, y) = K(y, x)$;

\smallskip

\item[(ii)] $\rho^{-2} \int_{B_\rho(x_0)} |x_0 - y|^2 K(x_0, y) \, dy + \int_{B_\rho^c(x_0)} K(x_0, y) \, dy \leq \Lambda\rho^{-2\gamma}$;

\smallskip

\item[(iii)] Any time $B_\rho(x_0) \subset B_1$ and $w \in H^\gamma(B_\rho(x_0))$,
$$
 \int_{B_\rho(x_0)} \int_{B_\rho(x_0)} K(x, y)(w(x) - w(y))^2 \, dydx \geq \Lambda^{-1}[w]_{H^{\gamma}(B_{\rho/2}(x_0))}^2.
 $$
\end{itemize}
Then we can find two constants \(\alpha\) and \(C\), depending on \(\Lambda\), \(\gamma\) and \(d\), such that
\[
[w]_{C^{0,\alpha}(B_{1/2})} \leq C\bigl(\|w\|_{L^\infty(\mathbb{R}^d)} + \|f\|_{L^\infty(\mathbb{R}^d)}\bigr).
\]
\end{lemma}

By means of Lemma \ref{lem7.1.0}, the nondegeneracy of the gradient derived in Corollary \ref{cor4-4} can provide coercivity for the linearized kernel $K^a_u(x,y)$ in \eqref{ku}. Then through adapting the kernel properly, this enables us to employ Lemma \ref{lem7.1.00} to infer a H\"older modulus of continuity for gradients. The specific conclusion is stated as follows.

\begin{proposition}
\label{pro7.1}
Let the assumptions $(A_1)$--$(A_4)$ and $(H_1)$--$(H_3)$ be in force. Assume $u$ is a solution to \eqref{main} in $B_2$ with $\|\nabla u\|_{L^\infty(B_1)}\le1$ and $\mathrm{Tail}_{q-1,tq}(u;1)\le C_1$. If there is a vector $e\in \mathbb{S}^{d-1}$ satisfying
$$
|\nabla u-e|<\frac{1}{4} \quad \text{for } x\in B_1,
$$
then $\nabla u$ is locally H\"older continuous, that is, there exists a universal constant $\alpha>0$ such that
$$
\|\nabla u\|_{C^{0,\alpha}(B_{1/16})}\le C.
$$
Here $C>0$ is a universal constant depending also on $C_1$.
\end{proposition}

\begin{proof}
 Define a function
$$
w_\sigma(x)=\eta(12x)(\sigma\cdot\nabla u(x)),
$$
where the direction $\sigma\in \mathbb{S}^{d-1}$ and $\eta$ is a smooth cut-off function satisfying $\eta(12x)>0$ for $|x|<\frac{7}{48}$ as well as $\eta(12x)=1$ for $|x|\le\frac{1}{8}$. From Proposition \ref{lem2-1}, there holds
\begin{equation}
\label{7.1-1}
\left|\mathcal{L}^a_uw_\sigma(x)\right|\le C  \quad\text{for any } x\in B_\frac{1}{12}.
\end{equation}
In view of the condition $|\nabla u-e|<\frac{1}{4}$ in $B_1$, for any vector $|\tau|\le1$ with $|\tau\cdot e|>\frac{1}{2}$, we know that
\begin{equation}
\label{7.1-2}
\left|\nabla u(x)\cdot \tau\right|>\frac{1}{4}  \quad\text{for each } x\in B_1.
\end{equation}
This indicates that for all $x\in B_\frac{1}{12}$, the kernel $K^a_u(x,y)$ is uniformly elliptic for $y$ in the cone of directions $\mathcal{C}(x):=\left\{x+z\in B_\frac{1}{2}:|z\cdot e|>\frac{|z|}{2}\right\}$. That is, for $t\in \big(0,\frac{1}{2}\big)$ and $|\nu\cdot e|>\frac{1}{2}$, by virtue of \eqref{7.1-2}, $a(\cdot)\ge0$ and letting $y=x+t\nu$, we have
\begin{align}
\label{7.1-3}
K^a_u(x,x+t\nu)&=(p-1)\frac{|u(x)-u(x+t\nu)|^{p-2}}{t^{d+sp}}+(q-1)a(0,t\nu)\frac{|u(x)-u(x+t\nu)|^{q-2}}{t^{d+tq}} \nonumber\\
&\ge (p-1)\frac{\big(\frac{1}{4}t\big)^{p-2}}{t^{d+sp}}=ct^{-d+p(1-s)-2}.
\end{align}

In order to utilize Lemmas \ref{lem7.1.0} and \ref{lem7.1.00}, we have to modify the kernel and so the related equation. Now from \eqref{7.1-1} and the fact that $x+z\notin B_\frac{7}{48}$ for $x\in B_\frac{1}{12}$, $z\notin B_\frac{1}{2}$ and thus $w_\sigma(x+z)=0$, we can find
\begin{align}
\label{7.1-4}
-C-w_\sigma(x)\int_{B_\frac{1}{2}^c}K^a_u(x,x+z)\,dz&\le \int_{B_\frac{1}{2}}(w_\sigma(x)-w_\sigma(x+z))K^a_u(x,x+z)\,dz \nonumber\\
&\le C- w_\sigma(x)\int_{B_\frac{1}{2}^c}K^a_u(x,x+z)\,dz.
\end{align}

Recall the bounds on $a(\cdot), \|\nabla u\|_{L^{\infty}(B_1)}$ and $\mathrm{Tail}_{q-1,tq}(u;1)$ together with the relation $tq\le sp+1$. We derive
\begin{equation}
\label{7.1-5}
\left|w_\sigma(x)\int_{B_{1/2}^c}K^a_u(x,x+z)\,dz\right|\le C,
\end{equation}
which is similar to Lemma \ref{lem3-2}.

Introduce a new kernel
$$
K(x,x+z)=K^a_u(x,x+z)\chi_{B_\frac{1}{2}}(z)+\overline{K}(z)\chi_{B^c_\frac{1}{2}}(z),
$$
where
$$
\overline{K}(z)=c|z|^{-d+p(1-s)-2}
$$
with the $c$ identical to that in \eqref{7.1-3}. As a consequence, it yields that
$$
-C\le \int_{\mathbb{R}^d}(w_\sigma(x)-w_\sigma(x+z))K(x,x+z)\,dz\le C,
$$
which is the objective we will apply Lemma \ref{lem7.1.00} to.

Let us examine that $K$ meets the assumptions (i)-(iii) in Lemma \ref{lem7.1.00} properly scaled with $x_0\in B_\frac{1}{12}$ and $0<\rho<\frac{1}{12}$. First, the symmetry of $K$ in (i) is trivial, because we can see
 $$
K(x,y)=K^a_u(x,y)\chi_{B_\frac{1}{2}(x)}(y-x)+\overline{K}(y-x)\chi_{B^c_\frac{1}{2}(x)}(y-x).
$$
Second, owing to $B_{2\rho}(x_0)\subset B_\frac{1}{2}$, the boundedness of $a(\cdot)$ and the local Lipschitz property of $u$, it holds
\begin{align*}
\rho^{-2}\int_{B_\rho}|z|^2K(x_0,x_0+z)\,dz&=\rho^{-2}\int_{B_\rho}|z|^2K_u^a(x_0,x_0+z)\,dz\\
&\le \rho^{-2}\int_{B_\rho}|z|^2\left(\frac{|z|^{p-2}}{|z|^{d+sp}}+\|a\|_\infty\frac{|z|^{q-2}}{|z|^{d+tq}}\right)\,dz\\
&\le C(1+\|a\|_\infty)\rho^{-2}\big(\rho^{p(1-s)}+\rho^{q(1-t)}\big)\\
&\le C\rho^{p(1-s)-2},
\end{align*}
where in the last inequality we have used $\rho<1$ and the condition $(H_3)$ implying $(1-s)p\le (1-t)q$. On the other hand, noticing again $\rho<1$, $(1-s)p\le (1-t)q$ and $p(1-s)-2<0$, $q(1-t)-2<0$, we arrive at
\begin{align*}
\int_{B^c_\rho}K(x_0,x_0+z)\,dz&=\int_{B_\frac{1}{2}\setminus B_\rho}K(x_0,x_0+z)\,dz+\int_{B^c_\frac{1}{2}}K(x_0,x_0+z)\,dz\\
&= \int_{B_\frac{1}{2}\setminus B_\rho}K^a_u(x_0,x_0+z)\,dz+\int_{B^c_\frac{1}{2}}\overline{K}(z)\,dz\\
&\le \int_{B_\frac{1}{2}\setminus B_\rho}\left(\frac{|z|^{p-2}}{|z|^{d+sp}}+\|a\|_\infty\frac{|z|^{q-2}}{|z|^{d+tq}}\right)\,dz+C\\
&\le C(1+\|a\|_\infty)\big(\rho^{p(1-s)-2}+\rho^{q(1-t)-2}\big)+C\\
&\le C\rho^{p(1-s)-2}.
\end{align*}
That is, the requirement (ii) in Lemma \ref{lem7.1.00} has been verified. Finally, the condition (iii) is proved by the hypothesis \eqref{7.1} of Lemma \ref{lem7.1.0}. Furthermore, the latter assumption is a consequence of the lower bound \eqref{7.1-3}. Specifically, when $\big|B\cap B_\frac{1}{2}\big|\le \frac{9}{10}|B|$, then according to the choice of $\overline{K}$, \eqref{7.1} holds for $\mu=\frac{1}{10}$; when $\big|B\cap B_\frac{1}{2}\big|>\frac{9}{10}|B|$, it follows from \eqref{7.1-3} that for $x\in B$
$$
|\big\{y\in B\cap B_\frac{1}{2}:K(x,y)\ge c|x-y|^{-d+p(1-s)-2}\big\}|\ge\big|\mathcal{C}(x)\cap B\cap B_\frac{1}{2}\big|\ge c_1|B|
$$
with $c_1>0$ depending only on $d$. Hence, we apply Lemma \ref{lem7.1.0} with $\mu=\min\big\{c_1,\frac{1}{10}\big\}$ to obtain the condition (iii).

At this moment, we could make use of Lemma \ref{lem7.1.00} to deduce the H\"older continuity of $w_\sigma$, i.e.,
$$
[w_\sigma]_{C^{0,\alpha}(B_{1/16})}\le C\big(\|w_\sigma\|_{L^\infty(\mathbb{R}^d)}+C\big)\le C.
$$
Due to the arbitrariness of the unit vector $\sigma$, we justify the claim.
\end{proof}

Next, combining Corollary \ref{cor4-4} with Proposition \ref{pro7.1} allows us to conclude the gradient H\"older continuity for solutions to Eq. \eqref{main}. We end this section by giving the proof of Theorem \ref{thm1}.

\begin{proof}[\textbf{Proof of Theorem \ref{main}}] Let $u$ be a solution to Eq. \eqref{main} in $B_2$. In this paper, we have assumed $u$ is $C^1$-regular. By translation and scaling, introducing the function
$$
v(x)=2^{-1}\left(\|u\|_{{\rm lip} (B_1)}+\mathrm{Tail}_{q-1,tq}(u;1)\right)^{-1}(u(x)-u(0)),
$$
we can naturally suppose that $u(0)=0$ and $\|u\|_{{\rm lip}(B_1)}+\mathrm{Tail}_{q-1,tq}(u;1)\le 1$. Here we need note the tail related to the $p$-growth, $\mathrm{Tail}_{p-1,sp+1}(u;1)$, could be controlled by $\mathrm{Tail}_{q-1,tq}(u;1)$ from Lemma \ref{lem2-0} (3), so the scaled function does not include the former tail. Repeating the analysis below for $v$, we obtain $\|v\|_{C^{1,\alpha}(B_1)}\le C$ and further have
$$
\|u\|_{C^{1,\alpha}(B_1)}\le C\left(\|u\|_{{\rm lip}(B_1)}+\mathrm{Tail}_{q-1,tq}(u;1)\right).
$$

Now we perform a rescaling to $u$ so that the hypotheses of Lemma \ref{lem3-0} are satisfied by the rescaled function. By reviewing Lemma \ref{lem3-9} and Corollary \ref{cor3-11}, this can be achieved if we consider $\overline{u}(x):=2^{K_0}(1-\delta)^{-K_0}u(2^{-K_0}x)$. For convenience, we still denote the rescaled function $\overline{u}(x)$ as $u(x)$.

At this point, we may apply Corollary \ref{cor4-4} to infer that either
\begin{equation*}
|\nabla u(x)|\le C|x|^{\alpha_0} \quad\text{for } x\in B_\frac{1}{2}
\end{equation*}
with $\alpha_0=\log_2(1-\delta)^{-1}$, or there is an integer $N>0$ such that
\begin{equation}
\label{t2}
|\nabla u(x)|\le (1-\delta)^n \quad\text{for } x\in B_{2^{-n}} \ \ n=0,1,\cdots,N,
\end{equation}
and
$$
\left|\left\{x\in B_{2^{-N}}:|\nabla u(x)-(1-\delta)^Ne|\le r(1-\delta)^N\right\}\right|\ge(1-\mu)|B_{2^{-N}}| \quad\text{for some } e\in\mathbb{S}^{d-1}.
$$
Then we proceed to apply Proposition \ref{pro7.1} to the function
$$
\widetilde{u}(x)=2^{N}(1-\delta)^{-N}u(2^{-N}x),
$$
and get $[\widetilde{u}]_{C^{1,\alpha}(B_{1/16})}\le C$. Rescaling back and taking $\alpha<\alpha_0$, we arrive at
\begin{equation}
\label{t3}
[\nabla{u}]_{C^{0,\alpha}(B_{2^{-N-4}})}\le C.
\end{equation}
Observe that we can translate this argument to every point $x\in B_\frac{1}{2} $, gaining the same dichotomy with $N = N(x)$. To end the proof, it remains to patch these estimates together in a way that is independent of the point $x$.

Take any other point $y \in B_{1/2}$. We shall prove that for some constant $C$ independent of $N$,
\[
|\nabla u(x) - \nabla u(y)| \leq C|x - y|^\alpha.
\]
There are two possibilities that either $|y-x|<2^{-N-4}$ or $2^{-N-4}\le |y-x|<1$. In the first case, we use \eqref{t3} to get the conclusion immediately. In the second case, we select $n \in \{0,1,\dots,N+3\}$ so that $2^{-n-1} \leq |x - y| < 2^{-n}$. If $n \leq N$, we employ \eqref{t2} to have
$$
|\nabla u(y) - \nabla u(x)| \leq 2(1 - \delta)^n \leq 2(1 - \delta)^{-1}|x - y|^{\alpha_0}.
$$
Finally, for $n \in \{N+1,N+2,N+3\}$, we still have $y \in B_{2^{-N}}(x)$ and get
$$
|\nabla u(y) - \nabla u(x)| \leq 2(1 - \delta)^N \leq 2(1 - \delta)^{-4}|x - y|^{\alpha_0}.
$$
 The proof is finished in all cases.
\end{proof}

\section{Proof of Lemma \ref{lem4-3}}
\label{sec5}

In this section, we are going to exploit the well-known Ishii-Lions methods to complete the proof of Lemma \ref{lem4-3}. The equation \eqref{main} is nondegenerate in the direction of $\nabla u$, which will be close to a unit vector $e$, and will be unrelated to the derivative of the modulus of continuity as in a standard application of this method. This discrepancy with the analysis in Sections \ref{sec6} and \ref{sec7} affects the calculation of several estimates in this part. In contrast to the fractional $p$-Laplace equation, the present setting involves nonstandard growth, mixed orders of differentiability as well as a variable coefficient $a(\cdot)$, which requires a careful balance of their subtle interplay.


Our aim is to employ the Ishii-Lions argument to infer the Lipschitz estimate for the function $u(x)-e\cdot x$ in $B_\frac{1}{2}$, more precisely, to get
$$
[u-e\cdot x]_{\rm lip {(B_{1/2})}}\le L, \quad\text{i.e., }|\nabla u(x)-e|\le L \quad\text{ for } x\in B_\frac{1}{2}.
$$

Let the function $u$ satisfy the conditions of Lemma \ref{lem4-3}. Denote
$$
\Psi(x,y)=u(x)-u(y)-e\cdot(x-y)-L\omega(|x-y|)-\kappa\epsilon_0\psi(x)
$$
and
$$
\phi(x,y)=e\cdot(x-y)+L\omega(|x-y|)+\kappa\epsilon_0\psi(x).
$$
The number $\kappa>0$ is to be determined soon, and the function $\omega$ is nonnegative, strictly concave and Lipschitz continuous, which is defined as
\begin{equation*}
\omega(\rho)=\rho+\frac{\rho}{20\log(\rho/4)}.
\end{equation*}
Through simple calculations, we know that $\omega'(\rho)\in(\frac{1}{4},1)$ for $\rho\in(0,2]$, and
\begin{equation*}
-\frac{C}{\rho\log^2\rho}\le\omega''(\rho)\le -\frac{c}{\rho\log^2\rho} \quad\text{for } \rho\in\Big(0,\frac{1}{2}\Big]
\end{equation*}
with $c,C>0$ being universal constants. Let $\psi_0(x)\ge0$ be a smooth function that is equal to 0 in $\overline{B}_\frac{1}{2}$ and is strictly positive in $\overline{B}_1\setminus \overline{B}_\frac{1}{2}$. We demand that the function $\psi(x)$ fulfills $\psi(x)=\psi_0(x)^m$ and
\begin{equation*}
|\nabla \psi(x)|\le C\psi(x)^\frac{m-1}{m} \quad\text{ for }  x\in B_1,
\end{equation*}
where $C=m\max{|\nabla\psi_0(x)|}$ and $m>2$ is a large integer to be selected later. Specifically, the function $(|x|^2-\frac{1}{4})_+$ can be taken as an example of $\psi_0(x)$, so $[(|x|^2-\frac{1}{4})_+]^m$ provides a model for $\psi(x)$.

In what follows, we are going to show that $u-e\cdot x$ has $L\omega$ as a modulus of continuity in $B_\frac{1}{2}$. To this end, we verify $\Psi(x,y)\le0$ for all $x, y\in B_1$. Then this indicates the desired modulus of continuity in $B_\frac{1}{2}$, because $\psi$ is zero in $B_\frac{1}{2}$. This thrives on contradiction. Suppose there are $\overline{x},\overline{y}\in B_1$ such that
\begin{equation}
\label{4-8}
\Psi(\overline{x},\overline{y})=\sup_{(x,y)\in B_1\times B_1}\Psi(x,y)>0,
\end{equation}
that is,
$$
u(\overline{x})-u(\overline{y})-e\cdot(\overline{x}-\overline{y})>L\omega(|\overline{x}-\overline{y}|)+\kappa\epsilon_0\psi(\overline{x}).
$$
It is known that the left-hand side is bounded by $\epsilon_0$. Thus we could take $\kappa$ so large that $\overline{x}$ is in $B_\frac{5}{8}$ by the second term at the right-hand side. Let
$$
\overline{a}:=\overline{x}-\overline{y}.
 $$
With the help of $L\omega(|\overline{x}-\overline{y}|)\le \epsilon_0$, we can make $\overline{a}$ arbitrarily small by selecting $\epsilon_0$ small enough (depending on $L$), so that $\overline{x},\overline{y}$ both belong to $B_\frac{3}{4}$. This observation is helpful for the processes below.

Given $\Omega\subset \mathbb{R}^d$, we define
\begin{align}
\label{Lop}
\mathcal{L}[\Omega]u(x)&:=\int_\Omega\frac{J_p(u(x)-u(x+z))}{|z|^{d+sp}}\,dz+\int_\Omega a(x,x+z)\frac{J_q(u(x)-u(x+z))}{|z|^{d+tq}}\,dz \notag\\
&=\int_\Omega\frac{J_p(u(x)-u(x+z))}{|z|^{d+sp}}\,dz+\int_\Omega a(0,z)\frac{J_q(u(x)-u(x+z))}{|z|^{d+tq}}\,dz  \notag\\
&=:\mathcal{L}_p[\Omega]u(x)+\mathcal{L}_q[\Omega]u(x),
\end{align}
where we have used the translation invariance of $a(\cdot)$. Since $u$ is a solution to Eq. \eqref{main}, we have
$$
\mathcal{L}u(\overline{x})-\mathcal{L}u(\overline{y})=0
$$
and further
\begin{align}
\label{4-9}
0&=\mathcal{L}[\mathcal{C}]u(\overline{x})-\mathcal{L}[\mathcal{C}]u(\overline{y})+\mathcal{L}[\mathcal{D}]
u(\overline{x})-\mathcal{L}[\mathcal{D}]u(\overline{y}) \nonumber\\
&\quad+\mathcal{L}[\mathcal{E}]u(\overline{x})-\mathcal{L}[\mathcal{E}]u(\overline{y})+
\mathcal{L}\big[B^c_\frac{1}{16}\big]u(\overline{x})-\mathcal{L}\big[B^c_\frac{1}{16}\big]u(\overline{y}) \nonumber\\
&=:T_1+T_2+T_3+T_4,
\end{align}
where
$$
\mathcal{C}:=\mathcal{C}(\overline{a})=\left\{z\in B_\frac{|\overline{a}|}{2}:|\langle\overline{a},z\rangle|\ge\sqrt{1-\delta_0^2(\overline{a})}|\overline{a}||z| \right\}
$$
and
$$
\mathcal{D}=B_{\delta_1(\overline{a})|\overline{a}|}\cap \mathcal{C}^c, \quad \mathcal{E}=B_\frac{1}{16}\setminus(\mathcal{D}\cup\mathcal{C})
$$
with
$$
\delta_0(\overline{a})=\frac{c_0}{|\log|\overline{a}||} \quad\text{and}\quad \delta_1(\overline{a})=\frac{c_1}{|\log|\overline{a}||^\nu}.
$$
Here $c_0,c_1,\nu$ are positive constants to be chosen. In this section, we will use the notation $\delta^2u$ many times that represents the second increment of $u$ as
$$
\delta^2u(x,z)=u(x)+\nabla u(x)\cdot z-u(x+z).
$$

Now let us first state the lemma below, coming from \cite{Sil25}, which give the estimates on the second increment of $\phi$ in the cone $\mathcal{C}(a)$.


\begin{lemma}
\label{lem5-2}
Let \( x, y \in B_{3/4} \) and \( z \in \mathcal{C}(a) \) with \( |a| = |x-y| \) sufficiently small. Then the forthcoming estimates hold for \( z \in \mathcal{C}(a) \),
\[
\begin{aligned}
\frac{cL}{|a| \log^2 |a|} |z|^2 &\leq \delta^2 \phi(x, \cdot)(y, z) \leq \frac{CL}{|a| \log^2 |a|} |z|^2, \\
\frac{cL}{|a| \log^2 |a|} |z|^2 &\leq \delta^2 \phi(\cdot, y)(x, z) \leq \frac{CL}{|a| \log^2 |a|} |z|^2,
\end{aligned}
\]
where \( c \) and \( C \) are universal constants.
\end{lemma}

In addition, the gross upper bound on $\delta^2 \phi(x, \cdot)(y, z)$,
\begin{equation}
\label{6.14}
|\delta^2 \phi(x, \cdot)(y, z)|\le CL\frac{|z|^2}{|a|} \quad\text{for } z\in B_\frac{|a|}{2},
\end{equation}
will be utilized several times later. Meanwhile, $|\delta^2 \phi(\cdot, y)(x, z)|$ has the same crude upper bound for $z\in B_\frac{|a|}{2}$.

Next, according to the basic assumption \eqref{4-8}, we aim to get a contradiction from \eqref{4-9}. We are going to exactly estimate the lower bounds on the terms $T_1$--$T_4$ in \eqref{4-9}, which could imply the desired result. Now let us consider $T_1$ on $\mathcal{C}$.

\begin{lemma}[Estimate on $\mathcal{C}$]
\label{lem4-6}
Let the hypotheses $(A_1),(A_2)$ and $(A_4)$ on the coefficient $a(\cdot)$ hold. For each $L\in\big(0,\frac{1}{4}\big]$, there is $\epsilon_0>0$ small enough and a constant $C>0$ such that if $\overline{x},\overline{y}$ satisfy \eqref{4-8}, then
\begin{align*}
T_1\ge CL\delta_0^{d+p-1}|\overline{a}|^{p(1-s)-1}+a(0,0)CL\delta_0^{d+q-1}|\overline{a}|^{q(1-t)-1}-CL[a]_{\rm lip}\delta_0^{d+q-1}|\overline{a}|^{q(1-t)},
\end{align*}
where $\delta_0:=\delta_0(\overline{a})$ and $C$ depends only on $d,s,t,p,q$.
\end{lemma}

\begin{proof}
In order to get the lower bound on $T_1$, we will separately estimate $\mathcal{L}[\mathcal{C}]u(\overline{x})$ and $-\mathcal{L}[\mathcal{C}]u(\overline{y})$, but we just deal with the term $\mathcal{L}[\mathcal{C}]u(\overline{x})$ because of the similar analysis for the latter. Denote $l(z)=-\nabla_x\phi(\overline{x},\overline{y})\cdot z$. Notice that, by the symmetry and translation-invariance of $a(\cdot)$,
\begin{align*}
a(0,-z)J_q(\nabla_x\phi(\overline{x},\overline{y})\cdot (-z))&=-a(z,0)|\nabla_x\phi(\overline{x},\overline{y})\cdot z|^{q-2}\nabla_x\phi(\overline{x},\overline{y})\cdot z\\
&=-a(0,z)J_q(\nabla_x\phi(\overline{x},\overline{y})\cdot z).
\end{align*}
Since the set $\mathcal{C}$ is symmetric, we can discover
\begin{align*}
\mathcal{L}_q[\mathcal{C}]l(\overline{x})&=\int_\mathcal{C}a(0,z)J_q\big((-\nabla_x\phi(\overline{x},\overline{y})\cdot \overline{x})-(-\nabla_x\phi(\overline{x},\overline{y})\cdot (\overline{x}+z))\big)|z|^{-d-tq}\,dz\\
&=\int_\mathcal{C}a(0,z)J_q(\nabla_x\phi(\overline{x},\overline{y})\cdot z)|z|^{-d-tq}\,dz=0.
\end{align*}
Recall that $(\overline{x},\overline{y})$ is the maximum point of $\Psi$ in \eqref{4-8}. We have
$$
u(\overline{x})-u(\overline{x}+z)\ge\phi(\overline{x},\overline{y})-\phi(\overline{x}+z,\overline{y}) \quad\text{for } z\in\mathcal{C}.
$$
Then via the monotonicity of $J_q$ and Lemma \ref{lem2-0-0},
\begin{align*}
\mathcal{L}_q[\mathcal{C}]u(\overline{x})&\ge\mathcal{L}_q[\mathcal{C}]\phi(\cdot,\overline{y})(\overline{x})\\
&=\mathcal{L}_q[\mathcal{C}]\phi(\cdot,\overline{y})(\overline{x})-\mathcal{L}_q[\mathcal{C}]l(\overline{x})\\
&=(q-1)\int_\mathcal{C}\int^1_0a(0,z)|l(z)+\tau\delta^2\phi(\cdot,\overline{y})(\overline{x},z)|^{q-2}\delta^2\phi(\cdot,\overline{y})
(\overline{x},z)\,d\tau|z|^{-d-tq}\,dz.
\end{align*}

Now define a cone
$$
\widetilde{\mathcal{C}}:=\left\{z\in\mathbb{R}^d:|\langle\nabla_x\phi(\overline{x},\overline{y}),z\rangle|\ge\frac{\delta_0}{4}|z|\right\}.
$$
Via the definition of $\phi(x)$, we have
\[
\nabla_x \phi(x, y) = e + L \omega'(|x - y|) \frac{x - y}{|x - y|} + \kappa \epsilon_0 \nabla \psi(x).
\]
Notice the value of \( L \) is chosen in the interval \((0, 1/4]\). Hence, taking \( \epsilon_0 \) small, we can guarantee that \( |\nabla_x \phi(x, y)| \in (1/2, 3/2) \). Therefore, the cone \( \widetilde{\mathcal{C}}\) is a set whose complement has width at most \( \delta_0/2 \). It must intersect with no less than half of \( \mathcal{C} \). 
Based on the lower bound on $\delta^2\phi(\cdot,\overline{y})(\overline{x},z)$ for $z\in\mathcal{C}$ in Lemma \ref{lem5-2}, we obtain
\begin{align*}
\mathcal{L}_q[\mathcal{C}]u(\overline{x})&\ge C\int_{\mathcal{C}\cap\widetilde{\mathcal{C}}}a(0,z)|l(z)|^{q-2}\delta^2\phi(\cdot,\overline{y})(\overline{x},z)|z|^{-d-tq}\,dz\\
&\ge C\int_{\mathcal{C}\cap\widetilde{\mathcal{C}}}a(0,z)\left(\frac{\delta_0}{4}|z|\right)^{q-2}\delta^2\phi(\cdot,\overline{y})(\overline{x},z)|z|^{-d-tq}\,dz\\
&\ge \frac{CL\delta_0^{q-2}}{|\overline{a}|\log^2|\overline{a}|}\int_{\mathcal{C}\cap\widetilde{\mathcal{C}}}\frac{a(0,z)}{|z|^{d+q(t-1)}}\,dz,
\end{align*}
where in the first line we have utilized Lemma \ref{lem2-0-0}. Let us focus on the last integral. Applying the Lipschitz continuity of $a(\cdot)$, we find
\begin{align*}
\int_{\mathcal{C}\cap\widetilde{\mathcal{C}}}\frac{a(0,z)}{|z|^{d+q(t-1)}}\,dz&=\int_{\mathcal{C}\cap\widetilde{\mathcal{C}}}\frac{a(0,0)
+a(0,z)-a(0,0)}{|z|^{d+q(t-1)}}\,dz\\
&\ge a(0,0)\int_{\mathcal{C}\cap\widetilde{\mathcal{C}}}\frac{dz}{|z|^{d+q(t-1)}}-[a]_{\rm lip}\int_{\mathcal{C}\cap\widetilde{\mathcal{C}}}\frac{dz}{|z|^{d+q(t-1)-1}}.
\end{align*}

First, it yields that
\begin{align*}
\int_{\mathcal{C}\cap\widetilde{\mathcal{C}}}\frac{dz}{|z|^{d+q(t-1)}}&=\frac{|\mathcal{C}\cap\widetilde{\mathcal{C}}|}
{\big|B_\frac{|\overline{a}|}{2}\big|}\int_{B_\frac{|\overline{a}|}{2}}\frac{dz}{|z|^{d+q(t-1)}}\\
&=\left(\frac{|\overline{a}|}{2}\right)^{q(1-t)}\frac{|\mathcal{C}\cap\widetilde{\mathcal{C}}|}
{\big|B_\frac{|\overline{a}|}{2}\big|}\int_{B_1}\frac{dz}{|z|^{d+q(t-1)}}\\
&\ge C\delta_0^{d-1}|\overline{a}|^{q(1-t)}.
\end{align*}
See also \cite[Example 1]{Bar12} for this computation. Moreover, we can evaluate
\begin{align*}
\int_{\mathcal{C}\cap\widetilde{\mathcal{C}}}\frac{dz}{|z|^{d+q(t-1)-1}}&\le \int_{\mathcal{C}}\frac{dz}{|z|^{d+q(t-1)-1}}\\
&=\frac{|\mathcal{C}|}{\big|B_\frac{|\overline{a}|}{2}\big|}\int_{B_\frac{|\overline{a}|}{2}}\frac{dz}{|z|^{d+q(t-1)-1}}\\
&=C\frac{|\mathcal{C}_1|}{|B_1|}|\overline{a}|^{q(1-t)+1}\int_{B_1}\frac{dz}{|z|^{d+q(t-1)-1}}\\
&\le C\delta_0^{d-1}|\overline{a}|^{q(1-t)+1},
\end{align*}
where $\mathcal{C}_1$ stands for the cone $\mathcal{C}$ with radius 1, and we have exploited the fact that the aperture $\theta$ of the cone $\mathcal{C}$ satisfies $\theta\approx2\delta_0$. At this time, we arrive at
$$
\mathcal{L}_q[\mathcal{C}]u(\overline{x})\ge a(0,0)CL\delta_0^{d+q-1}|\overline{a}|^{q(1-t)-1}-[a]_{\rm lip}CL\delta_0^{d+q-1}|\overline{a}|^{q(1-t)}
$$
with $C>0$ depending only on $d,t,q$. Analogously and more easily, we deduce
$$
\mathcal{L}_p[\mathcal{C}]u(\overline{x})\ge CL\delta_0^{d+p-1}|\overline{a}|^{p(1-s)-1}
$$
with $C>0$ depending only on $d,s,p$. A similar estimate holds for $-\mathcal{L}[\mathcal{C}]u(\overline{y})$. Thus the desired result is obtained.
\end{proof}

\begin{lemma}[Estimate on $\mathcal{D}$]
\label{lem4-7}
Let the point $(\overline{x},\overline{y})$ satisfy \eqref{4-8}, and let also the hypotheses $(A_1),(A_2)$ and $(A_4)$ on the coefficient $a(\cdot)$ hold. Then for each $L\in\big(0,\frac{1}{4}\big]$ and $\epsilon_0>0$ small enough, there holds that
\begin{align*}
T_2\ge -CL\delta_1^{p(1-s)}|\overline{a}|^{p(1-s)-1}-a(0,0)CL\delta_1^{q(1-t)}|\overline{a}|^{q(1-t)-1}-CL[a]_{\rm lip}\delta_1^{q(1-t)+1}|\overline{a}|^{q(1-t)}.
\end{align*}
Here $\delta_1:=\delta_1(\overline{a})$ and $C>0$ depends only on $d,s,t,p,q$.
\end{lemma}

\begin{proof}
Here we make use of the notations above. In order to get the lower bound on $\mathcal{L}[\mathcal{D}]u(\overline{x})$, we evaluate the lower bounds on $\mathcal{L}_p[\mathcal{D}]u(\overline{x})$ and $\mathcal{L}_q[\mathcal{D}]u(\overline{x})$ respectively. As the calculations in the proof of Lemma \ref{lem4-6}, we also derive
\begin{align*}
\mathcal{L}_q[\mathcal{D}]u(\overline{x})\ge(q-1)\int_\mathcal{D}\int^1_0a(0,z)|l(z)+\tau\delta^2\phi(\cdot,\overline{y})
(\overline{x},z)|^{q-2}\delta^2\phi(\cdot,\overline{y})(\overline{x},z)\,d\tau\,|z|^{-d-tq}\,dz.
\end{align*}
For small $\epsilon_0$ and $\tau\in[0,1]$, we note
$$
|l(z)+\tau\delta^2\phi(\cdot,\overline{y})(\overline{x},z)|\le |l(z)|+|\phi(\overline{x},\overline{y})+\nabla_x\phi(\overline{x},\overline{y})\cdot z-\phi(\overline{x}+z,\overline{y})|\le 2|z|
$$
with $l(z)=-\nabla_x\phi(\overline{x},\overline{y})\cdot z$.

Combining the previous inequalities with the crude upper bound \eqref{6.14}, we infer
\begin{align*}
\mathcal{L}_q[\mathcal{D}]u(\overline{x})&\ge -CL\int_{\mathcal{D}}\frac{a(0,z)}{|\overline{a}|}|z|^{q-d-tq}\,dz\\
&\ge -\frac{CL}{|\overline{a}|}\left(a(0,0)\int_\mathcal{D}\frac{dz}{|z|^{d+q(t-1)}}+[a]_{\rm lip}\int_\mathcal{D}\frac{dz}{|z|^{d+q(t-1)-1}}\right)\\
&\ge -\frac{CL}{|\overline{a}|}\left(a(0,0)\int_{B_{\delta_1|\overline{a}|}}\frac{dz}{|z|^{d+q(t-1)}}+[a]_{\rm lip}\int_{B_{\delta_1|\overline{a}|}}\frac{dz}{|z|^{d+q(t-1)-1}}\right)\\
&\ge -a(0,0)CL\delta_1^{q(1-t)}|\overline{a}|^{q(1-t)-1}-[a]_{\rm lip}CL\delta_1^{q(1-t)+1}|\overline{a}|^{q(1-t)},
\end{align*}
where $C>0$ depends on $d,t,q$. Similarly,
$$
\mathcal{L}_p[\mathcal{D}]u(\overline{x})\ge -CL\delta_1^{p(1-s)}|\overline{a}|^{p(1-s)-1}
$$
with $C>0$ depends on $d,s,p$. An analogous estimate is true for $-\mathcal{L}[\mathcal{D}]u(\overline{y})$, and then the desired result holds.
\end{proof}

\begin{lemma}[Estimate on $\mathcal{E}$]
\label{lem4-8}
Let the point $(\overline{x},\overline{y})$ satisfy \eqref{4-8}, and let also $p\in\big[2,\frac{2}{1-s}\big),q\in\big[2,\frac{1}{1-t}\big)$ and the condition $(A_4)$ on the coefficient $a(\cdot)$ hold.
Then for every $L\in\big(0,\frac{1}{4}\big]$, we can find a positive constant $C$, depending only on $d,s,t,p,q$ and the choice of $\nu$ in $\delta_1$, such that
\begin{align*}
T_3\ge -C\left(|\overline{a}|^{\frac{p(1-s)}{2}}+a(0,0)|\overline{a}|^{\frac{q(1-t)}{2}}+[a]_{\rm lip}
|\overline{a}|^{\frac{q(1-t)+1}{2}}\right).
\end{align*}
\end{lemma}

\begin{proof}
Since $|\overline{a}|$ is very small, we can have $\delta_1|\overline{a}|<\sqrt{|\overline{a}|}<\frac{1}{16}$, $\frac{|\overline{a}|}{2}<\sqrt{|\overline{a}|}$, and so
$$
\mathcal{E}=\Big(B_\frac{1}{16}\setminus B_{\sqrt{|\overline{a}|}}\Big)\cup\Big(\mathcal{E}\cap B_{\sqrt{|\overline{a}|}}\Big).
$$
We split $T_3$ as follows,
\begin{align*}
T_3&= \mathcal{L}_p[\mathcal{E}]u(\overline{x})-\mathcal{L}_p[\mathcal{E}]u(\overline{y})+\mathcal{L}_q[\mathcal{E}]u(\overline{x})-\mathcal{L}_q
[\mathcal{E}]u(\overline{y})\\
&=\mathcal{L}_p\Big[\mathcal{E}\cap B_{\sqrt{|\overline{a}|}}\Big]u(\overline{x})-\mathcal{L}_p\Big[\mathcal{E}\cap B_{\sqrt{|\overline{a}|}}\Big]u(\overline{y})\\
&\quad+\mathcal{L}_p\Big[B_\frac{1}{16}\setminus B_{\sqrt{|\overline{a}|}}\Big]u(\overline{x})-\mathcal{L}_p\Big[B_\frac{1}{16}\setminus B_{\sqrt{|\overline{a}|}}\Big]u(\overline{y})\\
&\quad+\mathcal{L}_q\Big[\mathcal{E}\cap B_{\sqrt{|\overline{a}|}}\Big]u(\overline{x})-\mathcal{L}_q\Big[\mathcal{E}\cap B_{\sqrt{|\overline{a}|}}\Big]u(\overline{y})\\
&\quad+\mathcal{L}_q\Big[B_\frac{1}{16}\setminus B_{\sqrt{|\overline{a}|}}\Big]u(\overline{x})-\mathcal{L}_q\Big[B_\frac{1}{16}\setminus B_{\sqrt{|\overline{a}|}}\Big]u(\overline{y})\\
&=:T_{31}+T_{32}+T_{33}+T_{34}.
\end{align*}
Let us first consider the term $T_{34}$. Applying Lemma \ref{lem2-0-0} derives
\begin{align}
\label{4-8-1}
T_{34}&=(q-1)\int_{B_\frac{1}{16}\setminus B_{\sqrt{|\overline{a}|}}}\int_0^1a(0,z)|\delta^1u(\overline{y},z)+\tau(\delta^1u(\overline{x},z)-\delta^1u(\overline{y},z))|^{q-2} \nonumber\\
&\qquad\qquad\qquad\qquad\qquad\times(\delta^1u(\overline{x},z)-\delta^1u(\overline{y},z))\,d\tau\,|z|^{-d-tq}\,dz.
\end{align}
Here $\delta^1u(x,z)=u(x)-u(x+z)$. Notice the fact $[u]_{{\rm lip}(B_1)}\le1$. We get
$$
|\delta^1u(\overline{x},z)-\delta^1u(\overline{y},z)|\le2|\overline{a}|
$$
and for $\tau\in[0,1]$
\begin{align}
\label{4-8-2}
|\delta^1u(\overline{y},z)+\tau(\delta^1u(\overline{x},z)-\delta^1u(\overline{y},z))|\le3|z|.
\end{align}
Merging the three displays above, Lipschitz continuity of $a(\cdot)$ and $q<\frac{1}{1-t}$, we infer
\begin{align*}
T_{34}&\ge-C|\overline{a}|\int_{B_\frac{1}{16}\setminus B_{\sqrt{|\overline{a}|}}}a(0,z)|z|^{q-2-d-tq}\,dz\\
&\ge -C|\overline{a}|\left(a(0,0)\int_{B_\frac{1}{16}\setminus B_{\sqrt{|\overline{a}|}}}|z|^{q-2-d-tq}\,dz+[a]_{\rm lip}
\int_{B_\frac{1}{16}\setminus B_{\sqrt{|\overline{a}|}}}|z|^{q-1-d-tq}\,dz\right)\\
&=-C|\overline{a}|\left(a(0,0)\int^\frac{1}{16}_{\sqrt{|\overline{a}|}}r^{q-3-tq}\,dr+[a]_{\rm lip}
\int^\frac{1}{16}_{\sqrt{|\overline{a}|}}r^{q-2-tq}\,dr\right)\\
&\ge -Ca(0,0)\frac{|\overline{a}|^{1+\frac{q(1-t)-2}{2}}}{q(t-1)+2}-C[a]_{\rm lip}\frac{|\overline{a}|^{1+\frac{q(1-t)-1}{2}}}{q(t-1)+1}\\
&=-Ca(0,0)|\overline{a}|^{\frac{q(1-t)}{2}}-C[a]_{\rm lip}|\overline{a}|^{\frac{q(1-t)+1}{2}}.
\end{align*}
For $T_{33}$, by Lemma \ref{lem2-0-0} we also obtain the display \eqref{4-8-1} with the integration domain $B_\frac{1}{16}\setminus B_{\sqrt{|\overline{a}|}}$ replaced by $\mathcal{E}\cap B_{\sqrt{|\overline{a}|}}$. Owing to $\Psi$ attaining the maximum at $(\overline{x},\overline{y})$, we discover
$$
\delta^1u(\overline{x},z)-\delta^1u(\overline{y},z)\ge -\kappa\epsilon_0\delta^1\psi(\overline{x},z),
$$
which together with \eqref{4-8-2} indicates
$$
T_{33}\ge -C\epsilon_0\int_{\mathcal{E}\cap B_{\sqrt{|\overline{a}|}}}a(0,z)\delta^1\psi(\overline{x},z)|z|^{q-2-d-tq}\,dz.
$$
From the smooth property of $\psi$ and Taylor expansion, there holds that
$$
|\delta^1\psi(\overline{x},z)|\le C(|\nabla \psi(\overline{x})||z|+|z|^2).
$$
Thereby, it yields that
\begin{align*}
T_{33}&\ge-C\epsilon_0\left(\int_{\mathcal{E}\cap B_{\sqrt{|\overline{a}|}}}\frac{a(0,z)}{|z|^{d+q(t-1)}}\,dz+|\nabla \psi(\overline{x})|\int_{\mathcal{E}\cap B_{\sqrt{|\overline{a}|}}}\frac{a(0,z)}{|z|^{d+q(t-1)+1}}\,dz\right)\\
&\ge-C\epsilon_0\Bigg(a(0,0)\int_{\mathcal{E}\cap B_{\sqrt{|\overline{a}|}}}\frac{dz}{|z|^{d+q(t-1)}}+[a]_{\rm lip}\int_{\mathcal{E}\cap B_{\sqrt{|\overline{a}|}}}\frac{dz}{|z|^{d+q(t-1)-1}}\\
&\qquad+a(0,0)|\nabla \psi(\overline{x})|\int_{\mathcal{E}\cap B_{\sqrt{|\overline{a}|}}}\frac{dz}{|z|^{d+q(t-1)+1}}+
[a]_{\rm lip}|\nabla \psi(\overline{x})|\int_{\mathcal{E}\cap B_{\sqrt{|\overline{a}|}}}\frac{dz}{|z|^{d+q(t-1)}}\Bigg)\\
&\ge -C\epsilon_0\Bigg(a(0,0)\int^{\sqrt{|\overline{a}|}}_{\delta_1|\overline{a}|}r^{q(1-t)-1}\,dr+[a]_{\rm lip}\int^{\sqrt{|\overline{a}|}}_{\delta_1|\overline{a}|}r^{q(1-t)}\,dr\\
&\qquad+a(0,0)|\nabla \psi(\overline{x})|\int^{\sqrt{|\overline{a}|}}_{\delta_1|\overline{a}|}r^{q(1-t)-2}\,dr+
[a]_{\rm lip}|\nabla \psi(\overline{x})|\int^{\sqrt{|\overline{a}|}}_{\delta_1|\overline{a}|}r^{q(1-t)-1}\,dr\Bigg)\\
&\ge -C\epsilon_0\Bigg(a(0,0)|\overline{a}|^\frac{q(1-t)}{2}+[a]_{\rm lip}|\overline{a}|^\frac{q(1-t)+1}{2}\\
&\qquad+a(0,0)\Big(\frac{|\overline{a}|}{\epsilon_0}\Big)^\frac{m-1}{m}\int^{\sqrt{|\overline{a}|}}_{\delta_1|\overline{a}|}r^{q(1-t)-2}\,dr+
[a]_{\rm lip}\Big(\frac{|\overline{a}|}{\epsilon_0}\Big)^\frac{m-1}{m}\int^{\sqrt{|\overline{a}|}}_{\delta_1|\overline{a}|}r^{q(1-t)-1}\,dr\Bigg),
\end{align*}
where in the last line we invoked the property of $\psi$ that $|\nabla \psi(\overline{x})|\le C\Big(\frac{|\overline{a}|}{\epsilon_0}\Big)^\frac{m-1}{m}$. Indeed, by \eqref{4-8}, we can see
$$
\psi(\overline{x})\le\frac{1}{\kappa\epsilon_0}(u(\overline{x})-u(\overline{y})-e\cdot(\overline{x}-\overline{y}))\le\frac{2|\overline{a}|}
{\kappa\epsilon_0}.
$$
Using the property of $\psi$ that $|\nabla \psi(\overline{x})|\le C\psi^\frac{m-1}{m}(x)$, we have $|\nabla \psi(\overline{x})|\le C\Big(\frac{|\overline{a}|}{\epsilon_0}\Big)^\frac{m-1}{m}$.

In a similar manner, we arrive at
$$
T_{31}+T_{32}\ge -C\left(|\overline{a}|^\frac{p(1-s)}{2}+\epsilon^\frac{1}{m}_0|\overline{a}|^\frac{m-1}{m}\int^{\sqrt{|\overline{a}|}}_{\delta_1
|\overline{a}|}r^{p(1-s)-2}\,dr\right).
$$
As a consequence,
\begin{align}
\label{4-8-3}
T_3\ge -C\Bigg(&|\overline{a}|^\frac{p(1-s)}{2}+a(0,0)|\overline{a}|^\frac{q(1-t)}{2}+[a]_{\rm lip}|\overline{a}|^\frac{q(1-t)+1}{2}+\epsilon^\frac{1}{m}_0|\overline{a}|^\frac{m-1}{m}\int^{\sqrt{|\overline{a}|}}_{\delta_1
|\overline{a}|}r^{p(1-s)-2}\,dr \nonumber\\
&+a(0,0)\epsilon_0^\frac{1}{m}|\overline{a}|^\frac{m-1}{m}\int^{\sqrt{|\overline{a}|}}_{\delta_1|\overline{a}|}r^{q(1-t)-2}\,dr+
[a]_{\rm lip}\epsilon_0^\frac{1}{m}|\overline{a}|^\frac{m-1}{m}\int^{\sqrt{|\overline{a}|}}_{\delta_1|\overline{a}|}r^{q(1-t)-1}\,dr \Bigg).
\end{align}
We directly know from the computations of \cite[Lemma 6.8]{Sil25} that
\begin{align}
\label{4-8-4}
|\overline{a}|^\frac{p(1-s)}{2}+\epsilon^\frac{1}{m}_0|\overline{a}|^\frac{m-1}{m}\int^{\sqrt{|\overline{a}|}}_{\delta_1
|\overline{a}|}r^{p(1-s)-2}\,dr \le C|\overline{a}|^\frac{p(1-s)}{2} \quad\text{for } p\in\left[2,\frac{2}{1-s}\right)
\end{align}
by taking $m$ sufficiently large. Note $q<\frac{1}{1-t}$. Then we can also get
\begin{align}
\label{4-8-5}
|\overline{a}|^\frac{q(1-t)}{2}+\epsilon^\frac{1}{m}_0|\overline{a}|^\frac{m-1}{m}\int^{\sqrt{|\overline{a}|}}_{\delta_1
|\overline{a}|}r^{q(1-t)-2}\,dr \le C|\overline{a}|^\frac{q(1-t)}{2} 
\end{align}
through selecting such large $m$ that $\frac{1}{m}<\frac{q(1-t)}{2}$. Besides, we can readily have
\begin{align}
\label{4-8-6}
|\overline{a}|^\frac{m-1}{m}\int^{\sqrt{|\overline{a}|}}_{\delta_1|\overline{a}|}r^{q(1-t)-1}\,dr\le \frac{1}{q(1-t)}|\overline{a}|^\frac{m-1}{m}|\overline{a}|^\frac{q(1-t)}{2}\le C|\overline{a}|^\frac{q(1-t)+1}{2}
\end{align}
via noticing that $|\overline{a}|$ is very small and choosing $m$ so large that $\frac{m-1}{m}\ge\frac{1}{2}$. Putting inequalities \eqref{4-8-3}--\eqref{4-8-6} together immediately gives
$$
T_3\ge -C\left(|\overline{a}|^\frac{p(1-s)}{2}+a(0,0)|\overline{a}|^\frac{q(1-t)}{2}+[a]_{\rm lip}|\overline{a}|^\frac{q(1-t)+1}{2}\right),
$$
where we note $|\overline{a}|<1$, and $C>0$ depends only upon $d,s,t,p,q$ and the choice of $\nu$ in $\delta_1$.
\end{proof}

Finally, let us complete the estimate on $T_4$. We deduce the Lipschitz continuity for a function concerning the nonlocal integral, which could easily result in the bound on $T_4$.

Define
$$
F(x)=\int_{B_R^c(x)}\frac{|u(x)-u(z)|^{p-2}(u(x)-u(z))}{|x-z|^{d+sp}}+a(x,z)\frac{|u(x)-u(z)|^{q-2}(u(x)-u(z))}{|x-z|^{d+tq}}\,dz.
$$

\begin{lemma}
\label{lem4-9}
Let $tq\le sp+1$ and the conditions $(A_3),(A_4)$ on the coefficient $a(\cdot)$ hold. Assume that $u\in L^{q-1}_{tq}(\mathbb{R}^d)$ is Lipschitz in $B_r$ and bounded in $B_{R+r}$ with $0<r\le R$. 
Then for $x,y\in B_r$, it holds that
\begin{align*}
|F(x)-F(y)|\le C|x-y|,
\end{align*}
where the constant $C>0$ depends on $d,s,t,p,q,\|a\|_{\rm lip}$ and $R^{-1},\|u\|_{L^\infty(B_{R+r})}$, $[u]_{{\rm lip}(B_r)}$, $\mathrm{Tail}_{q-1,tq}(u;R+r)$.
\end{lemma}

\begin{proof}
We rewrite $F(x)$ as
$$
F(x)=F_p(x)+F_q(x):=\int_{B_R^c(x)}\frac{J_p(u(x)-u(z))}{|x-z|^{d+sp}}\,dz+\int_{B_R^c(x)}a(x,z)\frac{J_q(u(x)-u(z))}{|x-z|^{d+tq}}\,dz.
$$
We just consider $F_q(x)$. Let $x, y\in B_r$ fulfill $|x-y|<\frac{R}{2}$. We evaluate
\begin{align*}
&\quad|F_q(x)-F_q(y)|\\
&\le \int_{B_R^c(x)}a(x,z)|J_q(u(x)-u(z))|\left||x-z|^{-d-tq}-|y-z|^{-d-tq}\right|\,dz\\
&\quad+\left|\int_{B_R^c(x)}a(x,z)\frac{|J_q(u(x)-u(z))|}{|y-z|^{d+tq}}\,dz-\int_{B_R^c(y)}a(x,z)\frac{|J_q(u(x)-u(z))|}{|y-z|^{d+tq}}\,dz\right|\\
&\quad+\int_{B_R^c(y)}a(x,z)\frac{|J_q(u(x)-u(z))-J_q(u(y)-u(z))|}{|y-z|^{d+tq}}\,dz\\
&\quad+\int_{B_R^c(y)}|a(x,z)-a(y,z)|\frac{|J_q(u(y)-u(z))|}{|y-z|^{d+tq}}\,dz\\
&\le \|a\|_\infty\int_{B_R^c(x)}|u(x)-u(z)|^{q-1}\left||x-z|^{-d-tq}-|y-z|^{-d-tq}\right|\,dz\\
&\quad+\|a\|_\infty\int_{B_R(x)\triangle B_R(y)}\frac{|u(x)-u(z)|^{q-1}}{|x-z|^{d+tq}}\,dz\\
&\quad+\|a\|_\infty\int_{B_R^c(y)}\frac{|J_q(u(x)-u(z))-J_q(u(y)-u(z))|}{|y-z|^{d+tq}}\,dz\\
&\quad+\int_{B_R^c(y)}|a(x,z)-a(y,z)|\frac{|u(y)-u(z)|^{q-1}}{|y-z|^{d+tq}}\,dz\\
&=:\|a\|_\infty I_1+\|a\|_\infty I_2+\|a\|_\infty I_3+I_4,
\end{align*}
where $B_R(x)\triangle B_R(y)=\Big(B_R(x)\setminus B_R(y)\Big)\cup\Big(B_R(y)\setminus B_R(x)\Big)$ stands for the symmetric difference. Using the Lipschitz continuity of $a(\cdot)$, we handle $I_4$ as
\begin{align*}
I_4&\le 2^{q-1}[a]_{\rm lip}\int_{B_R^c(y)}|x-y|\frac{|u(y)|^{q-1}+|u(z)|^{q-1}}{|y-z|^{d+tq}}\,dz\\
&\le C[a]_{\rm lip}|x-y|\left(R^{-tq}\|u\|^{q-1}_{L^\infty(B_r)}+\mathrm{Tail}_{q-1,tq}(u;y,R)^{q-1}\right).
\end{align*}

The treatment of $I_1-I_3$ is similar to that of $J_1-J_3$ in \cite[Proof of Lemma 6.9]{Sil25}, so we directly write
\begin{align*}
I_1+I_2+I_3\le C|x-y|\Big(&R^{-tq-1}\|u\|^{q-1}_{L^\infty(B_{R+r})}+\mathrm{Tail}_{q-1,tq+1}(u;x,R)^{q-1}\\
&+[u]_{{\rm lip}(B_r)}\left(R^{-tq}\|u\|^{q-2}_{L^\infty(B_r)}+\mathrm{Tail}_{q-2,tq}(u;y,R)^{q-2}\right)\Big).
\end{align*}
Then we further have
\begin{align*}
\mathrm{Tail}_{q-1,tq+1}(u;x,R)^{q-1}\le R^{-1}\mathrm{Tail}_{q-1,tq}(u;x,R)^{q-1}
\end{align*}
and
\begin{align*}
\mathrm{Tail}_{q-2,tq}(u;y,R)^{q-2}&\le\left(\int_{B_R^c(y)}\frac{|u(z)|^{q-1}}{|y-z|^{d+tq}}\,dz\right)^\frac{q-2}{q-1}
\left(\int_{B_R^c(y)}\frac{dz}{|y-z|^{d+tq}}\right)^\frac{1}{q-1}\\
&=CR^{-\frac{tq}{q-1}}\mathrm{Tail}_{q-1,tq}(u;y,R)^{q-2}.
\end{align*}
For $x\in B_r$ and $z\in B^c_R(x)$, we get
$$
|z|\le |z-x|\left(1+\frac{|x|}{|z-x|}\right)\le|z-x|\left(1+\frac{r}{R}\right)\le 2|z-x|,
$$
and $B_R(x)\subset B_{R+r}$ as well as
$$
 B^c_R(x)=B_{R+r}^c\cup\big(B_{R+r}\setminus B_{R}(x)\big)\subset B_{R+r}^c\cup\big(B_{R+3r}(x)\setminus B_{R}(x)\big).
 $$
Thus there holds that
$$
\int_{B_R^c(x)}\frac{|u(z)|^{q-1}}{|x-z|^{d+tq}}\,dz\le 2^{d+tq}\int_{B_{R+r}^c}\frac{|u(z)|^{q-1}}{|z|^{d+tq}}\,dz+\int_{B_{R+r}\setminus B_{R}(x)}\frac{|u(z)|^{q-1}}{|x-z|^{d+tq}}\,dz
$$
and
\begin{align*}
\int_{B_{R+r}\setminus B_{R}(x)}\frac{|u(z)|^{q-1}}{|x-z|^{d+tq}}\,dz&\le\|u\|^{q-1}_{L^\infty(B_{R+r})}\int_{B_{R+r}\setminus B_{R}(x)}\frac{dz}{|x-z|^{d+tq}}\\
&\le \|u\|^{q-1}_{L^\infty(B_{R+r})}\int_{B_{R+3r}(x)\setminus B_{R}(x)}\frac{dz}{|x-z|^{d+tq}}\\
&\le CR^{-tq}\|u\|^{q-1}_{L^\infty(B_{R+r})}.
\end{align*}
Then,
$$
\mathrm{Tail}_{q-1,tq}(u;x,R)^{q-1}\le C\big(\mathrm{Tail}_{q-1,tq}(u;R+r)^{q-1}+R^{-tq}\|u\|^{q-1}_{L^\infty(B_{R+r})}\big)
$$
and analogously,
$$
\mathrm{Tail}_{q-1,tq}(u;y,R)^{q-1}\le C\big(\mathrm{Tail}_{q-1,tq}(u;R+r)^{q-1}+R^{-tq}\|u\|^{q-1}_{L^\infty(B_{R+r})}\big).
$$
Here the constant $C>0$ depends only on $d,t,q$.

At this stage, combing these preceding displays we obtain
\begin{align*}
&\quad|F_q(x)-F_q(y)|\\
&\le C(d,t,q,\|a\|_{\rm lip})|x-y|\Big((R^{-tq}+R^{-tq-1})\|u\|^{q-1}_{L^\infty(B_{R+r})}+R^{-tq}[u]_{{\rm lip}(B_r)}\|u\|^{q-2}_{L^\infty(B_{R+r})}\\
&\qquad+(1+R^{-1})\mathrm{Tail}_{q-1,tq}(u;R+r)^{q-1}+[u]_{{\rm lip}(B_r)}R^{-\frac{tq}{q-1}}\mathrm{Tail}_{q-1,tq}(u;R+r)^{q-2}\Big)\\
&=:C|x-y|,
\end{align*}
where the positive constant $C$ depends on $d,t,q,\|a\|_{\rm lip},[u]_{{\rm lip}(B_r)}$, $\|u\|_{L^\infty(B_{R+r})}$ and $R^{-1}$, $\mathrm{Tail}_{q-1,tq}(u;R+r)$. The evaluation of $|F_p(x)-F_p(y)|$ is easier, so we directly write
\begin{align*}
&\quad|F_p(x)-F_p(y)|\\
&\le C|x-y|\Big(R^{-sp-1}\|u\|^{p-1}_{L^\infty(B_{R+r})}+\mathrm{Tail}_{p-1,sp+1}(u;x,R)^{p-1}+R^{-sp}[u]_{{\rm lip}(B_r)}\|u\|^{p-2}_{L^\infty(B_{R+r})}\\
&\qquad+[u]_{{\rm lip}(B_r)}R^{\frac{p(1-s)-2}{p-1}}\mathrm{Tail}_{p-1,sp+1}(u;y,R)^{p-2}\Big).
\end{align*}
By means of Lemma \ref{lem2-0} (2), we can see
$$
\mathrm{Tail}_{p-1,sp+1}(u;\widetilde{x},R)^{p-1}\le CR^{-\left(sp+1-tq\frac{p-1}{q-1}\right)}\mathrm{Tail}_{q-1,tq}(u;\widetilde{x},R)^{p-1},  \quad \widetilde{x}\in\{x,y\}
$$
for $tq\le sp+1$. As results,
$$
|F_p(x)-F_p(y)|\le C|x-y|.
$$
Here the constant $C>0$ depends on $d,s,p$, $[u]_{{\rm lip}(B_r)}$, $\|u\|_{L^\infty(B_{R+r})}$ and $R^{-1}$, $\mathrm{Tail}_{q-1,tq}(u;R+r)$. Hence we have
$$
|F(x)-F(y)|\le C|x-y| \quad\text{for } |x-y|<\frac{R}{2}.
$$
In addition, if $x,y\in B_r$ satisfies $|x-y|\ge\frac{R}{2}$, by noting $|x-y|<2R$, we connect them with three intermediate points $x_1,x_2,x_3$ in $B_r$ such that $|x-x_1|<\frac{R}{2}$, $|x_i-x_{i+1}|<\frac{R}{2}$, $|x_3-y|<\frac{R}{2}$ ($i=1,2$). Then we can repeat the previous processes for these situations and get further the desired conclusion.
\end{proof}

Based on this lemma, we have the estimate on $B^c_\frac{1}{16}$.

\begin{corollary}[Estimate on $B^c_\frac{1}{16}$]
\label{cor4-10}
Let $tq\le sp+1$ and the conditions $(A_3), (A_4)$ on the coefficient $a(\cdot)$ be in force. It holds that
$$
|T_4|\le C|\overline{a}|
$$
with $C>0$ depending only on $d,s,t,p,q,C_1$ and $\|a\|_{\rm lip}$.
\end{corollary}

\begin{proof}
We take $R=\frac{1}{16}$ and $r=\frac{1}{32}$ in Lemma \ref{lem4-9} to get
\begin{align*}
|T_4|\le|F(\overline{x})-F(\overline{y})|\le C(d,s,t,p,q,C_1,\|a\|_{\rm lip})|\overline{x}-\overline{y}|,
\end{align*}
where we have exploited
$$
\|u\|_{L^\infty(B_{3/32})}, \  [u]_{{\rm lip}(B_{3/32})}\le1
$$
and
\begin{align*}
\mathrm{Tail}_{q-1,tq}(u;3/32)^{q-1}&=\mathrm{Tail}_{q-1,tq}(u;1)^{q-1}+\int_{B_1\setminus B_\frac{3}{32}}\frac{|u(z)|^{q-1}}{|z|^{d+tq}}\,dz\\
&\le C_1^{q-1}+C(d,t,q)\|u\|_{L^\infty(B_1)}^{q-1}\le C_1^{q-1}+C(d,t,q).
\end{align*}
\end{proof}

Finally, we conclude this section by presenting the proof of Lemma \ref{lem4-3}. This can be easily realized via combining Lemmas \ref{lem4-6}--\ref{lem4-8} and Corollary \ref{cor4-10} with the equality \eqref{4-9}.

\begin{proof}[\textbf{Proof of Lemma \ref{lem4-3}}] By the Ishii-Lions method, we aim to reach a contradiction in \eqref{4-9}. Collecting the estimates obtained in Lemmas \ref{lem4-6}--\ref{lem4-8} and Corollary \ref{cor4-10} and using the equality \eqref{4-9}, we have
\begin{align*}
0&\ge CL\delta_0^{d+p-1}|\overline{a}|^{p(1-s)-1}+a(0,0)CL\delta_0^{d+q-1}|\overline{a}|^{q(1-t)-1}-CL[a]_{\rm lip}\delta_0^{d+q-1}|\overline{a}|^{q(1-t)}\\
&\quad-CL\delta_1^{p(1-s)}|\overline{a}|^{p(1-s)-1}-a(0,0)CL\delta_1^{q(1-t)}|\overline{a}|^{q(1-t)-1}-CL[a]_{\rm lip}\delta_1^{q(1-t)+1}|\overline{a}|^{q(1-t)}\\
&\quad-C\left(|\overline{a}|^{\frac{p(1-s)}{2}}+a(0,0)|\overline{a}|^{\frac{q(1-t)}{2}}+[a]_{\rm lip}
|\overline{a}|^{\frac{q(1-t)+1}{2}}\right) -C|\overline{a}|.
\end{align*}
We rearrange this display as
\begin{align*}
0&\ge \bigg[CL\delta_0^{d+p-1}|\overline{a}|^{p(1-s)-1}-CL\delta_1^{p(1-s)}|\overline{a}|^{p(1-s)-1}-C|\overline{a}|^{\frac{p(1-s)}{2}}-C|\overline{a}|\\
&\quad-[a]_{\rm lip}\left(CL\delta_0^{d+q-1}|\overline{a}|^{q(1-t)}+CL\delta_1^{q(1-t)+1}|\overline{a}|^{q(1-t)}+C
|\overline{a}|^{\frac{q(1-t)+1}{2}}\right)\bigg]\\
&\quad+a(0,0)\left(CL\delta_0^{d+q-1}|\overline{a}|^{q(1-t)-1}-CL\delta_1^{q(1-t)}|\overline{a}|^{q(1-t)-1}-
C|\overline{a}|^{\frac{q(1-t)}{2}}\right)\\
&=:I_1+I_2.
\end{align*}
We shall justify the right-hand side is larger than 0 by choose suitable exponent $\nu$ in the definition of $\delta_1$ and taking $\epsilon_0$ small enough. Remember $|\overline{a}|<1$ is sufficiently small for small enough $\epsilon_0$. First, we select such large $\nu$ that $\nu q(1-t)>d+q-1$, and then get
$$
\delta_0^{d+q-1}|\overline{a}|^{q(1-t)-1}\gg\delta_1^{q(1-t)}|\overline{a}|^{q(1-t)-1}.
 $$
 Second, we have $q(1-t)-1<\frac{q(1-t)}{2}$ by $q<\frac{1}{1-t}$, and further
$$
 \delta_0^{d+q-1}|\overline{a}|^{q(1-t)-1}\gg|\overline{a}|^{\frac{q(1-t)}{2}}.
 $$
Thus we get $I_2\ge0$ since $a(0,0)\ge0$.

On the other hand, similarly, by picking such large $\nu$ that $\nu p(1-s)>d+p-1$ and recalling the fact $p<\frac{2}{1-s}$, we know
$$
\delta_0^{d+p-1}|\overline{a}|^{p(1-s)-1}\gg\delta_1^{p(1-s)}|\overline{a}|^{p(1-s)-1}+|\overline{a}|^{\frac{p(1-s)}{2}}+|\overline{a}|.
$$
The remaining task is to compare $\delta_0^{d+p-1}|\overline{a}|^{p(1-s)-1}$ with the last three terms in $I_1$. We first observe that
$$
\delta_0^{d+q-1}|\overline{a}|^{q(1-t)}\gg\delta_1^{q(1-t)+1}|\overline{a}|^{q(1-t)}
$$
by choosing sufficiently large $\nu$. Taking into the preconditions $p\le q$ and $tq\le sp+1$, we have $d+p-1\le d+q-1$ and $p(1-s)-1<q(1-t)$, and moreover get
$$
\delta_0^{d+p-1}|\overline{a}|^{p(1-s)-1}\gg\delta_0^{d+q-1}|\overline{a}|^{q(1-t)}.
$$
Finally, let us justify the relation $p(1-s)-1\le\frac{q(1-t)+1}{2}$ so that we derive
$$
\delta_0^{d+p-1}|\overline{a}|^{p(1-s)-1}\gg|\overline{a}|^{\frac{q(1-t)+1}{2}}.
 $$
 Indeed, resorting to $p\le q, tq\le sp+1$ and $p<\frac{2}{1-s}$, we can see
\begin{align*}
&tq\le sp+q-p+(s-1)\frac{2}{1-s}+3< sp+q-p+(s-1)p+3\\
&\qquad \Rightarrow \ p(1-s)-1<\frac{q(1-t)+1}{2}.
\end{align*}
At this time, we get $I_1>0$ and so arrive at
$$
0=T_1+T_2+T_3+T_4>0,
$$
which is a contradiction. This ends the proof.
\end{proof}

\section{Key estimates toward Lipschitz of solutions}
\label{sec6}

In this portion, we are going to make use of the Ishii-Lions methods to demonstrate Lipschitz property for viscosity solutions to \eqref{main} under the natural structure, which is of independent interest for nonlocal double phase equations. Although this method has been applied to fractional double phase equations with the condition $tq\le sp$ ($s\ge t$) to derive such regularity of solutions in \cite{BS}, there are some new challenges that we have to overcome, due to the uncertainty over the dominant role of either the $p$-growth or the $q$-growth.

In addition, we would like to point out that even if some ideas in Sections \ref{sec6} and \ref{sec7} are analogous to those in the proof of Lemma \ref{lem4-3} in Section \ref{sec5}, extra difficulties arise in establishing Lipschitz continuity of viscosity solutions without any \emph{a priori} regularity assumptions. To do so, we have to adopt the bootstrap argument to upgrade the regularity from $L^\infty$ to $C^{0,1}$.

Now let us clarify the setup concerning the proof of Lipschitz of viscosity solutions. We proceed by doubling the variables. Fix $1\le \rho_1<\rho_2\le2$. Denote the function
$$
\Psi(x,y)=u(x)-u(y)-L\omega(|x-y|)-m_1\psi(x) \quad x,y\in B_2
$$
with
$$
\psi(x)=[(|x|^2-\rho^2_1)_+]^m \quad x\in B_2,
$$
where $m\ge3$ assures $\psi\in C^2(B_2)$, and the regularizing function $\omega(\cdot)$, nonnegative and strictly concave, encodes the modulus of continuity for the solution $u$. For H\"older profile or Lipschitz profile of $u$, we use separately the following two classes of regularizing functions $\omega$ (after appropriately scaling):
 $$
 \omega_\beta(\rho)=\rho^\beta \quad \text{with } \beta\in(0,1),
 $$
 and
\begin{equation*}
\widetilde{\omega}(\rho)=\begin{cases}\rho+\frac{\rho}{\log\rho} &\text{for } \rho>0,\\[2mm]
0 &\text{for } \rho=0.
\end{cases}
\end{equation*}
Through simple calculations, we know that $\widetilde{\omega}'(\rho)\in(\frac{1}{4},1)$ for $\rho\in(0,2]$, and
\begin{equation*}
-\frac{C}{\rho\log^2\rho}\le\widetilde{\omega}''(\rho)\le -\frac{c}{\rho\log^2\rho} \quad\text{for } \rho\in\Big(0,\frac{1}{2}\Big]
\end{equation*}
with $c,C>0$ being absolute constants. For simplicity, let
$$
\phi(x,y)=L\omega(|x-y|)+m_1\psi(x).
$$

In the sequel, we are ready to demonstrate that $u$ has $L\omega$ as a modulus of continuity in $B_{\rho_1}$. To this end, we verify $\Psi(x,y)\le0$ for all $x, y\in B_2$, which then indicates the desired modulus of continuity in $B_{\rho_1}$, because $\psi$ is zero in $B_{\rho_1}$. This thrives on contradiction. Suppose there are $\overline{x},\overline{y}\in B_2$ such that
\begin{equation}
\label{M}
\Psi(\overline{x},\overline{y})=\sup_{(x,y)\in B_2\times B_2}\Psi(x,y)>0,
\end{equation}
that is,
$$
u(\overline{x})-u(\overline{y})>L\omega(|\overline{x}-\overline{y}|)+m_1\psi(\overline{x}).
$$
It is known that the left-hand side is bounded by $2\|u\|_{L^\infty(B_2)}$. Thus we could take $m_1$ so large that $\overline{x}$ is in $B_\frac{\rho_2+\rho_1}{2}$ by $m_1\psi(\overline{x})\le2\|u\|_{L^\infty(B_2)}$. Let
$$
\overline{a}:=\overline{x}-\overline{y}.
$$
With the help of
\begin{equation}
\label{M1}
L\omega(|\overline{a}|)\le 2\|u\|_{L^\infty(B_2)},
\end{equation}
we may get $|\overline{a}|<\frac{\rho_2-\rho_1}{4}$ by selecting $L$ large enough so that $\overline{x}\in B_\frac{\rho_2+\rho_1}{2},\overline{y}\in B_{\frac{3\rho_2}{4}+\frac{\rho_1}{4}}$. This observation is beneficial to the subsequent processes.

We now take the following test functions
\[
v_1(z)=
\begin{cases}
\phi(z,\ybar)+u(\xbar)-\phi(\xbar,\ybar) & z\in B_\delta(\xbar),\\
u(z) & \text{otherwise},
\end{cases}
\]
and
\[
v_2(z)=
\begin{cases}
-\phi(\xbar,z)+u(\ybar)+\phi(\xbar,\ybar) & z\in B_{\delta}(\ybar),\\
u(z) & \text{otherwise}.
\end{cases}
\]
It is easy to know by \eqref{M} that $v_1$ touches $u$ at $\xbar$ from above, and $v_2$ touches $u$ at $\ybar$ from below. Thus according to the definition of viscosity solution $u$ to Eq. \eqref{main}, we have the viscosity inequalities
\[
\cL_p[\Rd]v_1(\xbar)+\cL_q[\Rd]v_1(\xbar)\le0,
\qquad
\cL_p[\Rd]v_2(\ybar)+\cL_q[\Rd]v_2(\ybar)\ge0,
\]
and furthermore
\begin{align}
\label{ve}
0&\ge\cL_p[\cC]v_1(\xbar)-\cL_p[\cC]v_2(\ybar)+\cL_p[\cD_1]v_1(\xbar)-\cL_p[\cD_1]v_2(\ybar) \nonumber\\
&\quad+\cL_p[\cD_2]v_1(\xbar)-\cL_p[\cD_2]v_2(\ybar)+\cL_p[B^c_{\rho_0}]v_1(\xbar)-\cL_p[B^c_{\rho_0}]v_2(\ybar) \nonumber\\
&\quad+\cL_q[\cC]v_1(\xbar)-\cL_q[\cC]v_2(\ybar)+\cL_q[\cD_1]v_1(\xbar)-\cL_q[\cD_1]v_2(\ybar) \nonumber\\
&\quad+\cL_q[\cD_2]v_1(\xbar)-\cL_q[\cD_2]v_2(\ybar)+\cL_q[B^c_{\rho_0}]v_1(\xbar)-\cL_q[B^c_{\rho_0}]v_2(\ybar) \nonumber\\
&=:I_1+I_2+I_3+I_4+J_1+J_2+J_3+J_4,
\end{align}
where the integration domain $\Rd$ was decomposed as
\[
\cC :=
\bigl\{z\in \B_{\del_0\abs{\abar}}: \abs{\ip{\abar}{z}}\ge (1-\eta_0)\abs{\abar}\abs{z}\bigr\},
\]
and define
\[
\cD_1:=\B_{\del}\cap \cC^c,\qquad
\cD_2:=\B_{\rho_0}\setminus(\cD_1\cup \cC).
\]
Here $\delta_0=\eta_0$, \(\rho_0:=\frac{\rho_2-\rho_1}{4}\), \(\del\ll \del_0\abs{\abar}\ll \rho_0 \), and the operators $\cL_p,\cL_q$ are given in \eqref{Lop}.

Next, we are going to address the terms $I_1-I_4$ and $J_1-J_4$. We will use the notation $\delta^2u$ many times to represent the second increment of $u$ as
$$
\delta^2u(x,z)=u(x)+\nabla u(x)\cdot z-u(x+z).
$$
Now we present useful basic estimates on the cone $\cC$ from \cite{BT25}, which is restated in our framework.


\begin{lemma}
\label{bte}
\begin{itemize}
\item[(1)] For \( \omega(\rho) = \omega_\beta(\rho) = \rho^\beta \), \( \beta \in (0, 1) \), there exist \( L_0, \varepsilon > 0 \), independent of \( \bar{a} \), such that for all \( \delta_0 \in (0, \varepsilon] \) and all \( L \geq L_0 \) it holds
\[
\frac{1}{\kappa} L |\bar{a}|^{\beta-2} |z|^2 \leq \delta^2 \phi(\cdot, \bar{y})(\bar{x}, z) \leq \kappa L |\bar{a}|^{\beta-2} |z|^2, \quad z \in \cC
\]
for some constant \( \kappa \) depending on \( \beta, \delta_0, d \).

\smallskip

\item[(2)] Denote
\[
\widetilde{\omega}(\beta) =
\begin{cases}
\rho + \dfrac{\rho}{\log \rho} & \text{if } \rho \in (0, r_0), \\
0 & \text{if } \rho = 0,
\end{cases}
\]
for some \( r_0 \in (0, 1) \) small enough. For \( \omega(\rho) = \tilde{\omega}(\frac{r_0}{3}\rho) \), there exist \( L_0, \varepsilon > 0 \), independent of \( \bar{a} \), such that for \( \delta_0 = \delta_1 (\log^2 |\bar{a}|)^{-1} \) and \( \delta_1 \in (0, \varepsilon] \), \( L \geq L_0 \), it holds
\[
\frac{1}{\kappa} L \bigl(|\bar{a}| \log^2 |\bar{a}|\bigr)^{-1} |z|^2 \leq \delta^2 \phi(\cdot, \bar{y})(\bar{x}, z)
\leq \kappa L \bigl(|\bar{a}| \log^2 |\bar{a}|\bigr)^{-1} |z|^2, \quad z \in \cC,
\]
where $\kappa>0$ depends only on $d$.
\smallskip

\item[(3)] Same estimates hold for \( \delta^2 \phi(\bar{x}, \cdot)(\bar{y}, z) \) in all the cases above.
\end{itemize}
\end{lemma}

\subsection{Lower bounds on the integrals in \eqref{ve}}

\mbox{}\par
\medskip

First, we evaluate the integrals over the cone $\cC$.

\begin{lemma}
\label{lem6-1}
Let \(p,q\ge2\) and \(\omega(\rho):=\omega_\beta(\rho)=\rho^\beta\) with \(\beta\in(0,1)\).
\begin{itemize}

\item[(1)] {There exist
\(L_0>0\) and \(\delta_0>0\), dependent on \(d,t,q,\beta\), such that for all
\(L\ge L_0\),
\begin{equation*}
J_1\ge
C\del_0^{\frac{d-1}{2}+q(1-t)} L^{q-1}
\left[
a(0,0)\abs{\abar}^{\beta(q-1)-tq}
-[a]_\alpha \del_0^{\alpha}
 \abs{\abar}^{\beta(q-1)-tq+\alpha}
\right],
\end{equation*}
where $C>0$ depends only on $d,t,q,\beta$, provided the conditions ($A_1$) ($A_2$) ($A'_4$) hold true.}

\smallskip

\item[(2)] {There exist
\(L_0>0\) and \(\delta_0>0\), dependent on \(d,s,p,\beta\), such that for all
\(L\ge L_0\),
\begin{equation*}
I_1\ge C\del_0^{\frac{d-1}{2}+p(1-s)} L^{p-1}\abs{\abar}^{\beta(p-1)-sp},
\end{equation*}
where $C>0$ depends only on $d,s,p,\beta$.}

\end{itemize}
\end{lemma}

\begin{proof}
Let $\ell(z):=-\nabla_x\phi(\xbar,\ybar)\cdot z $. By symmetry and translation-invariance of $a(\cdot)$, we see
$$
a(0,-z)J_q(\nabla_x\phi(\xbar,\ybar)\cdot(- z))=-a(0,-z)J_q(\nabla_x\phi(\xbar,\ybar)\cdot z)=-a(0,z)J_q(\nabla_x\phi(\xbar,\ybar)\cdot z)
$$
and further
\[
\cL_q[\cC]\ell(\xbar)=
\int_{\cC}a(0,z)\Jop q(\ell(\xbar)-\ell(\xbar+z))\,\frac{dz}{\abs{z}^{d+tq}}=0
\]
due to the symmetry of the cone \(\cC\). Recalling the monotonicity of \(\Jop q\) and the maximum property \eqref{M}, we obtain
\begin{align*}
\cL_q[\cC]v_1(\xbar)
&\ge
\cL_q[\cC]\phi(\cdot,\ybar)(\xbar)-\cL_q[\cC]\ell(\xbar)\\
&=(q-1)\int_{\cC}\int_0^1
 a(0,z)
\abs{\ell(z)+\tau\delta^2\phi(\cdot,\ybar)(\xbar,z)}^{q-2}
\delta^2\phi(\cdot,\ybar)(\xbar,z)
\,d\tau\frac{dz}{\abs{z}^{d+tq}} \\
&\ge C\int_{\cC}a(0,z)
\abs{\ell(z)}^{q-2}\delta^2\phi(\cdot,\ybar)(\xbar,z)\frac{dz}{\abs{z}^{d+tq}},
\end{align*}
where in the last line we used Lemma \ref{lem2-0-0}. For \(L\) sufficiently large, we get
\[
\abs{\ell(z)}\ge \frac{L}{2}\omega'(\abs{\abar})\abs{z},
\qquad z\in \cC .
\]
Consequently,
\begin{align}
\label{l3-1}
\cL_q[\cC]v_1(\xbar)
&\ge
C\bigl(L\omega'(\abs{\abar})\bigr)^{q-2}
\int_{\cC}
\bigl(a(0,0)+a(0,z)-a(0,0)\bigr)
\delta^2\phi(\cdot,\ybar)(\xbar,z)
\frac{dz}{\abs{z}^{d+tq-q+2}} \nonumber\\
&\ge
C\bigl(L\omega'(\abs{\abar})\bigr)^{q-2}
\left[a(0,0)\int_{\cC}\frac{\delta^2\phi(\cdot,\ybar)(\xbar,z)}
     {\abs{z}^{d+q(t-1)+2}}\,dz
-[a]_\alpha\int_{\cC}
\frac{\delta^2\phi(\cdot,\ybar)(\xbar,z)}
     {\abs{z}^{d+q(t-1)+2-\alpha}}\,dz
\right].
\end{align}
For $\omega(\rho)=\rho^\beta$, then via Lemma \ref{bte} we arrive at
\begin{align*}
\cL_q[\cC]v_1(\xbar)
&\ge
C\bigl(L\abs{\abar}^{\beta-1}\bigr)^{q-2}
\left[a(0,0)\int_{\cC}
\frac{L|\abar|^{\beta-2}}{\abs{z}^{d+q(t-1)}}\,dz-[a]_\alpha\int_{\cC}\frac{L|\abar|^{\beta-2}}
     {\abs{z}^{d+q(t-1)-\alpha}}\,dz\right]\\
&=CL^{q-1}|\abar|^{(\beta-1)(q-1)-1}\left(a(0,0)\int_{\cC}
\frac{dz}{\abs{z}^{d+q(t-1)}}-[a]_\alpha\int_{\cC}\frac{dz}{\abs{z}^{d+q(t-1)-\alpha}}\right).
\end{align*}
Observe the fact that the aperture $\theta$ of the cone $\cC$ fulfills $\theta\approx \eta_0^\frac{1}{2}$. We could obtain
\begin{align}
\label{l3-2}
\int_{\cC}\frac{dz}{\abs{z}^{d+q(t-1)}}&=\frac{|\cC|}{|B_{\delta_0|\abar|}|}\int_{B_{\delta_0|\abar|}}\frac{dz}{\abs{z}^{d+q(t-1)}}
=\frac{|\cC_1|}{|B_1|}\int_{B_1}\frac{(\delta_0|\abar|)^{q(1-t)}}{\abs{z}^{d+q(t-1)}}\,dz \nonumber\\
&=C(\delta_0|\abar|)^{q(1-t)}|\cC_1|\ge C(\del_0\abs{\abar})^{q(1-t)}\eta_0^\frac{d-1}{2}
\end{align}
and
\begin{align}
\label{l3-3}
\int_{\cC}\frac{dz}{\abs{z}^{d+q(t-1)-\alpha}}=C(\delta_0|\abar|)^{q(1-t)+\alpha}|\cC_1|\le C(\del_0\abs{\abar})^{q(1-t)+\alpha}\eta_0^\frac{d-1}{2},
\end{align}
where $\mathcal C_1$ denotes the cone with radius 1, and the constant $C>0$ just depends on $d,t,q$.
At this time, it follows from $\delta_0=\eta_0$ that
$$
\cL_q[\cC]v_1(\xbar)\ge C\delta_0^{\frac{d-1}{2}+q(1-t)}L^{q-1}\left(a(0,0)|\abar|^{\beta(q-1)-tq}-[a]_\alpha\delta_0^\alpha|\abar|^{\beta(q-1)-tq+\alpha}\right)
$$
with $C>0$ depending only on $d,t,q,\beta$. The similar argument gives
the corresponding lower bound for \(-\cL_q[\cC]v_2(\ybar)\), and hence the desired conclusion on $J_1$ is valid. Moreover, lower bound on $I_1$ can be recovered in the case of $a(\cdot)\equiv1,t=s,q=p$.
\end{proof}

Now we consider the integrals on the transverse near-field $\cD_1$.

\begin{lemma}
\label{lem6-2}
Let \(p,q\ge2\) and \(\del=\eps_1\abs{\abar}\) with \(\eps_1\in(0,1/2)\).
\begin{itemize}
\item[(1)]{There are constants \(C,L_0>0\), independent of \(\eps_1\) and \(\abs{\abar}\), such that
\begin{equation*}
I_2\ge-C L^{p-1}\eps_1^{p(1-s)}\bigl(\omega'(\abs{\abar})\bigr)^{p-1}\abs{\abar}^{p(1-s)-1}
\end{equation*}
for all \(L\ge L_0\). Here $C$ depends on $d,p$.}

\smallskip

\item[(2)] {Assume the conditions ($A_1$) ($A_2$) ($A'_4$) is true.  Then there are constants \(C,L_0>0\), independent of \(\eps_1\) and \(\abs{\abar}\), such that
\begin{equation*}
J_2\ge-C L^{q-1}\eps_1^{q(1-t)}\bigl(\omega'(\abs{\abar})\bigr)^{q-1}
\left(
a(0,0)\abs{\abar}^{q(1-t)-1}+[a]_\alpha\eps_1^\alpha\abs{\abar}^{q(1-t)+\alpha-1}
\right)
\end{equation*}
for all \(L\ge L_0\). Here $C$ depends on $d,q$.}
\end{itemize}
\end{lemma}

\begin{proof}
Let us prove (2) here, since (1) is a special case of (2). Because the domain \(\cD_1\) is symmetric, as done before in Lemma \ref{lem6-1} we can write
\[
\cL_q[\cD_1]v_1(\xbar)\ge(q-1)\int_{\cD_1}\int_0^1a(0,z)
\abs{\ell(z)+\tau\delta^2\phi(\cdot,\ybar)(\xbar,z)}^{q-2}
\delta^2\phi(\cdot,\ybar)(\xbar,z)\,d\tau\frac{dz}{\abs{z}^{d+tq}} .
\]
From inequalities (3.2), (3.3) in \cite[Lemma 3.2]{BT25}, we know that for $L$ large enough (and so $|\abar|$ sufficiently small),
\[
-C L\frac{\omega'(\abs{\abar})}{\abs{\abar}}\abs{z}^2
\le
\delta^2\phi(\cdot,\ybar)(\xbar,z)
\le
-C L\left(\omega''(\abs{\abar})
-\frac{\omega'(\abs{\abar})}{\abs{\abar}}\right)\abs{z}^2,
\]
and
\[
\abs{\ell(z)}+|\delta^2\phi(\cdot,\ybar)(\xbar,z)|\le CL\left(1+\omega'(|\abar|)+|\abar|\omega''(|\abar|)\right)\abs{z}
\le C L\omega'(\abs{\abar})\abs{z}
\]
for \(z\in \B_\del\). Merging these three displays and ($A'_4$), we get
\begin{align*}
\cL_q[\cD_1]v_1(\xbar)&\ge-CL\frac{\omega'(\abs{\abar})}{|\abar|}\int_{\cD_1}a(0,z)(\abs{\ell(z)}+|\delta^2\phi(\cdot,\ybar)
(\xbar,z)|)^{q-2}\frac{|z|^2}{|z|^{d+tq}}\,dz   \\
&\ge-C L^{q-1}\frac{(\omega'(\abs{\abar}))^{q-1}}{\abs{\abar}}
\int_{\cD_1}
\frac{a(0,z)}{\abs{z}^{d+q(t-1)}}\,dz\\
&\ge-C L^{q-1}\frac{(\omega'(\abs{\abar}))^{q-1}}{\abs{\abar}}\left(a(0,0)\int_{\cD_1}
\frac{dz}{\abs{z}^{d+q(t-1)}}+[a]_\alpha\int_{\cD_1}
\frac{dz}{\abs{z}^{d+q(t-1)-\alpha}}\right)\\
&\ge-\frac{C L^{q-1}(\omega'(\abs{\abar}))^{q-1}}{\abs{\abar}}\left(a(0,0)\int_0^{\eps_1|\abar|}r^{q-tq-1}\,dr+
[a]_\alpha\int_0^{\eps_1|\abar|}r^{q-tq+\alpha-1}\,dr\right)\\
&\ge-C L^{q-1}\eps_1^{q(1-t)}(\omega'(\abs{\abar}))^{q-1}
\left(a(0,0)\abs{\abar}^{q(1-t)-1}
+[a]_\alpha\eps_1^\alpha\abs{\abar}^{q(1-t)+\alpha-1}
\right),
\end{align*}
where $C>0$ depends only on $d,q$. The same lower bound holds for the corresponding term \(-\cL_q[\cD_1]v_2(\ybar)\), so the desired result follows.
\end{proof}

Next, we are about to evaluate the integrals on the middle filed $\cD_2$.

\begin{lemma}
\label{lem6-3}
Let \(p,q\ge2\) and \(u\in C^{0,\sigma}(\overline{B}_{\rho_2})\) for some \(\sigma\in[0,1]\). Set \(\del=\eps_1\abs{\abar}\).
\begin{itemize}
\item[(1)] {
Let also the assumptions ($A_2$) ($A'_4$) be in force. Then for every \(\theta_q\in(0,1)\), there exists a constant \(L_0>0\) such that
\begin{align*}
J_3&\ge-C a(0,0)\biggl[\int_{\del}^{\abs{\abar}^{\theta_q}}\left(r^{\sigma(q-2)+1-tq}+\abs{\abar}^{\frac{m-1}{m}\sigma}r^{\sigma(q-2)-tq}
\right)\,dr+\abs{\abar}^{\sigma}\int_{\abs{\abar}^{\theta_q}}^{\rho_0}r^{\sigma(q-2)-tq-1}\,dr\biggr] \\
&\quad-C[a]_\alpha\biggl[\int_{\del}^{\abs{\abar}^{\theta_q}}\left(r^{\sigma(q-2)+1+\alpha-tq}+\abs{\abar}^{\frac{m-1}{m}\sigma}
r^{\sigma(q-2)+\alpha-tq}\right)\,dr\\
&\hspace{7em}
\quad+\abs{\abar}^{\sigma}\int_{\abs{\abar}^{\theta_q}}^{\rho_0}r^{\sigma(q-2)+\alpha-tq-1}\,dr\biggr]
\end{align*}
 for \(L\ge L_0\). Here \(C>0\) depends only on \(d,t,q,m,m_1\) and \([u]_{C^{0,\sigma}(\B_{\rho_2})}\).}

 \smallskip

\item[(2)] {For all \(\theta_p\in(0,1)\), there exists a constant \(L_0>0\) such that
\begin{align*}
I_3&\ge-C\biggl[\int_{\del}^{\abs{\abar}^{\theta_p}}\left(r^{\sigma(p-2)+1-sp}+\abs{\abar}^{\frac{m-1}{m}\sigma}r^{\sigma(p-2)-sp}
\right)\,dr+\abs{\abar}^{\sigma}\int_{\abs{\abar}^{\theta_p}}^{\rho_0}r^{\sigma(p-2)-sp-1}\,dr\biggr]
\end{align*}
 for \(L\ge L_0\). Here \(C>0\) depends only on \(d,s,p,m,m_1\) and \([u]_{C^{0,\sigma}(\B_{\rho_2})}\).}
 \end{itemize}
\end{lemma}

\begin{proof}
Fix \(\theta:=\theta_q\in(0,1)\) for simplicity.  Owing to \eqref{M},
\[
L\omega(\abs{\abar})\le u(\xbar)-u(\ybar)\le 2\|u\|_{L^\infty(B_2)},
\quad\text{and then } \abs{\abar}\to0\quad\text{as }L\to\infty .
\]
Thus, for \(L\) large enough, it holds that
\[
\del<\abs{\abar}^{\theta}<\rho_0 .
\]
Write
\begin{align}
\label{6-3-1}
J_3&=
\cL_q[\cD_2\cap \B_{\abs{\abar}^{\theta}}]v_1(\xbar)-\cL_q[\cD_2\cap \B_{\abs{\abar}^{\theta}}]v_2(\ybar) \nonumber\\
&\quad+\cL_q[\cD_2\cap \B_{\abs{\abar}^{\theta}}^c]v_1(\xbar)
-\cL_q[\cD_2\cap \B_{\abs{\abar}^{\theta}}^c]v_2(\ybar)=:J_{31}+J_{32}
\end{align}
We first handle $J_{31}$. As $\Psi$ has the maximum at $(\xbar,\ybar)$ in \eqref{M} and $u\in C^{0,\sigma}(\overline{B}_{\rho_2})$, there holds
\[
\delta^1 u(\xbar,z)-\delta^1 u(\ybar,z)\ge m_1\delta^1\psi(\xbar,z)
\]
with $\delta^1u(x,z):=u(x)-u(x+z)$, and
\[
\psi(\xbar)\le
\frac{u(\xbar)-u(\ybar)}{m_1}
\le \frac{[u]_{C^{0,\sigma}(\B_{\rho_2})}}{m_1}\abs{\abar}^{\sigma}.
\]
Furthermore,
\[
\abs{\delta^1u(\xbar,z)}+\abs{\delta^1u(\ybar,z)}
\le 2[u]_{C^{0,\sigma}(\B_{\rho_2})}\abs{z}^{\sigma}.
\]
Therefore,
\begin{align*}
J_{31}&=(q-1)\int_{\cD_2\cap B_{|\abar|^\theta}}\int^1_0a(0,z)\left|\delta^1 u(\ybar,z)+\tau(\delta^1 u(\xbar,z)-\delta^1 u(\ybar,z))\right|^{q-2} \\
&\qquad\qquad\qquad\qquad\qquad\times(\delta^1 u(\xbar,z)-\delta^1 u(\ybar,z))\,d\tau\frac{dz}{|z|^{d+tq}} \\
&\ge-C\int_{\cD_2\cap \B_{|\abar|^\theta}}a(0,z)\left(|\delta^1 u(\xbar,z)|+|\delta^1 u(\ybar,z)|\right)^{q-2}|\delta^1 \psi(\xbar,z)|
\frac{dz}{\abs{z}^{d+tq}} \\
&\ge
-C\int_{\cD_2\cap \B_{|\abar|^\theta}}a(0,z)|\delta^1 \psi(\xbar,z)|\abs{z}^{\sigma(q-2)-d-tq}\,dz,
\end{align*}
where $C>0$ depends on $q,m_1,[u]_{C^{0,\sigma}(\B_{\rho_2})}$.
 Taylor's expansion of $\psi$ gives
\[
\abs{\delta^1\psi(\xbar,z)}
\le C\bigl(\abs{\nabla\psi(\xbar)}\abs{z}+\abs{z}^2\bigr),
\qquad
\abs{\nabla\psi(\xbar)}\le 2m\psi(\xbar)^\frac{m-1}{m}
\le C\abs{\abar}^{\frac{m-1}{m}\sigma}.
\]

As a result, it follows from the last three inequalities that
\begin{align}
\label{6-3-2}
J_{31}&\ge
-C\int_{\cD_2\cap \B_{|\abar|^\theta}}a(0,z)
\left(\abs{z}^{\sigma(q-2)+2}+\abs{\abar}^{\frac{m-1}{m}\sigma}\abs{z}^{\sigma(q-2)+1}\right)\frac{dz}{\abs{z}^{d+tq}}
\nonumber\\
&\ge-Ca(0,0)\int_{\B_{|\abar|^\theta}\setminus B_\delta}\left(\abs{z}^{\sigma(q-2)+2}+\abs{\abar}^{\frac{m-1}{m}\sigma}
\abs{z}^{\sigma(q-2)+1}\right)\frac{dz}{\abs{z}^{d+tq}} \nonumber\\
&\quad-C[a]_\alpha\int_{\B_{|\abar|^\theta}\setminus B_\delta}\left(\abs{z}^{\sigma(q-2)+2+\alpha}+\abs{\abar}^{\frac{m-1}{m}\sigma}
\abs{z}^{\sigma(q-2)+1+\alpha}\right)\frac{dz}{\abs{z}^{d+tq}} \nonumber\\
&\ge-Ca(0,0)\int_{\del}^{\abs{\abar}^{\theta}}\left(r^{\sigma(q-2)+1-tq}+\abs{\abar}^{\frac{m-1}{m}\sigma}r^{\sigma(q-2)-tq}\right)\,dr
\nonumber\\
&\quad
-C[a]_\alpha\int_{\del}^{\abs{\abar}^{\theta}}\left(r^{\sigma(q-2)+1+\alpha-tq}+\abs{\abar}^{\frac{m-1}{m}\sigma}r^{\sigma(q-2)+\alpha-tq}
\right)\,dr,
\end{align}
where we have utilized H\"older continuity of $a(\cdot)$ so that $a(0,z)\le a(0,0)+[a]_\alpha|z|^\alpha$.

For \(J_{32}\), utilizing again H\"older continuity of \(u\) and \(a\) gives
\begin{align}
\label{6-3-3}
J_{32}&\ge -C\int_{\cD_2\cap B^c_{|\abar|^\theta}}a(0,z)\left(|\delta^1 u(\xbar,z)|+|\delta^1 u(\ybar,z)|\right)^{q-2}|\delta^1 u(\xbar,z)-\delta^1 u(\ybar,z)|\frac{dz}{\abs{z}^{d+tq}} \nonumber\\
&\ge
-C\int_{\cD_2\cap B^c_{|\abar|^\theta}}a(0,z)\abs{\abar}^{\sigma}\abs{z}^{\sigma(q-2)-d-tq}\,dz
\nonumber\\
&\ge-C\int_{B_{\rho_0}\setminus B_{|\abar|^\theta}}a(0,0)\abs{\abar}^{\sigma}\abs{z}^{\sigma(q-2)-d-tq}\,dz \nonumber\\
&\quad-C[a]_\alpha\int_{B_{\rho_0}\setminus B_{|\abar|^\theta}}\abs{\abar}^{\sigma}\abs{z}^{\sigma(q-2)+\alpha-d-tq}\,dz \nonumber\\
&=-Ca(0,0)\abs{\abar}^{\sigma}
\int_{\abs{\abar}^{\theta}}^{\rho_0}
r^{\sigma(q-2)-tq-1}\,dr
-C[a]_\alpha\abs{\abar}^{\sigma}
\int_{\abs{\abar}^{\theta}}^{\rho_0}
r^{\sigma(q-2)+\alpha-tq-1}\,dr .
\end{align}
Finally, combining the inequalities \eqref{6-3-2}, \eqref{6-3-3} with \eqref{6-3-1} leads to the desired estimate. The evaluation on $I_3$ is easier.
\end{proof}

Eventually, we estimate the integrals on the exterior field $B_{\rho_0}^c$, which is analogous to the treatment of Lemma \ref{lem4-9}.

\begin{lemma}
\label{lem6-4}
Let \(1<p,q<\infty\), and \(u\in C^{0,\sigma}(B_{\rho_2})\) for some \(\sigma\in[0,1]\). Suppose the hypotheses ($A_2$) ($A_3$) ($A'_4$) hold.  Then there exists \(C>0\) such that
\[
\abs{I_4}\le C\max\{\abs{\abar}^\sigma,\abs{\abar}^{\sigma(p-1)}\}
\]
and
\[
\abs{J_4}\le C\max\{|\abar|^\alpha,\abs{\abar}^{\sigma},\abs{\abar}^{\sigma(q-1)}\}
\]
for any $L\ge L_0$, where the constant $C$ depends on $d,s,t,p,q,\|a\|_\alpha,\|u\|_{C^{0,\sigma}(B_{\rho_2})}$, $\rho_1,\rho_2$, $\mathrm{Tail}_{p-1,sp}(u;2)$ and $\mathrm{Tail}_{q-1,tq}(u;2)$. 
\end{lemma}

\begin{proof}
Through translation invariance of $a(\cdot)$ and the choice of test functions $v_1,v_2$, we have
\begin{align}
\label{6-4-1}
J_4
&=\int_{\B_{\rho_0}^c}a(0,z)\Jop q(v_1(\xbar)-v_1(\xbar+z))\frac{dz}{\abs{z}^{d+tq}}-\int_{\B_{\rho_0}^c}a(0,z)\Jop q(v_2(\ybar)-v_2(\ybar+z))\frac{dz}{\abs{z}^{d+tq}} \nonumber\\
&=\int_{\B_{\rho_0}^c(\xbar)}a(\xbar,z)\frac{\Jop q(v_1(\xbar)-v_1(z))}{\abs{\xbar-z}^{d+tq}}\,dz-\int_{\B_{\rho_0}^c(\ybar)}a(\ybar,z)\frac{\Jop q(v_2(\ybar)-v_2(z))}{\abs{\ybar-z}^{d+tq}}\,dz \nonumber\\
&=\int_{B_{\rho_0}^c(\xbar)}a(\xbar,z)J_q(u(\xbar)-u(z))\left(|\xbar-z|^{-d-tq}-|\ybar-z|^{-d-tq}\right)\,dz \nonumber\\
&\quad+\int_{B_{\rho_0}^c(\xbar)}a(\xbar,z)\frac{J_q(u(\xbar)-u(z))}{|\ybar-z|^{d+tq}}\,dz-
\int_{B_{\rho_0}^c(\ybar)}a(\xbar,z)\frac{J_q(u(\xbar)-u(z))}{|\ybar-z|^{d+tq}}\,dz \nonumber\\
&\quad+\int_{B_{\rho_0}^c(\ybar)}a(\overline{x},z)\frac{J_q(u(\overline{x})-u(z))-J_q(u(\overline{y})-u(z))}{|\overline{y}-z|^{d+tq}}\,dz \nonumber\\
&\quad+\int_{B_{\rho_0}^c(\ybar)}(a(\overline{x},z)-a(\overline{y},z))\frac{J_q(u(\overline{y})-u(z))}{|\overline{y}-z|^{d+tq}}\,dz \nonumber\\
&=:J_{41}+J_{42}+J_{43}+J_{44}.
\end{align}
Let us first consider \(J_{44}\). Note \(\bar y\in B_{\frac{3\rho_2}{4}+\frac{\rho_1}{4}}\) and \(\rho_0=\frac{\rho_2-\rho_1}{4}\) with
\(1\le \rho_1<\rho_2\le 2\). We get \(B_{\rho_0}(\bar y)\subset B_2\) and
\[
\abs{z}
\le \abs{\bar y-z}\left(1+\frac{\abs{\bar y}}{\abs{\bar y-z}}\right)
\le
\left(1+\frac{3\rho_2+\rho_1}{\rho_2-\rho_1}\right)
\abs{\bar y-z}
\qquad \text{for } z\in B_2^c .
\]
Therefore, using ($A'_4$) and the boundedness of \(u\), it follows that
\begin{align*}
\abs{J_{44}}&\le
2^{q-1}[a]_{\alpha}\abs{\bar a}^{\alpha}
\int_{B_{\rho_0}^{c}(\bar y)}
\frac{\abs{u(\bar y)}^{q-1}+\abs{u(z)}^{q-1}}
     {\abs{\bar y-z}^{d+tq}}\,dz
\\
&\le
C\norm{u}_{L^\infty(B_{\rho_2})}^{q-1}\abs{\bar a}^{\alpha}
+C\abs{\bar a}^{\alpha}
\int_{B_{\rho_0}^{c}(\bar y)}
\frac{\abs{u(z)}^{q-1}}{\abs{\bar y-z}^{d+tq}}\,dz
\\
&\le
C\norm{u}_{L^\infty(B_2)}^{q-1}\abs{\bar a}^{\alpha}
+C\abs{\bar a}^{\alpha}
\int_{B_2\setminus B_{\rho_0}(\bar y)}
\frac{\abs{u(z)}^{q-1}}{\abs{\bar y-z}^{d+tq}}\,dz
+C\abs{\bar a}^{\alpha}
\int_{B_2^c}
\frac{\abs{u(z)}^{q-1}}{\abs{\bar y-z}^{d+tq}}\,dz
\\
&\le
C\norm{u}_{L^\infty(B_2)}^{q-1}\abs{\bar a}^{\alpha}
+C\abs{\bar a}^{\alpha}
\int_{B_2^c}
\frac{\abs{u(z)}^{q-1}}{\abs{z}^{d+tq}}\,dz
\\
&=
C\left(
\norm{u}_{L^\infty(B_2)}^{q-1}
+\Tail_{q-1,tq}(u;2)^{q-1}
\right)\abs{\bar a}^{\alpha},
\end{align*}
where \(C>0\) depends on \(d,t,q,\rho_1,\rho_2\) and \([a]_{\alpha}\).

Observe \(\abs{\bar a}<\frac{\rho_0}{2}\) if \(L\) is sufficiently large. We can see \(\abs{\bar y-z}\ge \frac{\rho_0}{2}\) when \(\abs{\bar x-z}\ge \rho_0\). It holds for \(z\in B_{\rho_0}^{c}(\bar x)\) that
\[
\left|
\abs{z-\bar x}^{-d-tq}
-\abs{z-\bar y}^{-d-tq}
\right|
\le
C\abs{z-\bar x}^{-d-tq-1}\abs{\bar a}
\le
C\abs{z-\bar x}^{-d-tq}\abs{\bar a}.
\]
So we get
\begin{align*}
\abs{J_{41}}
&\le
C\norm{a}_{\infty}\int_{B_{\rho_0}^{c}(\bar x)}\abs{\bar a}\frac{\abs{u(\bar x)}^{q-1}+\abs{u(z)}^{q-1}}
     {\abs{\bar x-z}^{d+tq}}\,dz  \\
&\le C\left(\norm{u}_{L^\infty(B_2)}^{q-1}+\Tail_{q-1,tq}(u;2)^{q-1}\right)\abs{\bar a}
\end{align*}
with \(C>0\) depending on \(d,t,q,\norm{a}_{\infty},\rho_1,\rho_2\). Denote the symmetric difference
\[
B_{\rho_0}(\bar x)\Delta B_{\rho_0}(\bar y)=\bigl(B_{\rho_0}(\bar x)\setminus B_{\rho_0}(\bar y)\bigr)
\cup\bigl(B_{\rho_0}(\bar y)\setminus B_{\rho_0}(\bar x)\bigr).
\]
We now compute \(J_{42}\),
\begin{align*}
\abs{J_{42}}
&\le C\int_{B_{\rho_0}(\bar x)\Delta B_{\rho_0}(\bar y)}\frac{\abs{u(\bar x)}^{q-1}+\abs{u(z)}^{q-1}}{\abs{z-\bar y}^{d+tq}}\,dz   \\
&\le C\norm{u}_{L^\infty(B_2)}^{q-1}\int_{B_{\rho_0}(\bar x)\Delta B_{\rho_0}(\bar y)}\frac{dz}{\abs{z-\bar y}^{d+tq}}  \\
&\le C\norm{u}_{L^\infty(B_2)}^{q-1}\abs{B_{\rho_0}(\bar x)\Delta B_{\rho_0}(\bar y)}   \\
&=C\norm{u}_{L^\infty(B_2)}^{q-1}\abs{\bar a},
\end{align*}
where \(C>0\) depends on \(d,t,q,\rho_1,\rho_2\).

Finally, we focus on the term \(J_{43}\). For \(q\ge2\), owing to the fact \(u\in C^{0,\sigma}(B_{\rho_2})\), we have
\begin{align*}
\abs{\Jop q(u(\xbar)-u(z))-\Jop q(u(\ybar)-u(z))}
&\le C\left(\norm{u}_{L^\infty(\B_2)}+\abs{u(z)}\right)^{q-2}\abs{u(\xbar)-u(\ybar)}\\
&\le C\left(\norm{u}_{L^\infty(\B_2)}+\abs{u(z)}\right)^{q-2}\abs{\abar}^{\sigma},
\end{align*}
and further
\begin{align*}
\abs{J_{43}}&\le C\|a\|_\infty\abs{\abar}^{\sigma}\int_{B_{\rho_0}^c(\bar y)}\frac{(\|u\|_{L^\infty(B_2)}+\abs{u(z)})^{q-2}}{\abs{\bar y-z}^{d+tq}}\,dz\\
&\le C\|u\|_{L^\infty(B_2)}^{q-2}|\abar|^\sigma+C|\abar|^\sigma\left(\int_{B_{\rho_0}^c(\bar y)}\frac{dz}{\abs{\bar y-z}^{d+tq}}\right)^\frac{1}{q-1}\left(\int_{B_{\rho_0}^c(\bar y)}\frac{\abs{u(z)}^{q-1}}{\abs{\bar y-z}^{d+tq}}\,dz\right)^\frac{q-2}{q-1}\\
&\le C\left(\|u\|_{L^\infty(B_2)}^{q-2}+\Tail_{q-1,tq}(u;2)^{q-2}\right)|\abar|^\sigma.
\end{align*}
Here $C>0$ depends on $d,t,q,\rho_1,\rho_2,\|a\|_\infty,[u]_{C^{0,\sigma}(B_{\rho_2})}$. For \(q\in(1,2)\), one instead uses
\[
\abs{\Jop q(u(\xbar)-u(z))-\Jop q(u(\ybar)-u(z))}
\le 2\abs{u(\xbar)-u(\ybar)}^{q-1}
\le C\abs{\abar}^{\sigma(q-1)}
\]
to derive
$$
\abs{J_{43}}\le C|\abar|^{\sigma(q-1)}.
$$

Putting together the estimates of $J_{41}-J_{44}$ with \eqref{6-4-1}, we arrive at
$$
|J_4|\le C\max\{|\bar a|^\alpha,|\bar a|^\sigma,|\bar a|^{\sigma(q-1)}\}
$$
with $C>0$ depending on $d,t,q,\|a\|_\infty,[a]_\alpha,\|u\|_{L^\infty(B_{\rho_2})},[u]_{C^{0,\sigma}(B_2)}$, $\rho_2-\rho_1$, and $\mathrm{Tail}_{q-1,tq}(u;2)$. The bound on $I_4$ follows in the same way with \(a\equiv1\).
\end{proof}

\subsection{Combination of these estimates}

\mbox{}\par
\medskip

Now letting $\omega(\rho)=\rho^\beta$, $\theta:=\theta_p=\theta_q$ in Lemma \ref{lem6-3} and gathering the bounds on \(I_1-I_4\) and \(J_1-J_4\) with \eqref{ve}, for the case $2\le p\le q$ we obtain
\begin{align}
\label{ce}
0&\ge C L^{p-1}\del_0^{p_s}\abs{\abar}^{\beta(p-1)-sp}-C L^{p-1}\eps_1^{p(1-s)}\abs{\abar}^{\beta(p-1)-sp}-C\max\{\abs{\abar}^{\alpha},\abs{\abar}^{\sigma}\} \nonumber\\
&\quad-C\left[\int_{\del}^{\abs{\abar}^{\theta}}\left(r^{\sigma(p-2)+1-sp}+\abs{\abar}^{\frac{m-1}{m}\sigma}r^{\sigma(p-2)-sp}\right)\,dr
+\abs{\abar}^{\sigma}\int_{\abs{\abar}^{\theta}}^{\rho_0}r^{\sigma(p-2)-sp-1}\,dr\right]
\nonumber\\
&\quad
+C a(0,0)
\Bigg[\del_0^{q_t}L^{q-1}\abs{\abar}^{\beta(q-1)-tq}-L^{q-1}\eps_1^{q(1-t)}\abs{\abar}^{\beta(q-1)-tq} \notag\\
&\qquad\qquad\qquad
-\int_{\del}^{\abs{\abar}^{\theta}}\left(r^{\sigma(q-2)+1-tq}+\abs{\abar}^{\frac{m-1}{m}\sigma}r^{\sigma(q-2)-tq}\right)\,dr-
\abs{\abar}^{\sigma}\int_{\abs{\abar}^{\theta}}^{\rho_0}r^{\sigma(q-2)-tq-1}\,dr\Bigg] \notag\\
&\quad-C[a]_\alpha L^{q-1}\left(\del_0^{q_t+\alpha}+\eps_1^{q(1-t)}\right)\abs{\abar}^{\beta(q-1)+\alpha-tq}-C[a]_\alpha \abs{\abar}^{\sigma}\int_{\abs{\abar}^{\theta}}^{\rho_0}r^{\sigma(q-2)+\alpha-tq-1}\,dr  \notag\\
&\quad-C[a]_\alpha\int_{\del}^{\abs{\abar}^{\theta}}\left(r^{\sigma(q-2)+1+\alpha-tq}+\abs{\abar}^{\frac{m-1}{m}\sigma}
r^{\sigma(q-2)+\alpha-tq}\right)\,dr,
\end{align}
where $p_s:=\frac{d-1}{2}+p(1-s),q_t:=\frac{d-1}{2}+q(1-t)$ and $\sigma\in[0,1]$.

In what follows, we will exploit the distance condition \eqref{dis} to simplify the display \eqref{ce}.
First, for $0<r<1$ we have 
\[
r^{\sigma(q-2)+\alpha-tq}\le r^{\sigma(p-2)-sp}\quad\Leftrightarrow\quad tq\le sp+\alpha+\sigma(q-p),
\]
which is ensured by \eqref{dis} and $p\le q$. Hence, it yields that
\begin{align*}
&\quad\int_{\del}^{\abs{\abar}^{\theta}}\left(r^{\sigma(q-2)+1+\alpha-tq}+\abs{\abar}^{\frac{m-1}{m}\sigma}
r^{\sigma(q-2)+\alpha-tq}\right)\,dr\\
&\le\int_{\del}^{\abs{\abar}^{\theta}}\left(r^{\sigma(p-2)+1-sp}+\abs{\abar}^{\frac{m-1}{m}\sigma}
r^{\sigma(p-2)-sp}\right)\,dr
\end{align*}
and moreover by \(\rho_0\le1\),
\[
\int_{\abs{\abar}^{\theta}}^{\rho_0}r^{\sigma(q-2)+\alpha-tq-1}\,dr\le\int_{\abs{\abar}^{\theta}}^{\rho_0}
r^{\sigma(p-2)-sp-1}\,dr.
\]

Next, it is easy to see
$$
C L^{p-1}\eps_1^{p(1-s)}\abs{\abar}^{\beta(p-1)-sp}\ll C L^{p-1}\del_0^{p_s}\abs{\abar}^{\beta(p-1)-sp},
$$
$$
L^{q-1}\eps_1^{q(1-t)}\abs{\abar}^{\beta(q-1)-tq}\ll \del_0^{q_t}L^{q-1}\abs{\abar}^{\beta(q-1)-tq}
$$
and
$$
C[a]_\alpha L^{q-1}\eps_1^{q(1-t)}\abs{\abar}^{\beta(q-1)+\alpha-tq}\ll C[a]_\alpha L^{q-1}\del_0^{q_t+\alpha}\abs{\abar}^{\beta(q-1)+\alpha-tq}
$$
via choosing small enough $\eps_1>0$ that depends on $\delta_0$ (not depend on $L,|\abar|$). We would like to justify
\begin{align}
\label{lp}
&C[a]_\alpha L^{q-1}\del_0^{q_t+\alpha}\abs{\abar}^{\beta(q-1)+\alpha-tq}\ll C L^{p-1}\del_0^{p_s}\abs{\abar}^{\beta(p-1)-sp} \notag\\
\Leftrightarrow & \ C[a]_\alpha L^{q-p}\abs{\abar}^{\beta(q-p)+\alpha+sp-tq}\ll\del_0^{p-q+tq-sp-\alpha}.
\end{align}
As a matter of fact, from \eqref{M} and $\omega(\rho):=\omega_\beta(\rho)=\rho^\beta$ we can discover
\begin{align*}
L\omega(\abs{\abar})\le u(\xbar)-u(\ybar)    \ \ \
\Rightarrow \ \ \ L\abs{\abar}^\beta\le 2\norm{u}_{L^\infty(\B_2)},
\end{align*}
and thus
\[
C[a]_\alpha L^{q-p}\abs{\abar}^{\beta(q-p)-tq+sp+\alpha}\le C\norm{u}_{L^\infty(\B_2)}^{q-p}\abs{\abar}^{sp+\alpha-tq}.
\]
If \(tq<sp+\alpha\), the last factor $\abs{\abar}^{sp+\alpha-tq}$ is sufficiently small for \(L\) large, so that the display \eqref{lp} is valid. In the borderline case \(tq=sp+\alpha\), the \eqref{lp} becomes
\[
C[a]_\alpha\bigl(L\abs{\abar}^{\beta}\bigr)^{q-p}\ll \delta_0^{p-q},
\]
which is assured by letting \(p<q\) and taking \(\del_0<1\) sufficiently small. Here we keep in mind $C[a]_\alpha\bigl(L\abs{\abar}^{\beta}\bigr)^{q-p}\le C[a]_\alpha\|u\|_{L^\infty(B_2)}^{q-p}$.

At this stage, the inequality \eqref{ce} turns into
\begin{align}
\label{so}
0&\ge C\del_0^{p_s}L^{p-1}\abs{\abar}^{\beta(p-1)-sp}-C\max\{\abs{\abar}^{\alpha},\abs{\abar}^{\sigma}\} \notag\\
&\quad-C\left[\int_{\del}^{\abs{\abar}^{\theta}}\left(r^{\sigma(p-2)+1-sp}+\abs{\abar}^{\frac{m-1}{m}\sigma}r^{\sigma(p-2)-sp}\right)\,dr
+\abs{\abar}^{\sigma}\int_{\abs{\abar}^{\theta}}^{\rho_0}r^{\sigma(p-2)-sp-1}\,dr\right] \notag\\
&\quad+C a(0,0)
\Bigg[\del_0^{q_t}L^{q-1}\abs{\abar}^{\beta(q-1)-tq}-
\abs{\abar}^{\sigma}\int_{\abs{\abar}^{\theta}}^{\rho_0}r^{\sigma(q-2)-tq-1}\,dr \notag\\
&\qquad\qquad\qquad
-\int_{\del}^{\abs{\abar}^{\theta}}\left(r^{\sigma(q-2)+1-tq}+\abs{\abar}^{\frac{m-1}{m}\sigma}r^{\sigma(q-2)-tq}\right)\,dr\Bigg].
\end{align}
Subsequently, we may deduce the H\"older continuity of viscosity solutions based on this inequality.

\section{Lipschitz regularity in the case $2\le p\le q$}
\label{sec7}

This section is dedicated to demonstrating the Lipschitz continuity and improved H\"older regularity of viscosity solutions to \eqref{main} in the superquadratic scenario, whose proof is completed by applying the bootstrap argument. 




We are going to show that the viscosity solution to \eqref{main} is locally Lipschitz continuous for $\frac{sp}{p-1}\ge1,\frac{tq}{q-1}\ge1$.

\noindent\textbf{Step 1}: let us prove $u$ is $\gamma$-H\"older continuous for any $\gamma<\gamma_0:=\min\left\{1,\frac{sp}{p-1},\frac{tq}{q-1}\right\}$.

\begin{proposition}
\label{pro7-1}
Suppose that the conditions on $a(\cdot)$, $(A_1)-(A_3),(A'_4)$, and the distance requirement \eqref{dis} with $0<s\le t<1$ hold true. For any viscosity solution $u$ to \eqref{main}, there holds 
$$
u\in C^{0,\gamma}_{\rm loc}(B_2) \ \ \ \text{for any } \ \gamma<\gamma_0:=\min\left\{1,\frac{sp}{p-1},\frac{tq}{q-1}\right\}.
$$
\end{proposition}

\begin{proof}
Fix \(1\le \rho_1<\rho_2\le 2\). Let \(u\in C^{0,\sigma}(B_2)\) with some \(\sigma\in[0,\gamma)\). Denote
\[
  \sigma_1 =\min\left\{\gamma,\sigma+\frac{1}{2(q-1)},\sigma+\frac{sp-(p-2)\gamma}{2(p-1)}, \sigma+\frac{tq-(q-2)\gamma}{2(q-1)}\right\}.
\]
With the choice of \(\omega_\beta(\rho):=\omega_{\sigma_1}(\rho)\), our aim is to get the following display from \eqref{so}
\begin{align}
\label{s-1}
0\ge&\ C_{\delta_0}L^{p-1}|\bar a|^{\sigma_1(p-1)-sp} - C\max\{|\bar a|^\alpha,|\bar a|^\sigma\} - C|\bar a|^{\sigma+\theta(\sigma(p-2)-sp)} \notag\\
&- C\int_\delta^{|\bar a|^\theta}\left( r^{\sigma(p-2)+1-sp} + |\bar a|^{\frac{m-1}{m}\sigma}r^{\sigma(p-2)-sp}\right)\,dr \notag\\
  &+ a(0,0)\Bigg[C_{\delta_0}L^{q-1}|\bar a|^{\sigma_1(q-1)-tq} - C|\bar a|^{\sigma+\theta(\sigma(q-2)-tq)}\notag\\
  &\qquad\qquad\qquad
    -C \int_\delta^{|\bar a|^\theta}\left( r^{\sigma(q-2)+1-tq} + |\bar a|^{\frac{m-1}{m}\sigma}r^{\sigma(q-2)-tq}\right)\,dr\Bigg]
  >0
\end{align}
with \(\theta\in(0,1)\) and $\delta=\eps_1|\abar|$, where by the fact that $\sigma<\min\left\{1,\frac{sp}{p-2},\frac{tq}{q-2}\right\}$, we have employed the inequalities
$$
\abs{\abar}^{\sigma}\int_{\abs{\abar}^{\theta}}^{\rho_0}r^{\sigma(p-2)-sp-1}\,dr\le\frac{1}{sp-\gamma(p-2)}|\bar a|^{\sigma+\theta(\sigma(p-2)-sp)}
$$
and
$$
\abs{\abar}^{\sigma}\int_{\abs{\abar}^{\theta}}^{\rho_0}r^{\sigma(q-2)-tq-1}\,dr\le\frac{1}{tq-\gamma(q-2)}|\bar a|^{\sigma+\theta(\sigma(q-2)-tq)}.
$$

Next we shall justify the claim \eqref{s-1}. First, note a simple fact that
\begin{equation}
\label{s-2}
  C\max\{|\bar a|^\alpha,|\bar a|^\sigma\}\le C
\end{equation}
via \(|\bar a|<1\) and \(\alpha>0\), \(\sigma\ge0\). Moreover, it is easy to see
\[
  \sigma_1<\sigma+\frac{tq-(q-2)\gamma}{q-1} \le \sigma+\frac{tq-(q-2)\sigma}{q-1} =\frac{\sigma+tq}{q-1},
\]
so we can take small enough \(\theta\in(0,1)\) to get
\[
  \sigma_1(q-1)\le\sigma+tq+\theta\bigl(\sigma(q-2)-tq\bigr).
\]
Thus,
\begin{equation}
\label{s-3}
  |\bar a|^{\sigma+\theta(\sigma(q-2)-tq)} \le |\bar a|^{\sigma_1(q-1)-tq}
\end{equation}
and similarly,
\begin{equation}
\label{s-4}
  |\bar a|^{\sigma+\theta(\sigma(p-2)-sp)} \le|\bar a|^{\sigma_1(p-1)-sp}.
\end{equation}

It follows from \(\sigma_1\le\sigma+\frac{1}{2(q-1)}\) that
\begin{equation}
\label{s-5}
  \sigma_1(q-1)-\sigma<\sigma(q-2)+1
 \end{equation}
and further
\[
  \sigma_1(q-1)-\frac{m}{m-1}\sigma\le \sigma(q-2)+1
\]
by selecting \(m\ge 3\) sufficiently large. As a consequence, noting \(r<1\) we have
\begin{align*}
  |\bar a|^{\frac{m-1}{m}\sigma}\int_\delta^{|\bar a|^\theta} r^{\sigma(q-2)-tq}\,dr
  &\le|\bar a|^{\frac{m-1}{m}\sigma}\int_\delta^{|\bar a|^\theta} r^{\sigma_1(q-1)-\frac{m}{m-1}\sigma-tq-1}\,dr \\
  &\le|\bar a|^{\frac{m-1}{m}\sigma} \int_\delta^1r^{\sigma_1(q-1)-\frac{m}{m-1}\sigma-tq-1}\,dr\\
  &\le\left( tq-\sigma_1(q-1)+\frac{m}{m-1}\sigma \right)^{-1} |\bar a|^{\frac{m-1}{m}\sigma}\delta^{\sigma_1(q-1)-\frac{m}{m-1}\sigma-tq},
\end{align*}
where we observed $\sigma_1(q-1)-tq\le \gamma(q-1)-tq<0$. Recall $\delta=\eps_1|\abar|$. Then we derive
\begin{equation}
\label{s-6}
  |\bar a|^{\frac{m-1}{m}\sigma}\int_{\delta}^{|\bar a|^\theta} r^{\sigma(q-2)-tq}\,dr\le\frac{\varepsilon_1^{\sigma_1(q-1)-\frac{m}{m-1}\sigma-tq}}{tq-\gamma(q-1)+\frac{m}{m-1}\sigma}|\bar a|^{\sigma_1(q-1)-tq}.
  \end{equation}
From \eqref{s-5}, it holds that
\[
  \sigma(q-2)+1-tq > \sigma_1(q-1)-\sigma-tq > \sigma_1(q-1)-1-tq .
\]
Thus,
\begin{align}
\label{s-7}
  \int_{\delta}^{|\bar a|^\theta}r^{\sigma(q-2)+1-tq}\,dr &< \int_{\delta}^{|\bar a|^\theta}r^{\sigma_1(q-1)-tq-1}\,dr\notag\\
  &\le\frac{1}{\sigma_1(q-1)-tq} r^{\sigma_1(q-1)-tq}\bigg|_{\varepsilon_1|\bar a|}^{1} \notag\\
  &\le \frac{\varepsilon_1^{\sigma_1(q-1)-tq}}{tq-\gamma(q-1)}|\bar a|^{\sigma_1(q-1)-tq}.
\end{align}

On the other hand, by $\sigma_1\le\sigma+\frac{1}{2(q-1)}$ we can find that $\sigma_1<\sigma+\frac{1}{p-1}$ and so
\[
  \sigma_1(p-1)-1-sp<\sigma(p-2)+1-sp
\]
along with
\[
  \sigma_1(p-1)-\frac{m}{m-1}\sigma-sp\le\sigma(p-2)-sp+1,
\]
provided $m\ge3$ is sufficiently large. In a similar way, we arrive at
\begin{equation}
\label{s-8}
  \int_{\delta}^{|\bar a|^\theta}\left(r^{\sigma(p-2)+1-sp}+|\bar a|^{\frac{m-1}{m}\sigma}r^{\sigma(p-2)-sp}\right)\,dr\le
  C_{\varepsilon_1}|\bar a|^{\sigma_1(p-1)-sp}.
\end{equation}

Combining \eqref{s-2}--\eqref{s-4} and \eqref{s-6}--\eqref{s-8} with \eqref{so}, we obtain
\begin{align*}
  0&\ge C_{\delta_0}L^{p-1}|\bar a|^{\sigma_1(p-1)-sp}- C|\bar a|^{\sigma_1(p-1)-sp}- C_{\varepsilon_1}|\bar a|^{\sigma_1(p-1)-sp} - C \\
  &\quad+  a(0,0) \left(C_{\delta_0}L^{q-1}|\bar a|^{\sigma_1(q-1)-tq} -C |\bar a|^{\sigma_1(q-1)-tq}- C_{\varepsilon_1}|\bar a|^{\sigma_1(q-1)-tq}\right) \\
  &=|\bar a|^{\sigma_1(p-1)-sp}\left( C_{\delta_0}L^{p-1}-C-C_{\varepsilon_1}\right)- C+a(0,0)|\bar a|^{\sigma_1(q-1)-tq} \left(C_{\delta_0}L^{q-1}-C-C_{\varepsilon_1}\right).
\end{align*}
Recalling that \(|\bar a|\to 0\) as \(L\to\infty\), and
\[
  \sigma_1(p-1)-sp\le\gamma(p-1)-sp<0,
\]
we will conclude that, by choosing \(L\) large enough,
\begin{equation}
\label{s-8-1}
   a(0,0)|\bar a|^{\sigma_1(q-1)-tq}\left(C_{\delta_0}L^{q-1}-C-C_{\varepsilon_1}\right)\ge 0
\end{equation}
and
\[
  |\bar a|^{\sigma_1(p-1)-sp}\left( C_{\delta_0}L^{p-1}-C-C_{\varepsilon_1}\right)-C >0 .
\]
At this point, the inequality \eqref{s-1} is proved to be valid, and this leads to a
contradiction. In other words, there holds that
\[
  \sup_{B_2\times B_2}\Psi(x,y)\le 0,
\]
and then
\[
  u(x)-u(y)  \le L\omega_{\sigma_1}(|x-y|)=L|x-y|^{\sigma_1}\qquad \text{for any } x,y\in B_{\rho_1} .
\]
This means that \(u\) is \(\sigma_1\)-H\"older continuous in \(B_{\rho_1}\). A standard bootstrap argument implies
\(u\in C_{\mathrm{loc}}^{0,\gamma}(B_2)\).
\end{proof}

\medskip

\noindent\textbf{Step 2}: we show $u\in C_{\mathrm{loc}}^{0,\gamma_0}(B_2)$ by distinguishing two cases that $\gamma_0<1$ and $\gamma_0=1$, which is stated by the upcoming proposition.

\begin{proposition}
\label{pro7-2}
Under the hypotheses of Proposition \ref{pro7-1}, there holds 
$$
u\in C^{0,\gamma_0}_{\rm loc}(B_2) \ \ \ \text{with } \ \gamma_0:=\min\left\{1,\frac{sp}{p-1},\frac{tq}{q-1}\right\}.
$$
\end{proposition}

\noindent \textbf{Step 2-1}: we give the proof of Proposition \ref{pro7-2} for $\gamma_0<1$.

\medskip

\noindent\textbf{\emph{Proof of Proposition \ref{pro7-2} for $\gamma_0<1$}.}
Suppose $1>\gamma_0=\frac{sp}{p-1}$. This means $\frac{sp}{p-1}\le \frac{tq}{q-1}$. It is known from Step 1 that $u\in C_{\mathrm{loc}}^{0,\gamma}(B_2)$ for all $\gamma<\gamma_0$. Thus we let \(\omega_\beta(\rho):=\omega_{\gamma_0}(\rho)\) and $u\in C_{\mathrm{loc}}^{0,\sigma_2}(B_2)$ with $\sigma_2<\gamma_0$ a number to be determined. Then \eqref{so} can be rewritten as
\begin{align}
\label{s2-1}
  0&\ge C_{\delta_0}L^{p-1}|\bar a|^{\gamma_0(p-1)-sp} - C\max\{|\bar a|^\alpha,|\bar a|^{\sigma_2}\}- C|\bar a|^{\sigma_2+\theta(\sigma_2(p-2)-sp)} \notag\\
  &\quad- C\int_{\delta}^{|\bar a|^\theta} \left( r^{\sigma_2(p-2)+1-sp} +|\bar a|^{\frac{m-1}{m}\sigma_2} r^{\sigma_2(p-2)-sp}\right)\,dr  \notag\\
  &\quad+ a(0,0)\Bigg[C_{\delta_0}L^{q-1}|\bar a|^{\gamma_0(q-1)-tq} -C|\bar a|^{\sigma_2+\theta(\sigma_2(q-2)-tq)} \notag\\
  &\qquad\qquad
    -C\int_{\delta}^{|\bar a|^\theta}
    \left(r^{\sigma_2(q-2)+1-tq}+|\bar a|^{\frac{m-1}{m}\sigma_2}r^{\sigma_2(q-2)-tq}\right)\,dr\Bigg]
  \notag\\
  &=: I_1+a(0,0)I_2 .
\end{align}

To attain a contradiction, we are ready to carefully pick the parameter \(\sigma_2\). Remembering $\gamma_0=\frac{sp}{p-1}<1$,
we can take such \(\sigma_2<\gamma_0\) that
\[
  \sigma_2(p-2)-sp =\left(\sigma_2-\frac{sp}{p-1}\right)(p-1)-\sigma_2>-1,
\]
which implies that
\[
  \int_{\delta}^{|\bar a|^\theta}\left(r^{\sigma_2(p-2)+1-sp}+|\bar a|^{\frac{m-1}{m}\sigma_2}r^{\sigma_2(p-2)-sp}\right)\,dr\le C
\]
with \(C>0\) being a universal constant independent of \(|\bar a|\). Moreover, we may choose \(\theta\in(0,1)\) so small that
\[
  \sigma_2+\theta\bigl(\sigma_2(p-2)-sp\bigr)>0,
\]
concluding
\[
  |\bar a|^{\sigma_2+\theta(\sigma_2(p-2)-sp)}<1 .
\]
Hence, we derive
\begin{equation}
\label{s2-2}
  I_1\ge C_{\delta_0}L^{p-1}-C>0,
\end{equation}
through letting \(L\) be sufficiently large.

On the other hand, we notice a simple fact that
\[
  \gamma_0(q-1)-tq =\left(\frac{sp}{p-1}-\frac{tq}{q-1}\right)(q-1)\le 0 .
\]
In the situation $\gamma_0=\frac{sp}{p-1}=\frac{tq}{q-1}$, we can immediately infer \(I_2>0\) as above. Now let us concentrate on the
case $\gamma_0=\frac{sp}{p-1}<\frac{tq}{q-1}$. Via taking \(\theta\in(0,1)\) small enough, we also have
\[
  \sigma_2+\theta\bigl(\sigma_2(q-2)-tq\bigr)\ge0 .
\]
This means
\[
  |\bar a|^{\sigma_2+\theta(\sigma_2(q-2)-tq)}\le 1 .
\]
We proceed to selecting \(\sigma_2<\gamma_0\) so close to $\gamma_0$ that
\begin{align*}
  &\ 
  \sigma_2(q-1)+1\ge\gamma_0(q-1)-\frac{1}{m-1}\sigma_2  \\
\Leftrightarrow & \
   \sigma_2(q-2)-tq\ge\gamma_0(q-1)-\frac{m}{m-1}\sigma_2-tq-1 ,
  \end{align*}
and
\begin{align*}
  &\
  \sigma_2(q-2)-tq+1\ge\gamma_0(q-2)-tq+\frac{sp}{p-1}-1 \\
\Leftrightarrow & \
  \sigma_2(q-2)-tq+1\ge\gamma_0(q-1)-tq-1.
\end{align*}
 Hence, by recalling \(\delta=\varepsilon_1|\bar a|\), we arrive at
\begin{align*}
  &\quad\int_{\delta}^{|\bar a|^\theta}
  \left(r^{\sigma_2(q-2)+1-tq}+|\bar a|^{\frac{m-1}{m}\sigma_2}r^{\sigma_2(q-2)-tq}\right)\,dr \\
  &\le\int_{\delta}^{1}\left(r^{\gamma_0(q-1)-tq-1}+|\bar a|^{\frac{m-1}{m}\sigma_2}r^{\gamma_0(q-1)-\frac{m}{m-1}\sigma_2-tq-1}
  \right)\,dr  \\
  &\le
  \left(\frac{\varepsilon_1^{\gamma_0(q-1)-tq}}{tq-\gamma_0(q-1)}+ \frac{\varepsilon_1^{\gamma_0(q-1)-\frac{m}{m-1}\sigma_2-tq}}
    {tq-\gamma_0(q-1)}\right)|\bar a|^{\gamma_0(q-1)-tq}.
\end{align*}
Therefore, it yields that
\begin{equation}
\label{s2-3}
  I_2\ge|\bar a|^{\gamma_0(q-1)-tq}\left(C_{\delta_0}L^{q-1}-C_{\varepsilon_1}\right)- C>0,
\end{equation}
when we choose large enough \(L\).

At this stage, it follows from \eqref{s2-1}-\eqref{s2-3} that
\[
  0\ge I_1+a(0,0)I_2>0.
\]
This leads to a contradiction. Thus we verify $u\in C_{\mathrm{loc}}^{0,\gamma_0}(B_2)$, as argued before.

Assume $\gamma_0=\frac{tq}{q-1}<1$. This indicates $\frac{tq}{q-1}\le\frac{sp}{p-1}$. This scenario is specular to the above, so we omit the details here.    \hfill $\square$

\medskip

\noindent\textbf{Step 2-2}: we present the proof of Proposition \ref{pro7-2} in the case $\gamma_0=1\le \min\left\{\frac{sp}{p-1},\frac{tq}{q-1}\right\}$.

Before proceeding with the rigorous verification of this Lipschitz result, we have to carry out a series of preliminary preparations. To prove Lipschitz property for $u$, we need employ the upcoming result that modifies the definition of viscosity solutions to \eqref{main}.

\begin{lemma}
\label{lemdef}
Assume that $\phi\in C^2(B_r(x_0))$ is a test function touching $u$ at the point $x_0$ from above in $B_r(x_0)$. Let $\delta, \overline{\delta}\in (0, r)$ and define
\[
u_\delta(x)=\left\{ \begin{array}{ll}
\phi(x) &  x\in B_\delta(x_0),
\\[2mm]
u(x) & \text{otherwise},
\end{array}
\right.
\quad and\quad
u_{\overline{\delta}}(x)=\left\{ \begin{array}{ll}
\phi(x) &  x\in B_{\overline\delta}(x_0), \\[2mm]
u(x) & \text{otherwise}.
\end{array}
\right.
\]
Then it holds
$$
\mathcal{L}_p[\Rd] u_\delta(x_0)+\mathcal{L}_q[\mathbb{R}^d]u_{\overline{\delta}}(x_0)\leq 0.
$$
Analogous result holds for supersolutions.
\end{lemma}

Following the proof of \cite[Lemma 2.2]{BS} verbatim, we could also prove Lemma \ref{lemdef} that does not rely on which of $s$ or $t$ is greater. This lemma is applied to make a more delicate choice of test functions related to viscosity solutions.

For $\delta,\overline\delta\in\Big(0,\frac{\rho_2-\rho_1}{4}\Big)$ to be chosen later, we take test functions as
\[
v_1(z)=
\begin{cases}
\bar\phi(z) & z\in B_\delta(\xbar),\\
u(z) & \text{otherwise},
\end{cases}
\ \ \text{and} \ \
v_2(z)=
\begin{cases}
\widetilde{\phi}(z) & z\in B_{\delta}(\ybar),\\
u(z) & \text{otherwise}
\end{cases}
\]
as well as
\[
\overline v_1(z)=
\begin{cases}
\bar\phi(z) & z\in B_{\overline{\delta}}(\xbar),\\
u(z) & \text{otherwise},
\end{cases}
\ \ \text{and} \ \
\overline v_2(z)=
\begin{cases}
\widetilde{\phi}(z) & z\in B_{\overline{\delta}}(\ybar),\\
u(z) & \text{otherwise}.
\end{cases}
\]
Here $\bar\phi(z):=\phi(z,\ybar)+u(\xbar)-\phi(\xbar,\ybar)$ and $\widetilde{\phi}(z):=-\phi(\xbar,z)+u(\ybar)+\phi(\xbar,\ybar)$.
Now the viscosity inequality \eqref{ve} may be rewritten as
\begin{align}
\label{ve2}
0&\ge\cL_p[\cC]v_1(\xbar)-\cL_p[\cC]v_2(\ybar)+\cL_p[\cD_1]v_1(\xbar)-\cL_p[\cD_1]v_2(\ybar) \nonumber\\
&\quad+\cL_p[\cD_2]v_1(\xbar)-\cL_p[\cD_2]v_2(\ybar)+\cL_p[B^c_{\rho_0}]v_1(\xbar)-\cL_p[B^c_{\rho_0}]v_2(\ybar) \nonumber\\
&\quad+\cL_q[\cC]\bar v_1(\xbar)-\cL_q[\cC]\bar v_2(\ybar)+\cL_q[\widetilde{\cD}_1]\bar v_1(\xbar)-\cL_q[\widetilde{\cD}_1]\bar v_2(\ybar) \nonumber\\
&\quad+\cL_q[\widetilde{\cD}_2]\bar v_1(\xbar)-\cL_q[\widetilde{\cD}_2]\bar v_2(\ybar)+\cL_q[B^c_{\rho_0}]\bar v_1(\xbar)-\cL_q[B^c_{\rho_0}]\bar v_2(\ybar) \nonumber\\
&=:I'_1+I'_2+I'_3+I'_4+J'_1+J'_2+J'_3+J'_4.
\end{align}
Here the domains $\cC,\cD_1,\cD_2,B_{\rho_0}$ is given in the beginning of Section \ref{sec6}, and
\[
  \widetilde{\cD}_1=B_{\bar\delta}\cap \cC^c,\qquad
  \widetilde{\cD}_2=B_{\rho_0}\setminus(\cC\cup \widetilde{D}_1),
\]
where $\delta,\bar\delta\ll \delta_0|\bar a|\ll \rho_0:=\frac{\rho_2-\rho_1}{4}$. Specifically, we set
\[
  \delta_0=\delta_1\log^{-2}|\bar a|,\ \ \
  \delta=\varepsilon_1|\bar a|\log^{-2\chi}|\bar a|, \ \ \
  \bar\delta =\varepsilon_1|\bar a|\log^{-2\overline{\chi}}|\bar a|
\]
with \(\delta_1,\varepsilon_1\in(0,\frac12)\) to be chosen, where
\[
  \chi=\frac{\frac{d+1}{2}+p-sp}{p-sp},
  \qquad
  \overline{\chi} =\frac{\frac{d+1}{2}+q-tq}{q-tq}.
\]

The estimates on $I'_4,J'_4$ are identical to Lemma \ref{lem6-4}. Corresponding to Lemmas \ref{lem6-1}--\ref{lem6-3}, we have the three results below for $\widetilde\omega$.

\begin{lemma}
\label{lem7-1}
Set \(p,q\ge2\) and $\omega(\rho)=\widetilde{\omega}\left(\frac{r_0}{3}\rho\right)$. Let \(r_0>0\) be a small number such that for each \(r\in (0,r_0]\),
\[
  \frac{r}{2}  \leq \widetilde{\omega}(r)\leq r,\ \ \frac12 \leq \widetilde{\omega}'(r)\leq 1, \ \ \text{and} \ \
  -2(r\log^2 r)^{-1}\leq \widetilde{\omega}''(r)\leq -(r\log^2 r)^{-1}.
\]
Then there are positive constants \(L_0\) and \(\delta_1\), independent of \(\bar a\), such that for $\delta_0=\delta_1\log^{-2}|\bar a|$
\begin{equation}
\label{7-1-1}
  I'_1 \geq C\delta_1^{p_s} L^{p-1}|\bar a|^{p(1-s)-1}\left(\log^2|\bar a|\right)^{-\xi},
\end{equation}
and if the conditions $(A_1), (A_2), (A'_4)$ hold,
\begin{equation}
\label{7-1-2}
  J'_1\geq C\delta_1^{q_t}L^{q-1}|\bar a|^{q(1-t)-1} \left(\log^2|\bar a|\right)^{-\bar\xi}\left[a(0,0)-[a]_\alpha
    \bigl(|\bar a|\log^{-2}|\bar a|\bigr)^\alpha\right]
\end{equation}
for any \(L\geq L_0\). Here $p_s=\frac{d-1}{2}+p(1-s),q_t=\frac{d-1}{2}+q(1-t)$ and $\xi=\frac{d+1}{2}+p(1-s)$, $\overline{\xi}=\frac{d+1}{2}+q(1-t)$ and \(C>0\) depends on \(d,s,t,p,q\).
\end{lemma}

\begin{proof}
Let us prove \eqref{7-1-2}. From \eqref{l3-1}, \(\omega'(|\bar a|)\approx 1\) and the estimate in Lemma \ref{bte} (2)
\[
  \frac{L}{c} \frac{|z|^2}{|\bar a|\log^2|\bar a|}\leq\delta^2\phi(\cdot,\bar y)(\bar x,z)\leq cL\frac{|z|^2}{|\bar a|\log^2|\bar a|}
  \qquad \text{for } z\in \cC ,
\]
where \(c>0\) depends only on \(d\), we have
\[
\begin{aligned}
  \cL_q[\cC]v_1(\bar x) &\geq\frac{CL^{q-1}}{|\bar a|\log^2|\bar a|}\left(a(0,0)\int_{\cC} \frac{dz}{|z|^{d+q(t-1)}}
    -[a]_\alpha\int_{\cC}\frac{dz}{|z|^{d+q(t-1)-\alpha}}\right)                                                \\[0.5em]
  &\geq\frac{CL^{q-1}}{|\bar a|\log^2|\bar a|}\left(a(0,0)\delta_0^{q(1-t)+\frac{d-1}{2}}|\bar a|^{q(1-t)}-[a]_\alpha
    \delta_0^{q(1-t)+\frac{d-1}{2}+\alpha}|\bar a|^{q(1-t)+\alpha}\right)                                                \\[0.5em]
  &\geq\frac{C\delta_1^{\frac{d-1}{2}+q(1-t)}L^{q-1}}{|\bar a|^{1+q(t-1)}\left(\log^2|\bar a|\right)^{\bar\xi}}
  \left(a(0,0)-[a]_\alpha|\bar a|^\alpha\left(\log^{-2}|\bar a|\right)^{\alpha}\right).
\end{aligned}
\]
Here the constant \(C>0\) depends on \(d,t,q\), and we have used \eqref{l3-2}, \eqref{l3-3} and the definition of \(\delta_0\). Therefore we can deduce \eqref{7-1-2} in a similar way to Lemma \ref{lem6-1}. The estimate \eqref{7-1-1} is a special case of \eqref{7-1-2}.
\end{proof}

\begin{lemma}
\label{lem7-2}
Let \(p,q\ge2\), \(\varepsilon_1\in (0,\frac12)\) and \(\omega(\rho)=\widetilde{\omega}\left(\frac{r_0}{3}\rho\right)\) from Lemma \ref{lem7-1}.
\begin{enumerate}
\item If we set $\delta=\varepsilon_1\left(\log^{2\chi}|\bar a|\right)^{-1}|\bar a|$, then we can find \(L_0>0\) fulfilling
\begin{equation}
\label{7-2-1}
  I'_2\geq -C\varepsilon_1^{p(1-s)}L^{p-1}|\bar a|^{p(1-s)-1}
  \left(\log^2|\bar a|\right)^{-\xi}
\end{equation}
for any \(L\geq L_0\);

\item If we assume $(A_1), (A_2), (A'_4)$ hold and let $\bar\delta=\varepsilon_1\left(\log^{2\bar\chi}|\bar a|\right)^{-1}|\bar a|$, we can find \(L_0>0\) satisfying
\begin{equation}
\label{7-2-2}
   J'_2\geq-C\varepsilon_1^{q(1-t)} L^{q-1}|\bar a|^{q(1-t)-1} \left(\log^2|\bar a|\right)^{-\bar \xi}\left( a(0,0)+[a]_\alpha
    \left(|\bar a|\log^{-2}|\bar a|\right)^\alpha
  \right)
\end{equation}
for any \(L\geq L_0\).
\end{enumerate}
Here the $C>0$ is a universal constant depending on $d$.
\end{lemma}

\begin{proof}
Following the proof of Lemma \ref{lem6-2}, we could establish \eqref{7-2-2}. Indeed, to justify \eqref{7-2-2}, we just utilize the facts that
$$
\omega'(|\abar|)\approx1, \quad -\log^{-2}|\abar|<-\log^{-2\bar\chi}|\abar|
$$
and substitute $\eps_1$ in Lemma \ref{lem6-2} with $\varepsilon_1\left(\log^{2\bar\chi}|\bar a|\right)^{-1}\in(0,\frac12)$ here. The assertion \eqref{7-2-1} is a special case of \eqref{7-2-2}.
\end{proof}

The proof of the forthcoming lemma is the same as that of Lemma \ref{lem6-3}, so the details are dropped here and directly give the conclusions.

\begin{lemma}
\label{lem7-3}
Let \(p,q \ge2\) and \(u\in C^{0,\sigma}(\overline{B}_{\rho_2})\) with some \(\sigma\in[0,1]\). Let $\delta,\bar\delta$ be defined as above. Then for any \(\theta\in(0,1)\), there is an \(L_0>0\) such that
\begin{equation*}
  I'_3\geq-C\biggl[\int_{\delta}^{|\bar a|^\theta}\left(r^{\sigma(p-2)+1-sp}+|\bar a|^{\frac{m-1}{m}\sigma}r^{\sigma(p-2)-sp}\right)\,dr
    +|\bar a|^\sigma\int_{|\bar a|^\theta}^{\rho_0}r^{\sigma(p-2)-sp-1}\,dr\biggr]
\end{equation*}
and if the assumptions $(A_2), (A'_4)$ are in force, then
\begin{align*}
  J'_3 &\ge -C a(0,0)\biggl[\int_{\bar\delta}^{|\bar a|^\theta}\left(r^{\sigma(q-2)+1-tq}+|\bar a|^{\frac{m-1}{m}\sigma}
      r^{\sigma(q-2)-tq}\right)\,dr+|\bar a|^\sigma\int_{|\bar a|^\theta}^{\rho_0}r^{\sigma(q-2)-tq-1}\,dr\biggr]                                                    \\
  & -C[a]_\alpha \biggl[\int_{\bar\delta}^{|\bar a|^\theta}\left(r^{\sigma(q-2)+1+\alpha-tq} +|\bar a|^{\frac{m-1}{m}\sigma}
      r^{\sigma(q-2)+\alpha-tq}\right)\,dr+|\bar a|^\sigma\int_{|\bar a|^\theta}^{\rho_0}r^{\sigma(q-2)+\alpha-tq-1}\,dr\biggr]
\end{align*}
for all \(L\geq L_0\). Here \(C>0\) depends on \(d,s,t,p,q,m,m_1\) and \([u]_{C^{0,\sigma}(B_{\rho_2})}\).
\end{lemma}

Now combing Lemmas \ref{lem7-1}--\ref{lem7-3} and Lemma \ref{lem6-4} with the viscosity inequality \eqref{ve2}, we arrive at
\begin{align}
\label{s3-1}
0 \geq & \left(C_{\delta_1}- C\eps_{1}^{p(1-s)}\right)L^{p-1}|\abar|^{p(1-s)-1}
    \bigl(\log^{2}|\abar|\bigr)^{-\xi}- C\max\{|\abar|^{\alpha},|\abar|^{\sigma}\}  \notag\\
&-C\left[\int_{\delta}^{|\abar|^{\theta}}\left(r^{\sigma(p-2)+1-sp}+|\abar|^{\frac{m-1}{m}\sigma} r^{\sigma(p-2)-sp}\right)\,dr
 + |\abar|^{\sigma}\int_{|\abar|^{\theta}}^{\rho_0}r^{\sigma(p-2)-sp-1}\,dr\right] \notag\\
&+ a(0,0)\Bigg[\left(C_{\delta_1}- C\eps_{1}^{q(1-t)}\right)L^{q-1}|\abar|^{q(1-t)-1}
   \bigl(\log^{2}|\abar|\bigr)^{-\bar \xi}- C|\abar|^{\sigma}\int_{|\abar|^{\theta}}^{\rho_0} r^{\sigma(q-2)-tq-1}\,dr  \notag\\
&\qquad- C\int_{\bar\delta}^{|\abar|^{\theta}}\left(r^{\sigma(q-2)+1-tq}+|\abar|^{\frac{m-1}{m}\sigma}r^{\sigma(q-2)-tq}\right)\,dr\Bigg] \notag\\
&- \left(C_{\delta_1}+C\eps_{1}^{q(1-t)}\right)L^{q-1}|\abar|^{q(1-t)+\alpha-1}\bigl(\log^{-2}|\abar|\bigr)^{\bar \xi+\alpha}  - C|\abar|^{\sigma}\int_{|\abar|^{\theta}}^{\rho_0}r^{\sigma(q-2)+\alpha-tq-1}\,dr \notag\\
&- C\int_{\bar\delta}^{|\abar|^{\theta}}\left(r^{\sigma(q-2)+1+\alpha-tq}+
|\abar|^{\frac{m-1}{m}\sigma}r^{\sigma(q-2)+\alpha-tq}\right)\,dr,
\end{align}
where $\theta,\eps_1\in(0,1),\sigma\in(0,1]$ and $\delta,\bar\delta$ are given in Lemma \ref{lem7-2}. Next, we simplify this inequality.

\medskip

(i) Since $\gamma_0=1$, we get $p-1-sp\leq 0$, and further $p(1-s)-1 < \min\{\alpha,\sigma\}$. Hence, there holds
\[
  C\max\{|\abar|^{\alpha},|\abar|^{\sigma}\}
  \leq \frac{C_{\delta_1}}{20}
  L^{p-1}|\abar|^{p(1-s)-1}
  \bigl(\log^{-2}|\abar|\bigr)^{\xi},
\]
provided we take $L$ large enough ( $|\abar|\to0$ as $L\to\infty$).

(ii) By selecting $\eps_{1}>0$ sufficiently small, we could have
$$
C_{\delta_1}- C\eps_{1}^{p(1-s)}\ge\frac{19C_{\delta_1}}{20},\ \ C_{\delta_1}- C\eps_{1}^{q(1-t)}\ge\frac{3C_{\delta_1}}{4},\ \ C_{\delta_1}+C\eps_{1}^{q(1-t)}\le2C_{\delta_1}.
$$

(iii) Via \eqref{dis}, $p\le q$ and $r\le1$, there holds $r^{\sigma(q-p)+\alpha+sp-tq}\le1$. Thus,
\begin{align*}
|\abar|^{\sigma}\int_{|\abar|^{\theta}}^{\rho_0}r^{\sigma(q-2)+\alpha-tq-1}\,dr\le |\abar|^{\sigma}\int_{|\abar|^{\theta}}^{\rho_0}
    r^{\sigma(p-2)-sp-1}\,dr\le \frac{|\abar|^{\sigma+\theta(\sigma(p-2)-sp)}}{sp-(p-2)},
\end{align*}
where we note $0<\sigma\le1<\frac{sp}{p-2}$. Furthermore, it holds that
$$
|\abar|^{\frac{m-1}{m}\sigma}\int_{\bar\delta}^{|\abar|^{\theta}}r^{\sigma(q-2)+\alpha-tq}\,dr\le |\abar|^{\frac{m-1}{m}\sigma}\int_{\bar\delta}^{|\abar|^{\theta}}r^{\sigma(p-2)-sp}\,dr.
$$

(iv) Let us invoke the distance condition \eqref{dis} to justify
\begin{equation}
\label{s3-2}
C_{\delta_1}L^{q-1}|\abar|^{q(1-t)+\alpha-1}\bigl(\log^{-2}|\abar|\bigr)^{\bar\xi+\alpha}\leq
\frac{C_{\delta_1}}{20}L^{p-1}|\abar|^{p(1-s)-1}\bigl(\log^{-2}|\abar|\bigr)^{\xi},
\end{equation}
or equivalently,
\begin{align*}
C_{\delta_1}(L|\abar|)^{q-p}
&\leq\frac{C_{\delta_1}}{20}|\abar|^{tq-sp-\alpha}\bigl(\log^{-2}|\abar|\bigr)^{tq-sp-\alpha+p-q}   \\
&=\frac{C_{\delta_1}}{20}\bigl(|\abar|\log^{-2}|\abar|\bigr)^{tq-sp-\alpha}\bigl(\log^{-2}|\abar|\bigr)^{p-q}.
\end{align*}
Here we note that $\xi-\bar\xi=\frac{d+1}{2}+p-sp-\left(\frac{d+1}{2}+q-tq\right)=tq-sp+p-q$. Because we know that $|\abar|$ and $\log^{-2}|\abar|$ tend to $0$ as $L\to\infty$, by \eqref{dis} we have
\begin{equation}
\label{s3-3}
\bigl(|\abar|\log^{-2}|\abar|\bigr)^{tq-sp-\alpha}\bigl(\log^{-2}|\abar|\bigr)^{p-q}\rightarrow \infty,
\end{equation}
when $p<q$ and $tq\le sp+\alpha$, or $p=q$ and $tq<sp+\alpha$. Let us consider $(L|\abar|)^{q-p}$. Due to \eqref{M}, we find $L\omega(|\abar|)\leq 2\|u\|_{L^{\infty}(B_{2})}$.
Therefore,
\[
L|\abar|\left(1+\frac{1}{\log|\abar|}\right)
\leq 2\|u\|_{L^{\infty}(B_{2})},
\]
and then
\begin{equation}
\label{s3-4}
(L|\abar|)^{q-p}
\leq
\bigl(4\|u\|_{L^{\infty}(B_{2})}\bigr)^{q-p},
\qquad
\text{for }|\abar|<\frac1{100}.
\end{equation}
Putting together \eqref{s3-3}, \eqref{s3-4} leads to \eqref{s3-2}, if $L$ is sufficiently large.

Now according to (i)--(iv), \eqref{s3-1} is reduced to
\begin{align}
\label{s3-5}
0 \geq &\ C_{\delta_1}L^{p-1}|\abar|^{p(1-s)-1}
    \bigl(\log^{2}|\abar|\bigr)^{-\xi}- C|\abar|^{\sigma+\theta(\sigma(p-2)-sp)}  \notag\\
&-C\int_{\delta}^{|\abar|^{\theta}}\left(r^{\sigma(p-2)+1-sp}+|\abar|^{\frac{m-1}{m}\sigma} r^{\sigma(p-2)-sp}\right)\,dr
  \notag\\
&- C\int_{\bar\delta}^{|\abar|^{\theta}}\left(r^{\sigma(p-2)+1-sp}+
|\abar|^{\frac{m-1}{m}\sigma}r^{\sigma(p-2)-sp}\right)\,dr \notag\\
&+ a(0,0)\Bigg[C_{\delta_1}L^{q-1}|\abar|^{q(1-t)-1}
   \bigl(\log^{2}|\abar|\bigr)^{-\bar \xi}-  C|\abar|^{\sigma+\theta(\sigma(q-2)-tq)}  \notag\\
&\qquad- C\int_{\bar\delta}^{|\abar|^{\theta}}\left(r^{\sigma(q-2)+1-tq}+|\abar|^{\frac{m-1}{m}\sigma}r^{\sigma(q-2)-tq}\right)\,dr\Bigg].
\end{align}
We start from this inequality to demonstrate the Lipschitz continuity of solutions in the scenario $\gamma_0=1$.

\medskip

\noindent\textbf{\emph{Proof of Proposition \ref{pro7-2} for $\gamma_0=1$}.} Applying Proposition \ref{pro7-1}, we can assume $C^{0,\sigma_3}_{\rm loc}(B_2)$ with $\sigma_3\in(0,1)$ to be determined. Let $\sigma:=\sigma_3$ in \eqref{s3-5}, and then we analyze this inequality. First, we select $\sigma_3<1$ so large that
$$
\sigma_3(p-2)+2-sp=(\sigma_3-s)p+2-2\sigma_3>0
$$
and
$$
\sigma_3(q-2)+2-tq=(\sigma_3-t)q+2-2\sigma_3>0.
$$
Such choice implies that
\begin{align}
\int_{\delta}^{|\abar|^{\theta}}r^{\sigma_{3}(p-2)+1-sp}\,d r&\leq C|\abar|^{\theta(\sigma_{3}(p-2)+2-sp)}=:C|\abar|^{\theta_{1}},
 \label{s3-6}
\\
\int_{\bar\delta}^{|\abar|^{\theta}}r^{\sigma_{3}(p-2)+1-sp}\,d r &\leq
C|\abar|^{\theta(\sigma_{3}(p-2)+2-sp)}=:C|\abar|^{\theta_{1}}
\label{s3-7}
\end{align}
and
\begin{equation}
\label{s3-8}
  \int_{\bar\delta}^{|\abar|^{\theta}}r^{\sigma_{3}(q-2)+1-tq}\,d r\leq C|\abar|^{\theta(\sigma_{3}(q-2)+2-tq)}=:C|\abar|^{\theta_{2}}.
\end{equation}

We then choose $\theta\in(0,1)$ so small that
\[
\sigma_{3}+\theta\bigl(\sigma_{3}(p-2)-sp\bigr)>0,
\qquad
\sigma_{3}+\theta\bigl(\sigma_{3}(q-2)-tq\bigr)>0,
\]
and denote
\begin{equation}
\label{s3-9}
|\abar|^{\sigma_{3}+\theta(\sigma_{3}(p-2)-sp)}=:|\abar|^{\theta_3},
\quad
|\abar|^{\sigma_{3}+\theta(\sigma_{3}(q-2)-tq)}=:|\abar|^{\theta_4}.
\end{equation}
Next, observe that by $1\leq \min\left\{\frac{sp}{p-1},\,\frac{tq}{q-1}\right\}$,
\[
\sigma_{3}(p-2)-sp\leq\sigma_{3}(p-2)-(p-1)=(\sigma_{3}-1)(p-2)-1<-1,
\]
and
\[
\sigma_{3}(q-2)-tq\leq(\sigma_{3}-1)(q-2)-1<-1.
\]
Therefore, we could conclude
\begin{align*}
\int_{\delta}^{|\abar|^{\theta}}r^{\sigma_{3}(p-2)-sp}\,d r
&\leq\frac{1}{sp-\sigma_{3}(p-2)-1}\delta^{\sigma_{3}(p-2)-sp+1}\\
&\leq C_{\varepsilon_{1}}|\abar|^{\sigma_{3}(p-2)-sp+1}\left(\log^{2\chi}|\abar|\right)^{sp-\sigma_{3}(p-2)-1},
\end{align*}
\begin{align*}
\int_{\bar\delta}^{|\abar|^{\theta}}r^{\sigma_{3}(p-2)-sp}\,d r
&\leq\frac{1}{sp-\sigma_{3}(p-2)-1}{\bar\delta}^{\sigma_{3}(p-2)-sp+1}\\
&\leq C_{\varepsilon_{1}}|\abar|^{\sigma_{3}(p-2)-sp+1}\left(\log^{2\bar\chi}|\abar|\right)^{sp-\sigma_{3}(p-2)-1}
\end{align*}
and
\[
\int_{\bar\delta}^{|\abar|^{\theta}}r^{\sigma_{3}(q-2)-tq}\,d r\leq
C_{\varepsilon_{1}}|\abar|^{\sigma_{3}(q-2)-tq+1}
\left(\log^{2\bar\chi}|\abar|\right)^{tq-\sigma_{3}(q-2)-1}.
\]
Here we keep in mind that $sp-\sigma_{3}(p-2)-1>0,tq-\sigma_{3}(q-2)-1>0$. Now let us pick $m\ge3$ so large that
\[
\theta_5:=\frac{m-1}{m}\sigma_{3}+\sigma_{3}(p-2)-sp+1>0,
\quad
\theta_6:=\frac{m-1}{m}\sigma_{3}+\sigma_{3}(q-2)-tq+1>0,
\]
provided we take in advance $\sigma_{3}<1$ large enough to meet
\[
\sigma_{3}(p-1)-sp+1>0,
\quad
\sigma_{3}(q-1)-tq+1>0.
\]
As a result,
\begin{align}
|\abar|^{\frac{m-1}{m}\sigma_{3}}\int_{\delta}^{|\abar|^{\theta}}r^{\sigma_{3}(p-2)-sp}\,d r&\leq C_{\eps_1}|\abar|^{\theta_5}
\left(\log^{2\chi}|\abar|\right)^{sp-\sigma_{3}(p-2)-1}, \label{s3-10}
\\
|\abar|^{\frac{m-1}{m}\sigma_{3}}\int_{\bar\delta}^{|\abar|^{\theta}}r^{\sigma_{3}(p-2)-sp}\,d r&\leq C_{\eps_1}|\abar|^{\theta_5}
\left(\log^{2\bar\chi}|\abar|\right)^{sp-\sigma_{3}(p-2)-1}, \label{s3-11}
\\
|\abar|^{\frac{m-1}{m}\sigma_{3}}\int_{\bar\delta}^{|\abar|^{\theta}}r^{\sigma_{3}(q-2)-tq}\,d r&\leq
C_{\eps_1}|\abar|^{\theta_6}\left(\log^{2\bar\chi}|\abar|\right)^{tq-\sigma_{3}(q-2)-1}. \label{s3-12}
\end{align}

Finally, collecting the inequalities \eqref{s3-5}--\eqref{s3-12} arrives at
\begin{align*}
0&\ge   C_{\delta_1}L^{p-1}|\abar|^{p(1-s)-1}\log^{-2\xi}|\abar|- C\Big[|\abar|^{\theta_1}+|\abar|^{\theta_3}+|\abar|^{\theta_5}\left(\log^{2\chi}|\abar|\right)^{sp-\sigma_{3}(p-2)-1}
\\
&\quad+|\abar|^{\theta_5}\left(\log^{2\bar\chi}|\abar|\right)^{sp-\sigma_{3}(p-2)-1}\Big]+a(0,0)\Bigg[ C_{\delta_1}L^{q-1}|\abar|^{q(1-t)-1}\log^{-2\bar\xi}|\abar|\\
&\quad-C\Big(|\abar|^{\theta_2}+|\abar|^{\theta_4}
+|\abar|^{\theta_6}\left(\log^{2\bar\chi}|\abar|\right)^{tq-\sigma_{3}(q-2)-1}\Big)\Bigg],
\end{align*}
where $\theta_i>0$ for $i=1,\cdots,6$. We notice a simple fact that $|\abar|^l\log^{2\kappa}|\abar|\to 0$ as $|\abar|\to0$ for $l,2\kappa>0$. Since $p(1-s)-1\le0$ and $q(1-t)-1\le0$, we may take $L$ large enough so that the right-hand side of the display above is positive. Hence there is a contradiction. This implies that $u\in C^{0,1}_{\rm loc}(B_2)$ as argued at the end of Step 1. Up to now, we have completed the proof of Proposition \ref{pro7-2}.    \hfill$\square$


\medskip

We conclude this portion by an improved H\"older regularity of solutions. As proved above, it is well known that $u\in C^{0,\min\left\{1,\frac{sp}{p-1},\frac{tq}{q-1}\right\}}_{\rm loc}(B_2)$. So if $\frac{sp}{p-1}$ and $\frac{tq}{q-1}$ are less than 1, then $u\in C^{0,\min\left\{\frac{sp}{p-1},\frac{tq}{q-1}\right\}}_{\rm loc}(B_2)$. As a matter of fact, we may further demonstrate $u\in C^{0,\frac{sp}{p-1}}_{\rm loc}(B_2)$ when $1>\frac{sp}{p-1}\ge\frac{tq}{q-1}$, which is an interesting outgrowth.

\begin{proposition}
\label{pro4-3}
Suppose that the conditions on $a(\cdot)$, $(A_1)-(A_3),(A'_4)$, and the distance requirement \eqref{dis} with $0<s\le t<1$ hold true. For any viscosity solution $u$ to \eqref{main}, if $1>\frac{sp}{p-1}\ge\frac{tq}{q-1}$, then there holds 
$$
u\in C^{0,\gamma}_{\rm loc}(B_2) \ \ \ \text{for any } \ \gamma\le\frac{sp}{p-1}.
$$
\end{proposition}

\begin{proof}
Here we observe that
 $$
 \frac{tq}{q-1} \le \frac{sp}{p-1} \quad \text{implies} \quad \frac{tq}{q-2} \le \frac{sp}{p-2}
 $$
via $p\le q$. In this proof, we set $2<p\le q$, since for $p=q=2$ or $p=2$ the processes blow will be simplified. We first justify \(u\in C^{0,\gamma}(B_{\rho_1})\) for $\gamma\in \left(\gamma_0,\frac{sp}{p-1}\right)$ by using the bootstrap argument with $\gamma_0=\frac{tq}{q-1}$, and then get $u\in C^{0,\frac{sp}{p-1}}(B_{\rho_1})$ as the proof of Proposition \ref{pro7-2} for $\gamma_0<1$.

Assume $u\in C^{0,\sigma}(B_{\rho_2})$ for some $\sigma\in \left[\frac{tq}{q-1},\gamma\right)$. In fact, $u\in C^{0,\frac{tq}{q-1}}(B_{\rho_2})$ is as our starting point. 
Noting that \(\sigma<sp/(p-2)\), 
it follows from \eqref{so} and $0\le a(0,0)\le \|a\|_\infty$ that
\begin{align}
\label{7-5-1}
0 \ge {}&
 C_{\delta_0}L^{p-1}|\abar|^{\beta(p-1)-sp}- C\max\{|\abar|^\alpha,|\abar|^\sigma\}- C|\abar|^\sigma\int_{|\abar|^\theta}^{\rho_0} r^{\sigma(p-2)-sp-1}\,dr   \notag\\
&- C\int_{\delta}^{|\abar|^\theta}
 \left(r^{\sigma(p-2)+1-sp} +|\abar|^{\frac{m-1}{m}\sigma} r^{\sigma(p-2)-sp}\right)\,dr \notag\\
\ge{}& C_{\delta_0}L^{p-1}|\abar|^{\beta(p-1)-sp}- C\max\{|\abar|^\alpha,|\abar|^\sigma\}- C|\abar|^{\sigma+\theta(\sigma(p-2)-sp)}
 \notag\\
&- C\int_{\delta}^{|\abar|^\theta}
 \left(r^{\sigma(p-2)+1-sp} +|\abar|^{\frac{m-1}{m}\sigma} r^{\sigma(p-2)-sp}\right)\,dr
\end{align}
with $\delta=\varepsilon_1|\abar|$. Here we have utilized \(C a(0,0)\delta_0^{q_t}L^{q-1}|\abar|^{\beta(q-1)-tq}\ge 0\) and
$$
\sigma(q-2)-tq\ge\sigma(p-2)-sp.
$$
It is easy to find that
\[
    |\abar|^{\sigma+\theta(\sigma(p-2)-sp)}\rightarrow 0
\]
as $|\abar|\to 0$, provided \(\theta\in(0,1)\) is small enough.
Now set $\beta:=\sigma_4$ ($\omega_\beta(\rho):=\omega_{\sigma_4}(\rho)$) in \eqref{7-5-1} with
\[
    \sigma_4=\min\left\{\gamma,\sigma+\frac{1}{2(p-1)}\right\}.
\]
We next show that the right-hand side of \eqref{7-5-1} is greater than 0. Via the choice of $\sigma_4$,
we get
\[
    \sigma_4(p-1)-\sigma<\sigma(p-2)+1
    \quad\text{and even}\quad
    \sigma_4(p-1)-1<\sigma(p-2)+1,
\]
and infer
\[
    \int_{\delta}^{|\abar|^\theta}r^{\sigma(p-2)+1-sp}\,dr < \int_{\delta}^{|\abar|^\theta}r^{\sigma_4(p-1)-sp-1}\,dr
    \le C_{\eps_1}|\abar|^{\sigma_4(p-1)-sp}.
\]
We can select \(m\ge3\) so large that
\[
    \sigma_4(p-1)-\frac{m}{m-1}\sigma\le \sigma(p-2)+1,
\]
and then have
\begin{align*}
|\abar|^{\frac{m-1}{m}\sigma}\int_{\delta}^{|\abar|^\theta} r^{\sigma(p-2)-sp}\,dr
&\le
|\abar|^{\frac{m-1}{m}\sigma}\int_{\delta}^{1}r^{\sigma_4(p-1)-\frac{m}{m-1}\sigma-sp-1}\,dr  \\
&\le
C_{\eps_1}|\abar|^{\sigma_4(p-1)-sp}.
\end{align*}
Finally, we notice $\sigma_4(p-1)-sp<0$. Thereby, picking \(\eps_1\) small enough and \(L\) large enough, the positive term $C_{\delta_0}L^{p-1}|\abar|^{\sigma_4(p-1)-sp}$ could absorb the other terms in \eqref{7-5-1}, which indicates
\[
    0\ge CL^{p-1}|\abar|^{\sigma_4(p-1)-sp}>0.
\]
A contradiction. That is, we show $u\in C^{0,\sigma_4}(B_{\rho_1})$. Ultimately, applying the bootstrap argument as in Step 1 leads to
$u\in C^{0,\gamma}_{\rm loc}(B_2)$ for any $\gamma\in\left(\gamma_0,\frac{sp}{p-1}\right)$.

Repeating the procedures above and choosing sufficiently large $\sigma<\frac{sp}{p-1}$ as in Proof of Proposition \ref{pro7-2} for $\gamma_0<1$, we could conclude that $u$ is of the class $C^{0,\frac{sp}{p-1}}(B_{\rho_1})$.
\end{proof}

\medskip

\subsection*{Acknowledgements}
This work was supported by the National Natural Science Foundation of China (Nos. 12471128 and 12301245), and Natural Science Foundation of Heilongjiang Province (No. YQ2025A002).

\subsection*{Conflict of Interest} The authors declare that there is no conflict of interest. We also declare that this
manuscript has no associated data.

\subsection*{Data Availability} Data sharing is not applicable to this article as no datasets were generated or analysed
during the current study.


\end{document}